%% file: main.tex
\documentclass[11pt]{article}
\usepackage[margin=1in]{geometry}

\usepackage[table]{xcolor}
\definecolor{lightgray}{gray}{0.9}

\usepackage{tikz}
\definecolor{lime}{HTML}{A6CE39}
\DeclareRobustCommand{\orcidicon}{
	\begin{tikzpicture}
	\draw[lime, fill=lime] (0,0) 
	circle [radius=0.16] 
	node[white] {{\fontfamily{qag}\selectfont \tiny ID}};
	\draw[white, fill=white] (-0.0625,0.095) 
	circle [radius=0.007];
	\end{tikzpicture}
	\hspace{-2mm}
}
\foreach \x in {A, ..., Z}{\expandafter\xdef\csname orcid\x\endcsname{\noexpand\href{https://orcid.org/\csname orcidauthor\x\endcsname}
			{\noexpand\orcidicon}}
}
\usepackage{amssymb, amsmath, amscd, amstext}
\usepackage[english]{babel}
\usepackage{graphicx}
\usepackage{stmaryrd}
\usepackage{color}
\usepackage{subfig, color}
\usepackage{float}
\usepackage{url}
\usepackage{soul}
\usepackage{todonotes}
\usepackage{booktabs}
\usepackage{hyperref}
\usepackage{tikz-cd}
\usepackage{rotating}
\usepackage{authblk}

\usepackage{amsthm}

\numberwithin{equation}{section}

\newtheorem{theorem}{Theorem}

\usepackage[square,sort,comma,numbers]{natbib}

\begin{document}

\title{Convex Mixed-Integer Nonlinear Programs Derived from Generalized Disjunctive Programming using Cones}

% \author{ David E. Bernal Neira\\
% 	Department of Chemical Engineering\\
% 	Carnegie Mellon University\\
% 	Pittsburgh, PA 15213 \\
% 	Research Institute of Advanced Computer Science\\
% 	Universities Space Research Association\\
% 	Washington, DC 20024 \\
% 	Quantum Artificial Intelligence Laboratory\\
% 	NASA Ames Research Center\\
% 	Moffett Field, CA 94035 \\
% 	\texttt{bernalde@cmu.edu} \\
% 	%% examples of more authors
% 	\and
% 	Ignacio E.~Grossmann \\
% 	Department of Chemical Engineering\\
% 	Carnegie Mellon University\\
% 	Pittsburgh, PA 15213 \\
% 	\texttt{grossmann@cmu.edu}
% 	}

\author[1,2,3]{David E. Bernal Neira\orcidA{}}

\author[4]{Ignacio E. Grossmann\orcidB{}}

\affil[1]{Davidson School of Chemical Engineering, Purdue University, IN, 47907, USA}

\affil[2]{Research Institute of Advanced Computer Science, Universities Space Research Association, Washington, DC, 20024, USA}

\affil[3]{Quantum Artificial Intelligence Laboratory, NASA Ames Research Center, Moffett Field, CA, 94035, USA}

\affil[4]{Chemical Engineering Department, Carnegie Mellon University, Pittsburgh, PA, 15213, USA}

\date{}
	
\maketitle

\begin{abstract}
We propose the formulation of convex Generalized Disjunctive Programming (GDP) problems using conic inequalities leading to conic GDP problems.
We then show the reformulation of conic GDPs into Mixed-Integer Conic Programming (MICP) problems through both the big-M and hull reformulations.
These reformulations have the advantage that they are representable using the same cones as the original conic GDP.
In the case of the hull reformulation, they require no approximation of the perspective function.
Moreover, the MICP problems derived can be solved by specialized conic solvers and offer a natural extended formulation amenable to both conic and gradient-based solvers.
We present the closed form of several convex functions and their respective perspectives in conic sets, allowing users to formulate their conic GDP problems easily.
We finally implement a large set of conic GDP examples and solve them via the scalar nonlinear and conic mixed-integer reformulations.
These examples include applications from Process Systems Engineering, Machine learning, and randomly generated instances.
Our results show that the conic structure can be exploited to solve these challenging MICP problems more efficiently.
Our main contribution is providing the reformulations, examples, and computational results that support the claim that taking advantage of conic formulations of convex GDP instead of their nonlinear algebraic descriptions can lead to a more efficient solution to these problems.
\end{abstract}

\section{Introduction} 
\label{sec:intro}

A Mixed-Integer Nonlinear Programming (MINLP) problem involves nonlinear algebraic inequalities describing the constraints and objectives while the variables can take continuous or discrete values.
MINLP is a problem class of great interest, both theoretical~\citep{liberti2019undecidability} and practical~\citep{trespalacios2014review,lee2011mixed}.
In particular, MINLP problems \textit{formulations} allow modeling a wide range of applications.
Most industrial problems can be modeled using MINLP~\citep{Liberti2017}.

A particular class of MINLP problems is where the constraints are convex functions.
Although it is non-convex because of the nature of the discrete variables, this problem is known as convex MINLP~\citep{fletcher1994solving,kronqvist2019review}.
This class of MINLP is a subject of interest given the many applications that it can represent.% and the challenging algorithmic requirements that need to be tackled to address their problems.
For a review on convex MINLP, refer to Kronqvist et al.~\citep{kronqvist2019review}.

% In this manuscript, we are concerned with convex Mixed-Integer Nonlinear Programming (MINLP) problems.
A convex MINLP problem is defined as
\begin{equation}
	\label{prob:MINLP}
	\tag{MINLP}
	\begin{aligned}
  &\min_{\mathbf{x,y}} &&f(\mathbf{x,y})\\
	& \textnormal{s.t. } &&\mathbf{g}(\mathbf{x,y}) \leq \mathbf{0},\\
	& && \mathbf{y}^l \leq \mathbf{y} \leq \mathbf{y}^u,\\
	& &&\mathbf{x}\in \mathbb{R}_+^{n_x},\ 
	\mathbf{y} \in \mathbb{Z}^{n_y},
	\end{aligned}
\end{equation}
where the objective function $f: \mathbb{R}^{n_x + n_y} \to \mathbb{R} \cup \{ \infty \}$ is convex and the constraints $\mathbf{g}:\mathbb{R}^{n_x + n_y} \to (\mathbb{R} \cup \{ \infty \} )^{J}$ define a convex set $\mathcal{F} = \{ \mathbf{x} \in \mathbb{R}_+^{n_x}, \mathbf{y} \in \mathbb{R}^{n_y} \mid \mathbf{g}(\mathbf{x},\mathbf{y}) \leq \mathbf{0} \}$.
Although it is not necessary, we will consider that each constraint, $g_j(\mathbf{x},\mathbf{y})$ for $j \in \{1, \dots, J\} = \llbracket J \rrbracket$, is a convex function.
We consider bounded integer variables $\mathbf{y}$.
Without loss of generality, we will assume that the objective function is linear, which can be achieved through the epigraph reformulation~\citep{kronqvist2019review}.
Notice that, although the continuous relaxation of the feasible region $F$ is convex, the original convex MINLP feasible region is non-convex given the discrete nature of variables $\mathbf{y}$.

Among the solution techniques for convex MINLP, several have been adapted from the Mixed-Integer Linear Programming (MILP), including Branch \& Bound~\citep{dakin1965tree} and Benders Decomposition~\citep{geoffrion1972generalized}.
In contrast, others generalize the solutions methods for convex continuous Nonlinear Programming (NLP) problems, such as the Extended Cutting Plane methods~\citep{westerlund1998extended}.
A particularly successful approach to convex MINLP is the outer-approximation (OA) method proposed by Duran and Grossmann~\citep{duran1986outer}, where an iterative solution of convex NLP and MILP subproblems is performed.
The MILP is derived through first-order Taylor approximations, or gradient-based linearizations, of the nonlinear constraints at the NLP solutions. The NLPs stem from the problems appearing when fixing the values of the discrete variables at the MILP solution~\citep{duran1986outer,fletcher1994solving}.
Many current commercial tools to solve convex MINLP rely on the OA method~\citep{kronqvist2019review}.

In continuous convex programming, solution methods have also been derived by generalizing Linear Programming (LP) notions and techniques.
One of the most successful ones has been the proposal of convex optimization problems as problems defined over cones, or Conic Programming (CP) problems~\citep{ben2001lectures}.
CP is a numerically stable alternative for convex programming~\citep{ben2001lectures},  given that it exploits properties of the conic sets.
Convex programming problems described via algebraic convex nonlinear constraints of the form $f(x) \leq 0$ can be equivalently posed as a linear transformation of the variables belonging to proper cones $\mathcal{K}$, i.e., $\mathbf{A}\mathbf{x} - \mathbf{b} \in \mathcal{K}$~\citep{ben2001lectures,kilincc2016minimal}.
A generalization of CP, where some variables are required to take discrete values, is Mixed-Integer Conic Programming (MICP).
MICP problems are highly expressible and can represent a wide range of optimization problem~\citep{lubin2017mixed}.
Many of these applications have been gathered in the problem library CBLib~\citep{friberg2016cblib}.

The automatic identification and translation of the two equivalent descriptions of convex sets is a crucial feature for developing algorithmic solution software, \textit{solvers}.
This is because the description of problems using algebraic constraints might be more natural for certain practitioners, e.g., in fields such as chemical engineering, where material and energy balances are naturally representable via scalar algebraic constraints.
However, the conic description of the problem allows taking advantage of mathematical properties such as conic duality for more stable solution procedures.
Generic solvers have been designed to tackle CP problems, e.g., MOSEK~\citep{aps2018mosek}, ECOS~\citep{domahidi2013ecos}, and Hypatia~\citep{coey2020towards}.
This translation is not trivial~\citep{vanderbei1998using,Zverovich2015,erickson2019detection}.
However, it has been achieved for the quadratic case, allowing solution methods based on CP to be used for these problems.
An alternative to translating practical optimization problems into CP is via Disciplined Convex Programming (DCP)~\citep{grant2006disciplined}, where strict rules of function definitions guarantee the problem's convexity and perform the translation such that they can be solved through generic conic solvers.

In the mixed-integer setting, solvers have been designed to take as input the MICP problem, taking advantage of this form of the optimization problem structure, e.g., Mosek~\citep{aps2018mosek}, and Pajarito~\citep{lubin2016extended,lubin2018polyhedral,coey2020outer}.
Even for solvers that do not necessarily consider the conic representation of convex problems, identifying such structures leads to improvements in its performance, such as in SCIP~\citep{vigerske2013decomposition,bestuzheva2021computational} and BARON~\citep{khajavirad2018hybrid}.
There is a significant potential for MINLP solvers to perform automatic reformulations once they identify correct structures~\citep{gunluk2010perspective}.
An example of the automatic identification of conic structures is Mixed-Integer Quadratically-constrained Quadratic Programming (MIQCQP) problems can now be tackled through Mixed-Integer Second-Order Conic Programming (MISOCP) methods in commercial solvers such as Knitro~\citep{waltz2017knitro}, Xpress~\citep{belotti2016algorithms}, Gurobi~\citep{gurobi}, and CPLEX~\citep{cplex}.

The discrete nature of the integer variables in mixed-integer programming problems has been exploited to derive efficient solution methods for these problems.
In particular, deriving sets of extra inequalities, \textit{cutting planes} or \textit{cuts}, has allowed a considerable speedup in the solution of these problems, see~\citep{conforti2014integer}.
One of the key disciplines for deriving such cutting planes is Disjunctive Programming~\citep{balas2018disjunctive}, which considers the optimization over disjunctive sets such as the one given by the domain of the discrete variables.
In the convex nonlinear setting, the conic structure has been exploited to derive special cutting planes for MICP solution methods ~\citep{ccezik2005cuts,belotti2015conic,lodi2019disjunctive}.
A source of these problems are those driven by \textit{indicator variables}, that activate or deactivate sets of constraints~\citep{gunluk2010perspective}, see a review by Bonami et al.~\citep{bonami2015mathematical}.

Generalized Disjunctive Programming (GDP) was proposed by Grossmann and Lee~\citep{grossmann2003generalized} as an intuitive way of describing the logic behind applications.
In this setting, sets of constraints are activated with logical variables linked to each other by logical constraints, including disjunctions.
This mathematical description of the problem can be tackled directly by logic-based optimization methods~\citep{chenpyomo}, which generalize mixed-integer solution methods to the logical domain.
Another way of solving these problems is through reformulations into mixed-integer programs, where the logical variables are mapped to binary or indicator variables.
An important reference that addresses the reformulation of unions of convex sets is given by Vielma~\citep{vielma2019small}.
Depending on the linearity of the constraints within the GDP, the reformulations can yield a MILP or MINLP problem.
The two most common reformulations are: the big-M reformulations, where a large coefficient is added to make the constraints redundant in the case their associated indicator variable is inactive, and the hull reformulation (HR), where using Disjunctive Programming theory, a set of constraints in an extended space are derived such that their projection onto the space of the original variables is the convex hull of the disjunctive sets.con
These two reformulations yield different mixed-integer models, which can be characterized by size and tightness.
The tightness of a mixed-integer model is measured through the difference of the optimal solution of the problem, ignoring the discrete constraints, known as the \textit{continuous relaxation}, and the original problem's optimal solution~\citep{trespalacios2014review,vielma2019small}.
The big-M and hull reformulations offer a tradeoff between tightness and problem size.
The HR is the tightest possible model, while the big-M formulation requires the least additional continuous variables and constraints, adding a single constraint per constraint in the disjunctions and using the same number of continuous variables as in the original disjunctive model.
Both the model size and tightness are relevant to the efficiency of solution methods of mixed-integer programs~\citep{gunluk2012perspective,vielma2019small}.

Perspectives of nonlinear functions arise in formulations of convex MINLP problems.
We denote these expressions as perspective functions.
When the original functions are nonlinear, these can become challenging for solvers.
In particular, they can be non-differentiable at 0~\citep{gunluk2012perspective,furman2020computationally}.
The HR of a convex disjunctive program requires including perspective functions in the resulting MINLP. 
These perspective functions can be written in terms of the original variables of the disjunctive program\citep{hijazi2012mixed,bonami2015mathematical} or in a higher dimensional space~\citep{lee2000new,grossmann2003generalized}. 
The perspective functions can be included directly in the convex MINLP problem or indirectly by generating valid cutting-planes~\citep{stubbs1999branch,frangioni2006perspective}.
A recent computational study shows the positive impact of perspective cuts in the MINLP framework~\citep{bestuzheva2021computational}. 
The importance of the perspective functions and the challenges associated with their implementation have motivated their study.
$\varepsilon$-approximations exist for general convex functions~\citep{lee2000new,furman2020computationally}, while customized versions have been derived for special cases~\citep{akturk2009strong,gunluk2010perspective,hijazi2012mixed}.

\subsection{Contributions}
In this manuscript, we systematically study different convex GDP problems.
In these problems, the convex sets within the disjunctions are representable using conic constraints.
We denote these problems as conic GDPs.
We provide the reformulations of such conic GDP problems into MICP problems via the big-M and the hull reformulations.
The MICP instances can be addressed with specialized solvers that exploit their conic structure.
Moreover, in the case of the HR, we avoid numerical challenges derived from the perspective reformulation of nonlinear functions that appear when using scalar nonlinear constraints in the disjunctions.
We present a set of examples and perform a computational study that supports our central hypothesis: considering the original conic structure of conic GDP problems can lead to more efficient solutions to such problems.
This is when comparing the MICP reformulations of conic GDPs to the MINLPs derived from the scalar nonlinear descriptions of convex sets within disjunctions in convex GDPs.

\section{Generalized Disjunctive Programming}
\label{sec:background.gdp}

The framework of Generalized Disjunctive Programming (GDP) was introduced by Raman and Grossmann~\citep{raman1994modelling}.
This modeling paradigm extends the usual mathematical programming paradigm by allowing Boolean variables, logical constraints, and disjunctions to appear in the optimization problem formulation.
We define a GDP as follows:
\begin{equation}
\label{prob:gdp}
\begin{aligned}
\min_{\mathbf{x},\mathbf{Y}} &\ f(\mathbf{x}) \\
\textnormal{s.t. }&\ \mathbf{g}(\mathbf{x}) \leq \mathbf{0}\\
&\ \bigvee_{i\in D_k} \left[
    \begin{gathered}
    Y_{ik} \\
    \mathbf{h}_{ik}(\mathbf{x})\leq \mathbf{0}\\
    \end{gathered}
\right], \quad  k \in K\\
&\ \veebar_{i \in D_k} Y_{ik}, \quad  k \in K \\
&\ \Omega(\mathbf{Y}) = True \\
&\ \mathbf{x}^{l} \leq \mathbf{x} \leq \mathbf{x}^{u}\\
&\ \mathbf{x} \in \mathbb{R}^{n}\\
&\ Y_{ik} \in \{False, True\}, \quad  k \in K, i \in D_k,\\
\end{aligned} \tag{GDP}
\end{equation}
where constraints $\mathbf{g}(\mathbf{x}) \leq \mathbf{0}$ are called global constraints, the set $K$ represents the possible disjunctions in the problem, and each element $i$ of the set $D_k$ represents a disjunctive term, also called disjunct, in that disjunction.
In the disjunction $k \in K$, each disjunct $i \in D_k$ has a set of constraints $\mathbf{h}_{ik}(\mathbf{x}) \leq \mathbf{0}$ which are activated when a Boolean variable associated with the disjunct is equal to $True$, i.e., $Y_{ik}=True$.
Each disjunct may contain a different number of constraints $J_{ik}$, i.e., $\mathbf{h}_{ik}(\mathbf{x}) = (h_{ik1}(\mathbf{x}),\dots,h_{ikJ_{ik}}(\mathbf{x})) = (h_{ik\llbracket J_{ik} \rrbracket}(\mathbf{x}))$.
These constraints define the set $\mathcal{C}_{ik} = \{ \mathbf{x} \in \mathbb{R}^n \mid \mathbf{h}_{ik}(\mathbf{x}) \leq \mathbf{0} \}$, to which the point $\mathbf{x}$ belongs to when the disjunct is active, i.e., $Y_{ik}=True$.
The disjuncts within the disjunction are related through an inclusive-or operator $\vee$, which means that at least one Boolean variable in every disjunction, $Y_{ik},  k \in K$, is set to $True$.
Each disjunction defines a disjunctive set, like the ones introduced in the previous section.
$\Omega(\mathbf{Y})$ represent logical propositions in terms of the Boolean variables $\mathbf{Y}$.
These logical constraints can be written in Conjunctive Normal Form (CNF), i.e., $\Omega(\mathbf{Y}) = \bigwedge_{t\in T} \left[ \bigvee_{Y_{ik} \in R_t} (Y_{ik}) \bigvee_{Y_{ik} \in Q_t} (\neg Y_{ik}) \bigvee \right]$ where for each logical clause $t \in T$, the subset $R_t \subseteq \mathbf{Y}$ are non-negated Boolean variables and the subset $Q_t \subseteq \mathbf{Y}$ are the negated Boolean variables.
We assume that the exclusive-or operators among the Boolean variables for each disjunction $k \in K$, i.e., $\veebar_{i \in D_k} Y_{ik}$, are included in $\Omega(\mathbf{Y}) = True$~\citep{grossmann2012generalized,sawaya2006reformulations}.
% This would be wrong if we had extra binaries $Y$ with no constraints since the convex sets would be empty, and hence the HR reformulation wouldn't work
It has been proved that GDP is equivalent to disjunctive programming in the case that the constraints are linear~\citep{sawaya2012hierarchy} and convex~\citep{ruiz2012hierarchy}.

Besides offering an intuitive modeling paradigm of discrete problems through disjunctions, a GDP model can be used to inform computational solution tools, i.e., solvers, of the original problem's underlying structure, thus leading to improved solving performance.
The tailored solution methods for GDP are usually based on generalizing algorithms for MINLP, where the optimization problems are decomposed. In these methods, discrete variables are fixed and allow solving the problem only in terms of the continuous variables.
Different methods are used to select the combination of these discrete variables, including branching across the different values the discrete variables can take, i.e., Branch \& Bound (B\&B), or solving a linear approximation of the original problem~\citep{kronqvist2019review}.
For GDP algorithms, contrary to the case in MINLP, these Nonlinear Programming (NLP) subproblems only include the constraints that concern the logical variable combinations.
Among these tailored algorithms, we encounter the Logic-based Branch \& Bound (LBB) and the Logic-based Outer-Approximation (LOA).
For more information on general GDP algorithms, refer to~\citep{chenpyomo}.

Another route to solve these problems is through the reformulation to Mixed-integer problems, where binary variables $\mathbf{y} \in \{ 0,1 \}^{\sum_{k \in K} \lvert D_k \rvert}$ are added to the problem in exchange of the Boolean variables and constraints within the disjunction are enforced subject to the binary variables' value.
Notice that these reformulations yield problems of the form~\ref{prob:MINLP}.
The logical propositions $\Omega(\mathbf{Y})=True$ can be easily reformulated as a set of linear inequality constraints, $E\mathbf{y} \leq \mathbf{e}$, in terms of the binary variables~\citep{raman1994modelling,grossmann2012generalized,williams2013model}.
In the case that $\Omega(\mathbf{Y})$ is written in CNF, this reformulation is simply $\sum_{y_{ik} \in R_t} y_{ik} + \sum_{y_{ik} \in Q_t} (1-y_{ik}) \geq 1, t \in T$.
An example is the exclusive-or constraint $\veebar_{i \in D_k} Y_{ik}$ reformulated as a partitioning constraint $\sum_{i \in D_k} y_{ik} = 1,  k \in K$.
These approaches take advantage of the more mature mixed-integer solvers available commercially.

The big-M reformulation is among the best-known reformulations for GDP problems.
In this case, each disjunction's constraints are relaxed by adding a large term, $M$, if its corresponding binary variable equals zero.
The formulation of the big-M reformulation is as follows:
\begin{equation}
\label{prob:bigm}
\begin{aligned}
\min_{\mathbf{x},\mathbf{y}} &\ f(\mathbf{x}) \\
\textnormal{s.t. }&\ \mathbf{g}(\mathbf{x}) \leq \mathbf{0}\\
    &h_{ikj}(\mathbf{x})\leq M_{ikj}(1-y_{ik}), \quad  k \in K, i \in D_k, j \in \llbracket J_i \rrbracket, \\
    &\sum_{i \in D_k} y_{ik} = 1, \quad  k \in K \\
    & E\mathbf{y} \leq \mathbf{e}\\
&\ \mathbf{x}^{l} \leq \mathbf{x} \leq \mathbf{x}^{u}\\
&\ \mathbf{x} \in \mathbb{R}^{n}\\
&\ y_{ik} \in \{0,1\}, \quad  k \in K, i \in D_k,\\
\end{aligned} \tag{Big-M}
\end{equation}
where the coefficient $M_{ikj}$ has to be large enough to guarantee the enforcement of the original GDP logic, i.e., $y_{ik} = 1 \to \mathbf{h}_{ik}(\mathbf{x}) \leq \mathbf{0}$, but small enough to avoid numerical problems related to solving accuracy~\citep{trespalacios2014review}.
The numerical problems associated with large big-M values are related to possible ill-conditioned matrices appearing in the solution methods of the problem, rounding errors leading to false claiming of feasible solutions, and slow progress in B\&B algorithms; issues discussed by optimization practitioners \footnote{\url{https://thiagoserra.com/2017/06/15/big-m-good-in-practice-bad-in-theory-and-ugly-numerically/}}.
The smallest valid value for the big-M coefficient can be accomplished by setting $M_{ikj} = \max_{\mathbf{x} \in \{ \mathbf{x} : \mathbf{h}_{ik} \leq \mathbf{0},  \mathbf{x}^{l} \leq \mathbf{x} \leq \mathbf{x}^{u}\}} h_{ikj}(\mathbf{x}), j \in \llbracket J_i \rrbracket$.
This problem, in general, is not trivial to solve, given that it would involve maximizing a convex function over a convex domain; hence a potentially non-convex problem.
Although traditionally used, the big-M reformulation is well-known for its often weak continuous relaxation gap, i.e., the difference in the optimal objective function when solving the problem considering $y_{ik} \in [0,1] \subset \mathbb{R},  k \in K, i \in D_k$ compared to the original problem's optimal objective.%; even with the tightest values for the $M$ coefficients.
This is particularly important for solution methods based on B\&B, where this continuous relaxation gives the first node in the search tree.

Another valid transformation of problem~\ref{prob:gdp} into a mixed-integer problem is the hull reformulation (HR).
This reformulation uses the same mapping of Boolean into binary variables as in~\ref{prob:bigm}.
On the other hand, it introduces copies of the $\mathbf{x}$ variables, $\mathbf{v}_{ik}$ for each disjunct $k \in K, i \in D_k$ and uses the closure of the perspective function, $\left(\textnormal{cl }\tilde{\mathbf{h}}_{ik} \right)(\mathbf{v}_{ik},y_{ik}) \leq 0$, as defined in \eqref{eq:clperspective} to enforce the constraints when their corresponding binary variable is active.
Moreover, we use the standard topological notion of \emph{closure} of a set $\mathcal{C}$ as $\textnormal{cl }\mathcal{C}$.
We denote the unique closed extension or closure of a convex function $f$ as $\left( \textnormal{cl } f \right)$, as used in \citep{ceria1999convex,bonami2015mathematical}.
The formulation for the HR of a GDP is as follows:
\begin{equation}
\label{prob:hr}
\begin{aligned}
\min_{\mathbf{x},\mathbf{v},\mathbf{y}} &\ f(\mathbf{x}) \\
\textnormal{s.t. }&\ \mathbf{g}(\mathbf{x}) \leq \mathbf{0}\\
    & \mathbf{x} = \sum_{i \in D_k} \mathbf{v}_{ik}, \quad  k \in K \\
    &\left(\textnormal{cl }\tilde{\mathbf{h}}_{ik} \right)(\mathbf{v}_{ik},y_{ik}) \leq 0, \quad k \in K, i \in D_k\\
    &\sum_{i \in D_k} y_{ik} = 1, \quad k \in K \\
    & E\mathbf{y} \leq \mathbf{e}\\
&\ \mathbf{x}^{l}y_{ik} \leq \mathbf{v}_{ik} \leq \mathbf{x}^{u}y_{ik}\\
&\ \mathbf{x} \in \mathbb{R}^{n}\\
&\ \mathbf{v}_{ik} \in \mathbb{R}^{n}, \quad k \in K, i \in D_k\\
&\ y_{ik} \in \{0,1\}, \quad k \in K, i \in D_k.\\
\end{aligned} \tag{HR}
\end{equation}

We present the derivation of~\ref{prob:hr} in the Appendix, Section~\ref{sec:background.disjunct}.

In general, for GDP, no convexity assumptions are made for the functions $f, \mathbf{g}, \mathbf{h}_{ik}$ or the sets within the disjunctions $\mathcal{C}_{ik}$.
This means that the continuous relaxation of either~\ref{prob:bigm} or~\ref{prob:hr} might not have convex feasible regions.
We refer the interested reader to the review by Ruiz and Grossmann~\citep{ruiz2017global} that covers the techniques to solve these challenging optimization problems.

In order to use the theory from CP and Disjunctive programming, covered in Appendices~\ref{sec:background.cones} and~\ref{sec:background.disjunct}, respectively, we assume here that functions $f, \mathbf{g}, \mathbf{h}_{ik}$ are convex, hence the sets $\mathcal{C}_{ik}$ are convex too.
These are known as convex GDP problems~\citep{trespalacios2016cutting}.

For a literature review on GDP, we refer the reader to Grossmann and Ruiz~\citep{grossmann2012generalized}.

\section{Conic Generalized Disjunctive Programming}
\label{sec:cGDP}

The first step towards defining easily solvable convex MINLP problems via CP is to define a GDP with conic constraints.
As mentioned in Section~\ref{sec:background.cones}, we can use the tautological reformulation in~\eqref{eq:fascone} to write any convex GDP of form~\ref{prob:gdp} as follows:
\begin{equation}
\label{prob:gdpcone}
\begin{aligned}
\min_{\mathbf{x},\mathbf{Y}} &\ f(\mathbf{x}) \\
\textnormal{s.t. }&\ \mathbf{g}(\mathbf{x}) \leq \mathbf{0}\\
&\ \bigvee_{i\in D_k} \left[
    \begin{gathered}
    Y_{ik} \\
    \mathbf{A}_{ik}\mathbf{x} \succcurlyeq_{\mathcal{K}_{ik}} \mathbf{b}_{ik}\\
    \end{gathered}
\right], \quad  k \in K\\
&\ \Omega(\mathbf{Y}) = True \\
&\ \mathbf{x}^{l} \leq \mathbf{x} \leq \mathbf{x}^{u}\\
&\ \mathbf{x} \in \mathbb{R}^{n}\\
&\ Y_{ik} \in \{False, True\}, \quad k \in K, i \in D_k.\\
\end{aligned} \tag{GDP-Cone}
\end{equation}

Since the objective function $f(\mathbf{x})$ and the global constraints $\mathbf{g}(\mathbf{x}) \leq \mathbf{0}$ are convex, we can reformulate them to a conic program via~\eqref{eq:fascone} as in problem~\eqref{prob:MINLPcone}.
The sets defined within each disjunct
\begin{equation}
\label{eq:conicset}
P_{ik}: = \left \{ \mathbf{x}\in \mathbb{R}^n: \mathbf{A}_{ik}\mathbf{x} \succcurlyeq_{\mathcal{K}_{ik}} \mathbf{b}_{ik} \right \}
\end{equation}
are convex sets, where for every disjunct $\mathbf{A}_{ik} \in \mathbb{R}^{m_i \times n}$, $\mathbf{b}_{ik} \in \mathbb{R}^{m_i}$, and $\mathcal{K}_{ik}$ is a proper cone.

Although the derivation of specific solution algorithms for problem~\ref{prob:gdpcone} is a subject of active research, we focus on the reformulations of the given problem into Mixed-integer Programming problems.
These convex GDP problems can be reformulated into a convex ~\ref{prob:MINLP} problem, which in turn can be written down as a \ref{prob:MICP} problem.

The first trivial reformulation is the big-M reformulation, which yields the following problem:

\begin{equation}
\label{prob:bigmcone}
\begin{aligned}
\min_{\mathbf{x},\mathbf{y}} &\ f(\mathbf{x}) \\
\textnormal{s.t. }&\ \mathbf{g}(\mathbf{x}) \leq \mathbf{0}\\
    &\mathbf{A}_{ik}\mathbf{x} \succcurlyeq_{\mathcal{K}_{ik}} \mathbf{b}_{ik} + M_{ik}(1-y_{ik}), \quad k \in K, i \in D_k, \\
    &\sum_{i \in D_k} y_{ik} \leq 1, \quad k \in K \\
    & E\mathbf{y} \leq \mathbf{e}\\
&\ \mathbf{x}^{l} \leq \mathbf{x} \leq \mathbf{x}^{u}\\
&\ \mathbf{x} \in \mathbb{R}^{n}\\
&\ y_{ik} \in \{0,1\}, \quad k \in K, i \in D_k.\\
\end{aligned} \tag{Big-M-Cone}
\end{equation}

Contrary to the previous case, where the sets in the disjunctions were defined using scalar constraints, the tightest valid values for the big-M coefficients are given by $M_{ik} = \max_{\mathbf{x} \in P_{ik}} \mathbf{A}_{ik}\mathbf{x} - \mathbf{b}_{ik}$.
Notice that when $y_{ik}=1$, the original constraint within the disjunction is enforced, while when $y_{ik}=0$, the constrained is trivially satisfied as $\mathbf{A}_{ik}\mathbf{x} - \mathbf{b}_{ik} - M_{ik} \in \mathcal{K}_{ik} $.

To derive the HR of ~\ref{prob:gdpcone}, we need to characterize the convex hull of the disjunctive set~\eqref{eq:disjset} in the case that each convex and bounded set is defined using cones as in~\ref{eq:conicset}.

\begin{theorem}~\citep{lodi2019disjunctive}
\label{th:chullconic}
Let $\mathcal{P}_i = \left \{ \mathbf{x}\in \mathbb{R}^n: \mathbf{A}_{i}\mathbf{x} \succcurlyeq_{\mathcal{K}_{i}} \mathbf{b}_{i} \right \}$ for $i \in I$, where $\mathbf{A}_{i} \in \mathbb{R}^{m_i \times n}$, $\mathbf{b}_{i} \in \mathbb{R}^{m_i}$, and $\mathcal{K}_{i}$ is a proper cone, and let

\begin{equation}
\label{eq:chullconic}
\mathcal{P} = \left\{
\begin{aligned}
    &\mathbf{x} = \sum_{i \in I}\mathbf{v}_{i}, \\ 
    &\sum_{i \in I}\lambda_{i} = 1, \\
    &\mathbf{A}_{i}\mathbf{v}_{i} \succcurlyeq_{\mathcal{K}_{i}} \lambda_{i}\mathbf{b}_{i} ,  &i \in I, \\
    &\mathbf{v}_{i} \in \mathbb{R}^n ,  &i \in I, \\
    &\lambda_{i} \in \mathbb{R}_+ ,  &i \in I
\end{aligned}
\right \}.
\end{equation}

Then $\textnormal{conv}(\bigcup_{i \in I} \mathcal{P}_i) \subseteq \textnormal{proj}_{\mathbf{x}}(\mathcal{P})$ and:
\begin{enumerate}
    \item if $\mathcal{P}_i \neq \emptyset, \forall i \in I$, then $\textnormal{proj}_{\mathbf{x}}(\mathcal{P}) \subseteq \textnormal{cl conv}(\bigcup_{i \in I} \mathcal{P}_i)$
    \item if $\mathcal{P}_i = \mathcal{S}_i + \mathcal{W}, \forall i \in I$, where $\mathcal{S}_i, i \in I$ is a closed, bounded, convex, non-empty set and $\mathcal{W}$ is a convex closed set, then
    \begin{equation*}
        \textnormal{conv}\left(\bigcup_{i \in I} \mathcal{P}_i \right) = \textnormal{proj}_{\mathbf{x}}(\mathcal{P}) = \textnormal{cl conv} \left( \bigcup_{i \in I} \mathcal{P}_i \right).
    \end{equation*}
\end{enumerate}
\end{theorem}

\begin{proof}
See~\citep[Proposition 2.3.5]{ben2001lectures}.
\end{proof}
Using the characterization of the convex hull of the union of convex sets defined by cones that share the same recession cone ($\mathcal{W}$ in Theorem~\ref{th:chullconic}), we can define the HR of the~\ref{prob:gdpcone} as follows:

The formulation for the HR of a GDP is as follows:
\begin{equation}
\label{prob:hrcone}
\begin{aligned}
\min_{\mathbf{x},\mathbf{v},\mathbf{y}} &\ f(\mathbf{x}) \\
\textnormal{s.t. }&\ \mathbf{g}(\mathbf{x}) \leq \mathbf{0}\\
    & \mathbf{x} = \sum_{i \in D_k} \mathbf{v}_{ik}, \quad  k \in K \\
    &\mathbf{A}_{ik}\mathbf{v}_{ik} \succcurlyeq_{\mathcal{K}_{ik}} y_{ik}\mathbf{b}_{ik}, \quad  k \in K, i \in D_k\\
    &\sum_{i \in D_k} y_{ik} = 1, \quad  k \in K \\
    & E\mathbf{y} \leq \mathbf{e}\\
&\ \mathbf{x}^{l}y_{ik} \leq \mathbf{v}_{ik} \leq \mathbf{x}^{u}y_{ik}\\
&\ \mathbf{x} \in \mathbb{R}^{n}\\
&\ \mathbf{v}_{ik} \in \mathbb{R}^{n}, \quad  k \in K, i \in D_k\\
&\ y_{ik} \in \{0,1\}, \quad k \in K, i \in D_k.\\
\end{aligned} \tag{HR-Cone}
\end{equation}

This problem is of the form of~\ref{prob:MICP}, and more notably, uses the same cones within the disjunctions, $\mathcal{K}_{ik}$ in the extended formulation.
Contrary to problem~\ref{prob:hr}, problem~\ref{prob:hrcone} does not require an approximation of the perspective function.
Considering the HR reformulation as an optimization problem defined over convex cones allows exploiting the tight continuous relaxation of these problems while efficiently addressing the perspective reformulation's exact form.

To show several functions that appear in the normal context of convex MINLP that can be reformulated as the standard cones described in Section~\ref{sec:background.cones}, as well as their perspective function, we include Table~\ref{tab:reformulation}.
The conic representations in Table~\ref{tab:reformulation} are not unique and are given as a practical guide for implementing convex constraints using cones.
Notice that by applying the perspective reformulation, we recover the results found by several authors on stronger formulations for convex constraints activated through indicator variables.
Such examples include the epigraph of quadratic functions~\citep{gunluk2010perspective} and the epigraph of power functions with positive rational exponents~\citep{atamturk2018strong}.
The conic reformulation gives a natural and systematic procedure to perform extended reformulations~\citep{lubin2016extended}, which have been indicated to be helpful in solution methods for mixed-integer convex programs~\citep{tawarmalani2005polyhedral,hijazi2012mixed}.

To use the HR reformulation of GDP using conic constraints, it suffices to perform the take the perspective on its cones, i.e., for variables $\mathbf{z}$ defined over the cone $\mathcal{K}$ its perspective becomes $( y\frac{\mathbf{z}}{y}) \in \mathcal{K}$.
This has a considerable advantage, given that the HR reformulation is representable in the same cones as the ones used within the disjunctions.

\begin{sidewaystable}

\caption{Common convex constraints $\mathbf{h}(\mathbf{z}) \leq \mathbf{0}$ and perspective functions $\tilde{\mathbf{h}}(\mathbf{z},y) \leq \mathbf{0}$ with conic reformulation.}
\label{tab:reformulation}
\resizebox{\textwidth}{!}{%
\input{tables/reformulation}
}
\end{sidewaystable}

\section{Computational results}
\label{sec:results}

The computational results in this manuscript include the comparison of different mixed-integer reformulations of GDP problems.
The sources of these GDP problems are applications in Process Systems Engineering (PSE) and Machine Learning (ML), besides some randomly generated instances to benchmark the different solution methods.
Each different reformulation was tackled using MINLP solvers.
All the problems were implemented in the General Algebraic Modeling Software GAMS~\citep{bussieck2004general} 28.2.
The solvers used for this comparison are
% ANTIGONE~\citep{misener2014antigone} 1.1,
BARON~\citep{tawarmalani2005polyhedral} 19.7, CPLEX~\citep{cplex} 12.9, 
and KNITRO~\citep{waltz2017knitro} 11.1 
% and SCIP~\citep{achterberg2009scip} 6.0
for convex MINLP.
We also use as a MICP solver MOSEK~\citep{aps2018mosek} 9.0.98, using two different algorithms implemented within it for solving relaxations of the conic problems, either an interior-point solution or through an outer-approximation approach (\texttt{MSK\_IPAR\_MIO\_CONIC\_OUTER\_APPROXIMATION} set as \texttt{MSK\_OFF} or \texttt{MSK\_ON}), denoted MOSEK-IP and MOSEK-OA, respectively.
Given the sophistication of these solvers, the effects of the different problem formulations can be shadowed by the use of preprocessing techniques, heuristics, and other performance enhancement strategies within them.
To better observe the performance difference given by the problem formulation, we use the Simple Branch \& Bound SBB~\citep{bussieck2001sbb} implementation in GAMS and solve the respective continuous subproblems using gradient-based interior-point NLP solver
% IPOPT~\citep{wachter2006implementation} 3.12 using the Harwell Subroutine Library (HSL) MA27~\citep{hsl2002collection} as a solver for linear systems,
KNITRO~\citep{waltz2017knitro} 11.1, 
and MOSEK~\citep{aps2018mosek} 9.0.98 for the conic subproblems.
All experiments were run on a single thread of an Intel® Xeon® CPU (24 cores) 2.67 GHz server with 128GB of RAM running Ubuntu.
The termination criteria were a time limit of 3600 seconds or a relative optimality gap of $\epsilon_{rel}=10^{-5}$.
Unless otherwise stated, the conic reformulation of the constraints was written explicitly, meaning that the auxiliary variables required by the reformulation were introduced to the problem directly.
This is a weakness identified in the CP interface in GAMS, where the conic structure identification is not made automatically.
Although trivial, the definition of the cones had to be done manually.

For all these GDP problems, we present the big-M and HR reformulations.
When necessary, the conic representations for both cases, i.e.,~\ref{prob:bigmcone} and~\ref{prob:hrcone}, are presented separately from the algebraic description, i.e.,~\ref{prob:bigm} and ~\ref{prob:hr}.
For the HR, we use the $\varepsilon$-approximation \eqref{eq:furman_persp} proposed in ~\citep{furman2020computationally}.
We denoted it as HR-$\varepsilon$.
We use the recommended value of $\varepsilon=10^{-4}$ for all the cases presented herein.
We also implemented the perspective function explicitly and the $\varepsilon$-approximation \eqref{eq:lee_persp}. 
However, our results indicated that, in general, solvers could handle the numerical challenges associated with the perspective function better using the approximation in~\eqref{eq:furman_persp}.
Hence, we do not include the results of the direct implementation of the perspective function or the approximation given by~\eqref{eq:lee_persp}.
Yet, the interested reader can find the complete results online \footnote{\url{https://bernalde.github.io/conic_disjunctive/}}.

The mixed-integer big-M and hull reformulations of some of these instances are present in the benchmarking libraries MINLPLib~\citep{bussieck2003minlplib} and MINLP.org~\citep{cmuibmminlp}.
They have been widely used for MINLP solver benchmarks~\citep{Bonami2008,kronqvist2019review,furman2020computationally,bernal2020improving}.
This applies in particular for the PSE applications, Constrained Layout (\texttt{CLay*}), Process Networks (\texttt{proc*}), and Retrofit Synthesis instances (\texttt{RSyn*} and \texttt{Syn*}), motivating the study on these well-known instances.

Moreover, there has been recent interest from the Machine Learning (ML) community in using rigorous methods for non-convex optimization, contrary to heuristics based on convex relaxations.
Even considering the performance cost of the rigorous methods, the optimal solution to the original non-convex optimization problem is informative and valuable within an ML framework~\citep{flaherty2019map}.
Finding the optimal values of the parameters of a probability distribution such that a likelihood estimator is maximized, i.e., training, is known as Expectation-Maximization (EM) in ML~\citep{chen2010demystified}.
When the data labels are incomplete, the general problem can be stated as learning from weakly labeled data~\citep{li2013convex}.
While performing the training, the assignment of the labels is naturally representable through disjunctions, giving rise to mixed-integer programs.
For example, there has been a recent interest in tackling the clustering problem using mixed-integer programming~\citep{flaherty2019map}.
Optimally guaranteed solutions to a problem similar to~\ref{prob:kmean} lead to better results measured by the performance of the ML model arising from the clustering compared to local-optimization approaches to the EM problem.
The ML instances on $k$-means clustering (\texttt{kClus*}) and logistic regression (\texttt{LogReg*}) are inspired by problems proposed in the literature but are randomly generated for this manuscript.

The following results are presented in two subsections, one considering ``quadratic'' problems that can be formulated using second-order and rotated second-order cones, and the second one with problems modeled through the exponential cone.
Each formulation includes linear constraints, which can be managed by both gradient-based and conic mixed-integer convex programming solvers.
All the results from this manuscript are available in an open-access repository.

We report the nodes required by each solver.
The definition of a node might vary for every solver, and a detailed description of each case is not widely available.
To better control these reports, we compare SBB as a central manager for the branching procedures.
In this last procedure, we can guarantee that each node is the solution to a continuous convex optimization problem.

\subsection{Quadratic problems}

The three families of instances presented herein are the Constrained Layout problem, a $k$-means clustering optimization problem, and randomly generated instances.
All these problems share the characteristic that the constraints within the disjunctions are representable via second-order and rotated second-order cones.

The mixed-integer reformulations of these problems were implemented as in~\ref{prob:bigm} and~\ref{prob:hr}, both the HR-$\varepsilon$ and HR-Cone.
Notice that in the case of the second-order cone, the explicit definition of the cone can be replaced by the inequality~\citep{gunluk2010perspective}
\begin{equation}
    x^2 -ty \leq 0 \iff \sqrt{(2x)^2 + (y-t)^2} \leq y + t,
\end{equation}
that avoids the variable multiplication $ty$ and improves the performance of gradient-based solvers like IPOPT and KNITRO.
When implementing this alternative to the exact representation of the perspective function, the performance of KNITRO improves slightly at the expense of a significant decrease in BARON's performance.
Therefore, the implementation results are left out of this manuscript, although they are included in the repository for reference.

The examples in this section had constraints in their disjunctions directly identified as a cone by MOSEK in the GAMS interface.
This allowed the big-M instances to be written in their algebraic form.
In general, this might not be the case, with the cones requiring an explicit description for MOSEK to process them.
Simultaneously, the HR reformulation required the explicit introduction of additional constraints for the conic form to be accepted by the GAMS-MOSEK interface.
CPLEX, on the other hand, can automatically identify and transform certain general quadratic constraints into general and rotated second-order cones.

Below, we present the examples considered convex quadratic GDPs.

\subsubsection{Constrained layout problem}

The constrained layout problem is concerned with the minimization of the connection costs among non-overlapping rectangular units.
These units need to be packed within a set of fixed circles.
% Figure \ref{fig:Clay} illustrates the constrained layout problem. 
It can be formulated as the following convex GDP~\citep{sawaya2006reformulations}:
\begin{equation}\label{eq:Clay}
  \begin{aligned}
 	\min_{\mathbf{\delta x},\mathbf{\delta y},\mathbf{x},\mathbf{y},\mathbf{W},\mathbf{Y}} &\sum_{i,j\in N} c_{ij}(\delta x_{ij}+\delta y_{ij})\\ 
 	\textnormal{s.t. } &\delta x_{ij} \ge x_i - x_j & i,j \in N, i < j \\
 	&\delta x_{ij} \ge x_j - x_i & i,j \in N, i < j \\
 	&\delta y_{ij} \ge y_i - y_j & i,j \in N, i < j \\
 	&\delta y_{ij} \ge y_j - y_i & i,j \in N, i < j \\
 	&\left[
 	\begin{gathered}
 	Y_{ij}^1 \\
 	x_i + L_i/2 \leq x_j - L_j/2
 	\end{gathered}
 	\right] \vee
 	\left[
 	\begin{gathered}
 	Y_{ij}^2 \\
 	x_j + L_j/2 \leq x_i - L_i/2
 	\end{gathered}
 	\right] \\
 	&\vee
 	\left[
 	\begin{gathered}
 	Y_{ij}^3 \\
 	y_i + H_i/2 \leq y_j - H_j/2
 	\end{gathered}
 	\right] \vee
 	\left[
 	\begin{gathered}
 	Y_{ij}^4 \\
 	y_j + H_j/2 \leq y_i - H_i/2
 	\end{gathered}
 	\right]    & i,j \in N, i < j \\
 	&\underset{t \in T}{\vee} 
 	\left[
 	\begin{gathered}
 	W_{it} \\
 	(x_i+L_i/2-xc_t)^2 + (y_i+H_i/2-yc_t)^2 \leq r_t^2 \\
 	(x_i+L_i/2-xc_t)^2 + (y_i-H_i/2-yc_t)^2 \leq r_t^2 \\
 	(x_i-L_i/2-xc_t)^2 + (y_i+H_i/2-yc_t)^2 \leq r_t^2 \\
 	(x_i-L_i/2-xc_t)^2 + (y_i-H_i/2-yc_t)^2 \leq r_t^2
 	\end{gathered}
 	\right] &i \in N \\
 	&Y_{ij}^1 \veebar Y_{ij}^2 \veebar Y_{ij}^3 \veebar Y_{ij}^4 & i,j \in N, i < j\\
 	&\underset{t \in T}{\veebar} W_{it} &i \in N \\
 	&0 \leq x_i \leq x_i^{u} &i \in N \\
 	&0 \leq y_i \leq y_i^{u} &i \in N \\
 	&\delta x_{ij}, \delta y_{ij} \in \mathbb{R}_+  & i,j \in N, i < j \\
 	&x_{i}, y_{i} \in \mathbb{R}  & i \in N \\
 	&Y_{ij}^1,Y_{ij}^2,Y_{ij}^3,Y_{ij}^4 \in \{False,True\}  & i,j \in N, i < j \\
 	&W_{it} \in \{False,True\}  & i \in N, t \in T \\
  \end{aligned},
\end{equation}
where the coordinate centers of each rectangle $i \in N$ are represented through variables $x_i, y_i$, the distance between two rectangles $i,j \in N, i < j$ is given by variables $\delta x_{ij}$ and $\delta y_{ij}$, and $c_{ij}$ is the cost associated with it.
The first disjunction allows for the non-overlapping of the rectangles, and the second one ensures that each rectangle is inside of one of the circles $t \in T$, whose radius is given by $r_t$ and center specified by coordinates $(xc_t, yc_t)$.

The constraints in the second disjunction are representable through quadratic cones as follows:
\begin{equation}
\begin{aligned}
    & (x_i \pm L_i/2-xc_t)^2 + (y_i \pm H_i/2-yc_t)^2 \leq r_t^2  \\ & \iff 
 	(r_t,x_i \pm L_i/2-xc_t,y_i \pm H_i/2-yc_t) \in \mathcal{Q}^3.
\end{aligned}
\end{equation}

Seven different problem instances are defined through the variation of the number of circular areas to fit in the rectangle $\lvert T \rvert$ and the number of possible rectangles $N$, each instance being denoted \texttt{CLay$\lvert$T$\rvert\lvert$N$\rvert$}.

\subsubsection{\texorpdfstring{$k$}{k}-means clustering}

The $k$-means clustering problem is an optimization problem that appears in unsupervised learning.
This problem minimizes the total distance of a set of points to the center of $k$ clusters, varying the center's position and the assignment of which center determines the distance to each point.
This problem is usually solved through heuristics without guarantees of the quality of the solution.

Recently, Papageorgiou and Trespalacios~\citep{papageorgiou2018pseudo} proposed a GDP formulation for the $k$-means clustering problem, also used in~\citep{kronqvist2021between}.
The problem formulation reads as follows:
\begin{equation}
\label{prob:kmean}
  \begin{aligned}
 	\min_{\mathbf{c},\mathbf{d},\mathbf{Y}} &\sum_{i \in N} d_{i}\\ 
 	\textnormal{s.t. } \quad  
 	&c_{k-1,1} \leq c_{k,1}, \quad k \in \{2,\dots, \lvert K \rvert\} \\
 	&\underset{k \in K}{\bigvee}
 	\left[
 	\begin{gathered}
 	Y_{ik} \\
 	d_i \geq \sum_{j \in D}(p_{ij} - c_{kj})^2
 	\end{gathered}
 	\right], \quad i \in N, k \in K \\
 	&\veebar_{k \in K} Y_{ik}, \quad i \in N\\
 	&\mathbf{d} \in \mathbb{R}_+^{ \lvert N \rvert}\\
 	&\mathbf{c} \in \mathbb{R}^{\lvert K \rvert \times \lvert D \rvert} \\
 	&Y_{ik} \in \{False,True\}, \quad i \in N, k \in K, \\
  \end{aligned}
\end{equation}
where $N$ is the set of points given in $\lvert D \rvert$ dimensions, whose coordinates are given by $\mathbf{p} \in \mathbb{R}^{\lvert N \rvert \times \lvert D \rvert}$.
The variables are the center coordinates $\mathbf{c}$, and the squared distances of each point to its closest center are denoted by $\mathbf{d}$, where the sum of its elements is minimized.
The first constraint is a symmetry-breaking constraint. An arbitrarily increasing ordering in the first dimension is taken for the centers.
The disjunctions determine with which center $k$ is the distance to point $i$ computed, given that $Y_{ik} = True$.

The constraint for each disjunction $i \in N, k \in K$ is naturally representable as a rotated second-order cone
\begin{equation}
\begin{aligned}
    & d_i \geq \sum_{j \in D}(p_{ij} - c_{kj})^2  & \iff 
 	(0.5,d_i,p_{i1} - c_{k1},\dots,p_{i \lvert D \rvert} - c_{k \lvert D \rvert}) \in \mathcal{Q}_r^{2+ \lvert D \rvert}.
\end{aligned}
\end{equation}

We vary the number of clusters $\lvert K \rvert \in \{3,5\}$, the number of given points $\lvert N \rvert \in \{10,20\}$, and the dimensions of those points $\lvert D \rvert \in \{2,3,5\}$ leading to instance \texttt{kClus\_$\lvert$K$\rvert$\_$\lvert$N$\rvert$\_$\lvert$D$\rvert$\_x}.
\texttt{x} in this case denotes one of the random instances generated.
For this problem, we include 10 instances for each case varying the point coordinates $p_{ij}$ within bounds $[l,u]$ by sampling the random uniform distributions $U[l,b]$ as follows $ p_{ij} \in U[0,1], i \in N, k \in K$.

\subsubsection{Random examples}

We generate random quadratic GDP problems to test further the reformulations proposed in this manuscript.
The random quadratic GDP problems are of the form,

\begin{equation}
\label{prob:randquad}
\begin{aligned}
\min_{\mathbf{x},\mathbf{Y}} &\ \mathbf{c}^\top \mathbf{x} \\
\textnormal{s.t. } 
&\ \bigvee_{i\in D_k} \left[
    \begin{gathered}
    Y_{ik} \\
    \sum_{j \in \llbracket n \rrbracket}\left( a^{'}_{ijk} x_j^2 + a^{''}_{ijk} x_j \right) + a^{'''}_{ik} \leq 1\\
    \end{gathered}
\right], \quad  k \in K\\
&\ \veebar_{i \in D_k} Y_{ik}, \quad  k \in K \\
&\ \mathbf{x}^{l} \leq \mathbf{x} \leq \mathbf{x}^{u}\\
&\ \mathbf{x} \in \mathbb{R}^{n}\\
&\ Y_{ik} \in \{False, True\}, \quad  k \in K, i \in D_k,\\
\end{aligned} \tag{SOCP-rand-GDP}
\end{equation}
where the lower and upper bounds of variables $\mathbf{x}$, $\mathbf{x}^{l}$ and $\mathbf{x}^{u}$, are set at -100 and 100, respectively.

The constraint in each disjunct is representable as a rotated second-order cone,
\begin{equation}
\begin{aligned}
    & \sum_{j \in \llbracket n \rrbracket}\left( a^{'}_{ijk} x_j^2 + a^{''}_{ijk} x_j \right) + a^{'''}_{ik} \leq 1 
    \\ & \iff
     	\left( 0.5,t,\sqrt{a^{'}_{ijk}}x_j,\dots,\sqrt{a^{'}_{ink}}x_{n} \right) \in \mathcal{Q}_r^{n+2};
     	t + \sum_{j \in D}a^{''}_{ijk} x_j + a^{'''}_{ik} \leq 1.
\end{aligned}
\end{equation}

The different random instances were generated by varying the number of disjunctions $\lvert K \rvert \in \{5,10\}$, the number of disjunctive terms at each disjunction $\lvert D_k \rvert \in \{5,10\}$, and the dimensions of the $\mathbf{x}$ variables $n \in \{5,10\}$ leading to instance 
\texttt{socp\_random\_$\lvert$K$\rvert$\_$\lvert$Dk$\rvert$\_n\_x}.
\texttt{x} denotes the index of the random variable generated.
10 instances are generated for each case, varying the parameters within bounds $[l,u]$ by sampling the random uniform distributions $U[l,b]$ as follows: $a^{'}_{ijk} \in U[0.01,1], a^{''}_{ijk} \in U[-1,1], a^{'''}_{ik} \in U[-1,1], c_j \in U[-1000,1000], i \in D_k, k \in K, j \in \llbracket n \rrbracket$.
We also include instances \texttt{socp\_random\_2\_2\_2\_x}, which represent the illustrative example in ~\citep{papageorgiou2018pseudo}.

Notice that the $k$-means clustering formulation is a particular case of these randomly generated GDPs.
In particular, if we set  $\mathbf{a^{'}}=1, \mathbf{a^{''}}=2\mathbf{p}, \mathbf{a^{'''}}=\mathbf{p}^\top \mathbf{p}, \mathbf{c}=1$ we recover the $k$-means clustering problem.

\subsubsection{Results}

We generate a total of 217 GDP problems, which are transformed through a~\ref{prob:bigm} and~\ref{prob:hr}, this last using both HR-$\varepsilon$ and HR-Cone.
The main results are presented in Table~\ref{tab:quadratic_solvers}, where the solution times and nodes for the~\ref{prob:bigm}, ~\ref{prob:hr}, and ~\ref{prob:hrcone} reformulations using different commercial solvers are included.
Consider that the HR-$\varepsilon$ formulation introduces non-linearities in the formulation, preventing CPLEX and MOSEK from addressing it.

In general, we can observe that CPLEX applied to the big-M reformulation has the best performance for the \texttt{CLay*} and \texttt{kClus*} instances when considering runtime.
BARON applied to the same big-M formulation returns the optimal solution with the least explored nodes for the constrained layout problems.
This shows how the mature solvers for mixed-integer programming have implemented useful preprocessing techniques, heuristics, and performance enhancements to work with big-M formulation; the ubiquity of these formulations among practitioners motivates their development of strategies to work with these problems efficiently.
An example is that CPLEX identifies the big-M formulation and internally treats its constraints through specialized branching rules derived from indicator constraints~\citep{cplex}\footnote{\href{https://www.ibm.com/docs/en/icos/12.9.0?topic=optimization-best-practices-indicator-constraints}{IBM documentation}}.

When comparing only the big-M formulation solution, BARON solves all the problems with the least number of nodes for all instances.
This corresponds to the main focus of BARON on solving more ``meaningful'' nodes for the problem. However, it might incur a performance cost~\citep{khajavirad2018hybrid}.
This observation also appears when comparing all the HR formulation results, where BARON required the fewest nodes.
This applied to the HR-$\varepsilon$ and the HR-Cone formulations.

In terms of runtime, when comparing the HR formulations, we observe that CPLEX is the fastest solver for the \texttt{CLay*} instances, while MOSEK-IP is the one for the \texttt{kClus*} and \texttt{socp\_random*} problems.
Notice that the best-performing solvers for these instances, CPLEX and MOSEK, can be applied to the HR-$\varepsilon$ formulation.
This shows that using a conic formulation of the HR problem opens the possibility of using solvers that can provide better performance by, e.g., better exploiting the problem structure.
Even for general nonlinear solvers, such as BARON, the conic reformulation provides a performance improvement, given that the lifted reformulation can be exploited for tighter relaxations within the solver~\citep{tawarmalani2005polyhedral}.
On the other hand, solvers based on nonlinear B\&B, where each node is solved with a general NLP algorithm, such as interior-point methods, e.g., KNITRO, can worsen their performance when using the conic reformulation.
The non-differentiability of the cones, together with the larger subproblem sizes, can cause such a negative impact.
This can be alleviated by taking advantage of the conic structure, something that KNITRO has implemented as part of their presolve capabilities~\citep{waltz2017knitro}\footnote{\href{https://www.or2018.be/workshops/2/dl_ws_pdf}{KNITRO v11 presentation}}.
These observations are supported by the performance difference of the HR-$\varepsilon$ formulation and the HR-cone when using BARON and KNITRO.
Comparing these formulations, BARON performs slightly better using the HR-cone formulation.
On the other hand, the performance significantly drops when using KNITRO with the HR-cone formulation compared to HR-$\varepsilon$.

A better view of the general performance of the different solvers is given in Figure~\ref{fig:quadratic}.
These figures present absolute performance profiles~\cite{bussieck2014paver} accounting for the number of problems solved to a given gap of the optimal solution (0.1\% in this case) within a time or node limit.
These absolute performance profiles will be used in the manuscript to highlight the difference between the reformulations and solvers when addressing the problems presented herein by reporting the absolute metric of effort to solve the problems (time or number of subproblems/nodes) without scaling them with respect to any of the solution alternatives; hence providing an unbiased report of the computational results.
In general, as seen in the node profile of Figure~\ref{fig:quadratic}, the performance concerning nodes is superior for all solvers when using the HR, except for BARON.
Given the tightness of this formulation, this is expected behavior.
Moreover, regarding solution time, both algorithms used in MOSEK improve their performance when using an HR compared to the big-M case.
This shows that when modeling disjunctive CP, the HR is preferable for this solver.
The other solvers worsen their performance when using the extended formulations regarding solution time.

Of the total 217 instances, the solver that solved the most instances to within 0.1\% of the best-known solution was MOSEK-IP with 191, both using the big-M and HR-Cone formulations.
The alternative that solved the fewest instances was KNITRO applied to the HR-Cone formulation, solving 160.

For more granular results and explanations of the different models, we have included the constrained layout and $k$-mean clustering examples in the \url{https://minlp.org/} website\footnote{\url{https://minlp.org/library/problem/index.php?i=306&lib=MINLP} and \url{https://minlp.org/library/problem/index.php?i=307&lib=MINLP}}.

\begin{figure}
    \centering
    \includegraphics[width=\textwidth]{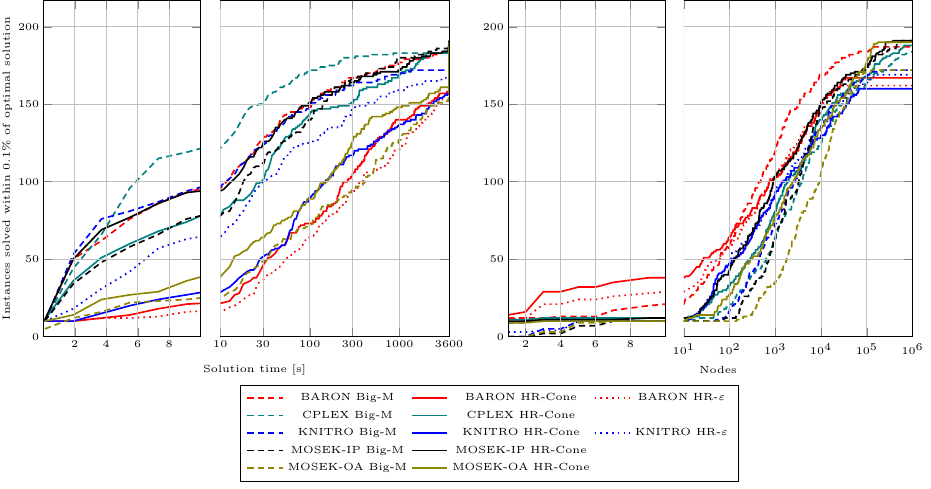}
    \caption{Time (left) and nodes (right) absolute performance profile for quadratic instances using the different GDP reformulations and commercial solvers.}
    \label{fig:quadratic}
\end{figure}

\subsection{Exponential problems}

As examples of problems representable using the exponential cone $\mathcal{K}_{exp}$, we present four families of problems: Process networks, Retrofit Synthesis Problems, Logistic Regression, and randomly generated instances.
The GAMS-MOSEK interface does not directly identify the exponential cone; therefore, we include algebraic and extended conic formulations,~\ref{prob:bigm} and ~\ref{prob:bigmcone}, respectively, for big-M.
\ref{prob:hr} and~\ref{prob:hrcone} are also tested for these problems, denoted as HR-$\varepsilon$ using the approximation in~\eqref{eq:furman_persp} and HR-Cone formulation through the extended formulation required by MOSEK for the exponential cones to be correctly identified.
The solver CPLEX was not used for these experiments since it cannot handle general nonlinear constraints beyond quadratics.

\subsubsection{Process Networks}

In the process network problem, we seek to maximize the profit from a process by deciding the equipment to be installed to fabricate some valuable product subject to material flows between the equipment pieces.
The total cost is computed from raw materials and equipment costs subtracted from the product's sales.
Alternative equipment pieces might induce a trade-off in terms of cost and production, defining the problem's constraints.
This classical problem in process design usually considers complex models describing each equipment piece.
For this simplified case~\citep{ruiz2012hierarchy,trespalacios2015improved}, we assume input-output correlations for each piece of equipment described by an exponential function.
This simplification still accounts for the non-linearity inherent to chemical processes.
The constraint considered here is a relaxation of the original equality constraints involving nonlinear terms, which is still valid given the direction of the optimization~\citep{ruiz2012hierarchy,trespalacios2015improved}.
% Figure \ref{fig:proc8} illustrates the superstructure for a process with potentially eight units presented by~\citep{kocis1989modelling}.
The problem can be modeled through the following convex GDP:
\begin{equation}
\label{prob:process}
  \begin{aligned}
 	\min_{\mathbf{c},\mathbf{x},\mathbf{Y}} &\sum_{k \in K}c_k + \sum_{j \in J}p_jx_j \\ 
 	\textnormal{s.t. } \quad &\sum\limits_{j \in J} r_{jn}x_j \leq 0, \quad  n \in N \\
 	&\bigvee_{i\in D_k} \left[
 	\begin{gathered}
 	Y_{ik} \\
 	\sum_{j \in J_{ik}} d_{ijk} (e^{x_j/t_{ijk}}-1)-\sum_{j \in J_{ik}}s_{ijk}x_j \leq 0 \\
 	c_k = \gamma_{ik}
 	\end{gathered}
 	\right], \quad k \in K\\
    &\veebar_{i \in D_k} Y_{ik}, \quad k \in K \\
 	&\Omega (\mathbf{Y})=True \\ 
 	&c_k, x_j \in \mathbb{R}_+, \quad j \in J_{ik}, i \in D_k, k \in K\\
 	&Y_{ik} \in \{False,True\}, \quad i \in D_k, k \in K.
  \end{aligned}
  \tag{Proc}
\end{equation}

In problem \ref{prob:process}, $c_k$ is the cost associated with the equipment chosen in disjunction $k \in K$.
The flow quantity $x_j$ is defined for each possible stream $j \in J$, with an associated profit.
The global mass balances are described for each node in the process $n \in N$ by the linear constraint $\sum_{j \in J} r_{jn}x_j \leq 0$, where $r_{jn}$ is the coefficient of the mass balance for flow $j$.
Each disjunction $k \in K$ presents the choice between $i \in D_k$ equipment alternatives.
When choosing each alternative ($Y_{ik} = True$), the corresponding input-output constraint in terms of the flows $j \in J_{ik}$ and parameters $d_{ijk}, t_{ijk}, s_{ijk} > 0$ is active, and the cost associated to that disjunction $c_k$ takes the value $\gamma_{ik}$.
The topology of the superstructure and extra logical constraints are included in $\Omega (\mathbf{Y})=True$.

% Using Figure~\ref{fig:proc8} as an example, there are two disjunctions $k \in \{ A,B \}$, where for the first one the set in the disjuncts $D_A = \{ 1,2 \}$, such that either equipment 1 ($Y_{1A}$) or equipment 2 ($Y_{2A}$) is chosen.
% Similarly $D_B = \{ 6,7 \}$.

An interesting alternative is where the sets $D_k$ yield a single element, and there is a Disjunction for every equipment piece.
This yields the following formulation:
\begin{equation}
\label{prob:processb}
  \begin{aligned}\min_{\mathbf{c},\mathbf{x},\mathbf{Y}} &\sum_{k \in K}c_k + \sum_{j \in J}p_jx_j \\  
 	\textnormal{s.t. } \quad &\sum_{j \in J} r_{jn}x_j \leq 0, \quad  n \in N \\
 	&\left[
 	\begin{gathered}
 	Y_{k} \\
 	\sum\limits_{j \in J_{k}} d_{jk} (e^{x_j/t_{jk}}-1)-\sum\limits_{j \in J_{k}}s_{jk}x_j \leq 0 \\
 	c_k = \gamma_{k}
 	\end{gathered}
 	\right] \vee
 	\left[
 	\begin{gathered}
 	\lnot Y_k \\
 	x_j = 0, \quad j \in J_k \\
 	c_k = 0
 	\end{gathered}
 	\right], \quad k \in K\\ 
 	&\Omega (\mathbf{Y})=True \\ 
 	&c_k, x_j \in \mathbb{R}_+, \quad j \in J_{k}, k \in K\\
 	&Y_{k} \in \{False,True\}, \quad k \in K.
  \end{aligned}
  \tag{Procb}
\end{equation}
This case allows several pieces of equipment to be built within each alternative as long as the objective is maximized.
The fact that it represents the disjunction of a convex set and a single point means that the HR formulation will yield the convex hull of the union of these sets without requiring an extended formulation~\citep[Corollary 1]{gunluk2010perspective}.

The exponential input-output constraint can be formulated in conic form as follows:
\begin{equation}
    \sum_{j \in J} d_{j} (e^{x_j/t_{j}}-1)-\sum_{j \in J}s_{j}x_j \leq 0 \iff
         \sum_{j \in J} d_{j} u_j-\sum_{j \in J}s_{j}x_j \leq 0;
         (t_j u_j + t_j, t_j, x_j) \in \mathcal{K}_{exp}.
\end{equation}

We include 5 variants of the process problem with $\lvert K \rvert \in \{ 21,31,36,48,100 \}$ possible units.
The first four cases are taken from~\citep{sawaya2006reformulations,ruiz2012hierarchy,trespalacios2016cutting}. This manuscript's last case was generated, given that commercial solvers can trivially solve the smaller cases.
The instances are denoted \texttt{process$\lvert$K$\rvert$} or \texttt{process$\lvert$K$\rvert$b} when implementing problems~\ref{prob:process} and~\ref{prob:processb}, respectively.
For the new instance, the parameters are chosen from the uniform distributions $d_{ijk} \in U[1,1.2], t_{ijk} \in U[1,1.3], s_{ijk} \in U[0.8,1.2],\gamma_{ik} \in U[2,3]$.

\subsubsection{Simultaneous Retrofit and Synthesis problems}

A generalization of the process network problem is the simultaneous retrofit and synthesis problem.
In this problem, an existing process network needs to be upgraded.
To do so, one can consider either installing new equipment or improving the existing one.
The potential of this process is to be maximized, given a budget constraint.
This problem was first proposed by Jackson and Grossmann~\citep{jackson2002high}, and its GDP implementation was done by Sawaya~\citep{sawaya2006reformulations}.
In the synthesis problem, the problem is equivalent to~\ref{prob:process} with an extra index for the time periods.
The retrofit synthesis problem contains additional linear constraints and disjunctions representing the conditions associated with retrofitting the existing process units.
The complete formulation is available in~\citep{furman2020computationally}.

The instances solved here are parametric to the number of synthesis processes $\lvert S \rvert \in \{5, 10, 15, 30, 40\}$, the number of retrofit units $\lvert R \rvert \in \{ 8 \}$ and the number of time periods considered $\lvert T \rvert \in {1, 2, 3, 4}$, leading to instances \texttt{Syn$\lvert$S$\rvert$M$\lvert$T$\rvert$} and \texttt{RSyn$\lvert$R$\rvert\lvert$S$\rvert$M$\lvert$T$\rvert$}.

\subsubsection{Logistic Regression}

Logistic regression is a training technique for binary classification.
In this training task, given a set of $D$-dimensional points $\mathbf{p}_i \in \mathbb{R}^{D}, i \in I$, we will assign a binary classifier $y \in \{ 0,1 \}$ to each point in the case that they lie above or below a hyperline given by $\mathbf{\theta}^\top \mathbf{p}_i$.
This line needs to be determined such that the logistic cost function is minimized.
The logistic cost function $\log(1/(1+e^{-\mathbf{\theta}^\top \mathbf{p}_i + \theta_0}))$ can be interpreted as the probability of a point belonging to the class given by $y=1$.
This problem can be modeled as a GDP by encoding the binary classifier $y$ in a Boolean variable $Y$ and writing the constraints within the disjunctions as follows:
\begin{equation}
\label{prob:logreg}
  \begin{aligned}\min_{\mathbf{\theta},\mathbf{t}} &\sum_{i \in I}t_i \\  
 	\textnormal{s.t. }  
 	&\left[
 	\begin{gathered}
 	Y_i \\
 	t_i \geq \log \left( 1+e^{-\mathbf{\theta}^\top \mathbf{p}_i + \theta_0} \right) \\
 	\mathbf{\theta}^\top \mathbf{p}_i \geq 0
 	\end{gathered}
 	\right] \vee
 	\left[
 	\begin{gathered}
 	\lnot Y_i \\
 	t_i \geq \log \left( 1+e^{\mathbf{\theta}^\top \mathbf{p}_i + \theta_0} \right) \\
 	\mathbf{\theta}^\top \mathbf{p}_i \leq 0
 	\end{gathered}
 	\right], \quad i \in I\\ 
 	&\Omega (\mathbf{Y})=True \\ 
 	&t_i \in \mathbb{R}_+, \quad i \in I\\
 	&\theta_0 \in \mathbb{R}\\
 	&\theta_j \in \mathbb{R}, \quad j \in \llbracket D \rrbracket\\
 	&Y_{k} \in \{False,True\}, \quad i \in I,
  \end{aligned}
  \tag{LogReg}
\end{equation}
where the logical constraints $\Omega (\mathbf{Y})=True$ can enforce symmetry-breaking constraints to help in the solution process or other additional constraints related to the regression task.

The logistic regression constraint can be expressed as the following conic inequality:
\begin{equation}
    t \geq \log \left( 1+e^{\mathbf{\theta}^\top \mathbf{p}_i + \theta_0} \right) \iff
         u + v \leq 1;
         x = \mathbf{\theta}^\top \mathbf{p}_i + \theta_0;
         (v,1,-t) \in \mathcal{K}_{exp};
         (u,1,x-t) \in \mathcal{K}_{exp},
\end{equation}
and equivalently for the complementary disjunction.

We generate ten random instances for each of the following settings in the examples presented herein.
We set the value of the points' dimensions within $D \in \{2,5,10\}$, $\lvert I \rvert = 20$, and we choose to generate 2 clusters of normally distributed points being at a Mahalanobis distance~\citep{de2000mahalanobis}, i.e., a distance metric between points and distributions, such that the points are at most $\sigma \in \{ 1, 2 \}$ standard deviations away from the center of the distributions.
This is computed via an inverse $\chi$-squared distribution with $D$ degrees of freedom calculated at probabilities \{0.68,0.95\} corresponding to Mahalanobis distances of $\sigma \in \{ 1, 2 \}$ in the one-dimensional case.
This distance is then divided in $2\sqrt{D}$, such that we place the centers of the distributions at opposite corners of the $D$-dimensional hypercube.
As mentioned in~\citep{flaherty2019map}, a natural advantage of the mathematical programming approach to the training tasks in ML, compared to the heuristics, is that additional constraints can be enforced through the problem formulation.
In this case, within $\Omega (\mathbf{Y})=True$, we force the split between the data points to be within 45\% and 55\% and also force the farthest two points in the set from the origin to belong to opposite classes as a symmetry breaking constraint.
Instances generated by this method are denominated \texttt{LogReg\_D\_$\lvert$I$\rvert$\_$\sigma$\_x}.

\subsubsection{Random examples}

Besides the applications-related instances listed above, we generate random instances whose disjunctive constraints can be represented using $\mathcal{K}_{exp}$.
The form of the GDP is:
\begin{equation}
\label{prob:randexp}
\begin{aligned}
\min_{\mathbf{x},\mathbf{Y}, z} &\ \mathbf{c}^\top \mathbf{x} \\
\textnormal{s.t. } 
&\ \bigvee_{i\in D_k} \left[
    \begin{gathered}
    Y_{ik} \\
    a^{'}_{ik} \exp{\sum_{j \in \llbracket n \rrbracket} a^{''''}_{ijk} x_j} \leq a^{''}_{ik} z + a^{'''}_{ik} \\
    \end{gathered}
\right], \quad  k \in K\\
&\ \veebar_{i \in D_k} Y_{ik}, \quad k \in K \\
&\ \mathbf{x}^{l} \leq \mathbf{x} \leq \mathbf{x}^{u}\\
&\ z \leq z^{u}\\
&\ \mathbf{x} \in \mathbb{R}^{n}\\
&\ z \in \mathbb{R}\\
&\ Y_{ik} \in \{False, True\}, \quad  k \in K, i \in D_k,\\
\end{aligned} \tag{EXP-rand-GDP}
\end{equation}
where the upper and lower bounds of variables $\mathbf{x}$, $\mathbf{x}^{l}$ and $\mathbf{x}^{u}$, are set at 0 and 10, respectively.
An upper bound for $z$ is given by
\begin{equation}
    z^u = \max_{i \in D_k, k \in K} \left[ \frac{a^{'}_{ik} \exp{\sum_{j \in \llbracket n \rrbracket} a^{''''}_{ijk} x_j^l} - a^{'''}_{ik}}{(a^{''}_{ik})^2} \right].
\end{equation}

The exponential constraint can be written equivalently as a logarithmic constraint and in a conic form as follows:
\begin{equation}
\begin{aligned}
    & a^{'}_{ik} \exp{\sum_{j \in \llbracket n \rrbracket} a^{''''}_{ijk} x_j} \leq a^{''}_{ik} z + a^{'''}_{ik} \\
    & \iff \log(a^{'}_{ik}) + \sum_{j \in \llbracket n \rrbracket} a^{''''}_{ijk} x_j \leq \log(a^{''}_{ik} z + a^{'''}_{ik}) \\
    & \iff 
         a^{'}_{ik} v_{ik} \leq a^{''}_{ik} z + a^{'''}_{ik};
         \left( v_{ik}, 1, \sum_{j \in \llbracket n \rrbracket} a^{''''}_{ijk} x_j  \right) \in \mathcal{K}_{exp}.
\end{aligned} 
\end{equation}

The generation of the random exponential GDPs uses the same parameters as the random quadratic GDPs, i.e., $\lvert K \rvert \in \{ 5, 10 \}, D_k \in \{ 5,10 \}$, and $n \in {5, 10}$.
Ten instances, denoted \texttt{exp\_random\_$\lvert$K$\rvert$\_$\lvert$Dk$\rvert$\_n\_x}, are generated for each combination, besides a simple case with \texttt{exp\_random\_2\_2\_2\_x} and the extra parameters are drawn from uniform distributions as $a^{'}_{ik} \in U[0.01,1], a^{''}_{ik} \in U[0.01,1], a^{'''}_{ik} \in U[0.01,1], a^{''''}_{ijk} \in U[0.01,1], c_j \in U[-1,-0.01], i \in D_k, k \in K, j \in \llbracket n \rrbracket$.

\subsubsection{Results}

We solve 208 GDP instances that are representable through the exponential cone.
These instances are transformed through big-M and HR.
Since the exponential cone is not automatically identified through the constraints defining it, the explicit description of the cone was required, giving rise to two different versions of each reformulation.
The big-M results are summarized in Table~\ref{tab:exp_solvers_bm}, and the HR results are included in Table~\ref{tab:exp_solvers_hr} in the Appendix.

Depending on the family of instances, a given combination of solver and reformulation was the best in runtime.
For the \texttt{LogReg*} and \texttt{RSyn*} instances, MOSEK-OA was the best solver when applied to the HR-Cone formulation.
The other algorithm for MOSEK, MOSEK-IP, was the best performance solver for the \texttt{proc*} instances, with the outstanding solution of the \texttt{proc\_100} problems in less than 5 seconds when most other approaches could not solve it within the 1-hour time limit.
The closest non-conic approach was BARON applied to the original big-M formulation. A $\approx$ 80x and 6x speedup was obtained with instances \texttt{proc\_100} and \texttt{proc\_100b}, respectively.
BARON applied to the big-M formulation was the best among all solvers for the \texttt{Syn*} instances.
This approach was the fastest for the \texttt{exp\_random*} instances, with a similar performance when applied to the big-M-cone formulation.
This was not the case in general, where the conic formulation of the big-M problem led to considerable performance degradation for BARON when solving the \texttt{LogReg*} and \texttt{Syn*} instances.
When considering the big-M-cone formulation, we see that both KNITRO and MOSEK-IP time out for most instances.

When considering the HR, using a conic formulation severely affected the performance of BARON and KNITRO.
This was a sign of the challenges that gradient-based methods encounter when facing exponential constraints such as the ones appearing in the conic reformulation.
For example, in instance \texttt{RSyn0805M02}, the HR-Cone formulation led to KNITRO failing to evaluate the gradients at every B\&B node, given numerical instability by the evaluation of exponential functions.
BARON, running its default version with a dynamic NLP subsolver selection, could not find a solution to this problem either. At the same time, a solver that takes advantage of the exponential cone, such as MOSEK, solved the problem in 2 seconds.

As with the quadratic instances, the most efficient solver in terms of nodes explored to find the optimal solution is BARON, both in the big-M and HR.

Absolute performance profiles are presented in Figure~\ref{fig:exponential} for the exponential instances.
In the time absolute performance profile in Figure~\ref{fig:exponential}, we observe a clear dominance of both MOSEK algorithms applied to the HR-Cone formulation, particularly within the first seconds.
Towards the end of the time limit, BARON applied to both the big-M and HR formulations, solving more instances to optimality.
BARON applied to the HR-$\varepsilon$ approximation can solve all the exponential problems within the time limit.
Except for BARON, all solvers improve their performance when comparing the big-M and HR formulations.
Having mentioned that, BARON and KNITRO have difficulties solving the HR-Cone formulation, with the extreme case of BARON failing in all instances.

When observing the node absolute performance profile for the exponential instances in Figure~\ref{fig:exponential}, the HR formulations require fewer nodes than the big-M formulations, except for BARON.
BARON proves that it generates strong relaxation nodes, requiring fewer to solve the problems, clearly dominating in this sense the other solvers.
A similar observation was made regarding the quadratic instances.

\begin{figure}
    \centering
    \includegraphics[width=\textwidth]{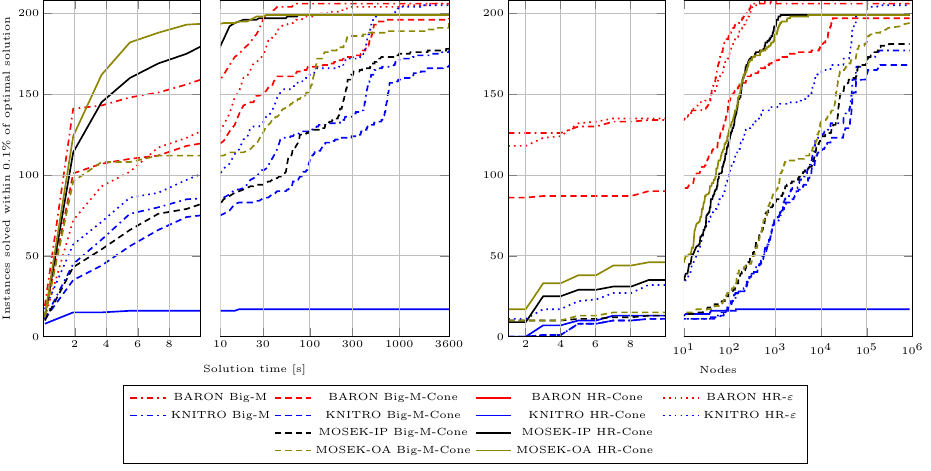}
    \caption{Time (left) and Nodes (right) absolute performance profile for exponential instances using the different GDP reformulations and commercial solvers.}
    \label{fig:exponential}
\end{figure}

\subsection{Controlling the Branch \& Bound search}

The implementations of modern solvers include an arsenal of preprocessing techniques, heuristic methods, and performance enhancements to tackle the challenging optimization problems at hand more efficiently.
Although this leads to performance improvements, it obscures the effect of better formulations when solving optimization problems.
To that end, we consider using the Simple Branch\& Bound (SBB) implementation in GAMS and solve the subproblems using both KNITRO and MOSEK.
These subproblems are continuous optimization problems, while SBB manages the discrete variables' exploration.
We present below two absolute performance profiles in Figure~\ref{fig:sbb} for all the problems solved in this manuscript, mainly including results of SBB-KNITRO and SBB-MOSEK.

In Figure~\ref{fig:sbb}, we observe the absolute performance profiles of the SBB implementation against the number of continuous convex subproblems solved.
The first observation is that the HR tight formulation allows a more efficient exploration of the subproblems solved than the big-M formulation.
The conic formulation of HR affects the performance of KNITRO when addressing the subproblems, leading to poor performance in this case.
Moreover, given the same branching rules, the big-M and HR formulations require approximately the same number of subproblems solved using the original or the extended formulations arising from the conic description of the problems.
This is an expected result, given that the extended formulation does not require additional binary variables.

Although the number of solved subproblems is similar, the time required to solve them varies depending on the chosen solver, as observed in Figure~\ref{fig:sbb}.
This figure includes the time absolute performance profiles for the SBB alternatives.
For reference, we include the best commercial alternative to each reformulation.
This corresponds to BARON for the big-M and HR-$\varepsilon$ and MOSEK-IP for the HR-Cone formulations.
The solver that solved the most instances was BARON applied to the big-M formulation, solving 393 out of the 425 problems, followed by MOSEK-IP applied to the HR-Cone formulation, solving 390 problems.
MOSEK is generally more efficient at solving the convex subproblems than KNITRO.
The difference is exacerbated in the HR formulation.
An interesting observation is that the gap in time performance between SBB and the best alternative is smaller for HR-$\varepsilon$ than for HR-Cone.
This indicates that the efficient exploitation of the conic constraints, in this case from MOSEK, can yield considerable performance advantages together with a tight reformulation of disjunctive constraints.

% \begin{sidewaystable}
% \caption{Results for Quadratic GDPs using different mixed-integer reformulations and solvers through SBB branching. The results corresponding to the solver with the least time and the fewest nodes for each instance with a reformulation are italized. The best results overall are bolded.}
% \label{tab:quadratic_sbb}
% \resizebox{\textwidth}{!}{%
% \input{tables/quadratic_sbb}
% }
% \end{sidewaystable}

\begin{figure}
    \centering
    \includegraphics[width=\textwidth]{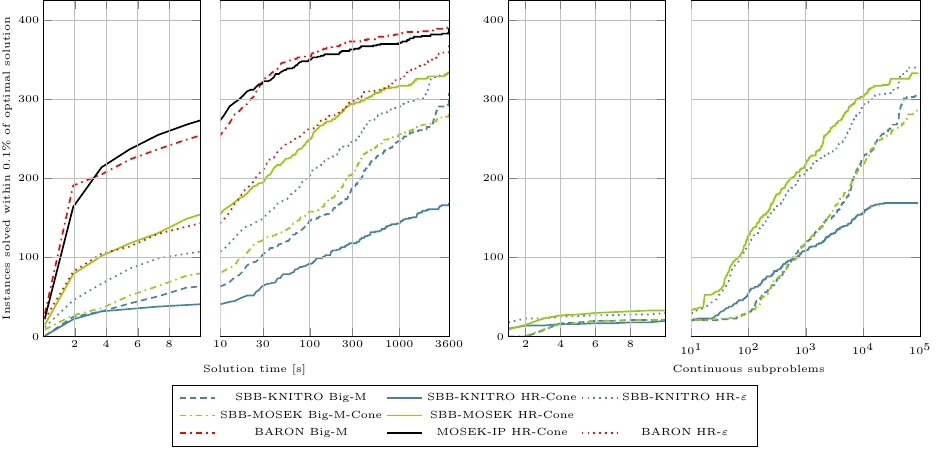}
    \caption{Time (left) and Solved subproblems (right) absolute performance profile for all instances using the different GDP reformulations and solvers through SBB. For the time profiles we include the best performing commercial solver results for each reformulation.}
    \label{fig:sbb}
\end{figure}

% \begin{figure}
%     \centering
%     \includegraphics[width=\textwidth]{tikz/sbb_time.pdf}
%     \caption{Time performance profile for all instances using the different GDP reformulations and solvers through SBB. We include the best-performing commercial solver results for each reformulation.}
%     \label{fig:sbb_time}
% \end{figure}

\section{Conclusions, discussion, and future work}
\label{sec:conclusions}

This work presents the formulation of convex Generalized Disjunctive Programming (GDP) problems using conic sets.
Convex GDP problems can be solved by reformulating them into convex Mixed-Integer Nonlinear Programming (MINLP) problems.
Two of those reformulations are covered in this manuscript: the big-M and hull reformulations.
The hull reformulation (HR) of a convex GDP problem requires implementing a perspective function, whose algebraic form is challenging for gradient-based nonlinear optimization solvers.
We present the big-M and hull reformulations into Mixed-Integer Conic Programming (MICP) problems through the conic formulation of the problem.
The MICP problems can be efficiently tackled using specialized conic solvers, which take advantage of the properties of the conic programs.
We provide a guide to reformulating common convex constraints through conic programming.
If those constraints appear inside disjunctions, we also provide a conic representation of its perspective, allowing the exact representation of the HR.

These reformulations were tested using a large set of convex GDP problems stemming from Process Systems Engineering, Machine Learning, and randomly generated instances.
These instances were classified as quadratic and exponential and solved through different reformulation alternatives and solvers.
Our results show how the conic reformulation gives a systematic and natural extended formulation of the convex MINLP problems stemming from GDP.
These can be exploited by solvers, allowing a more efficient solution to these problems.
Among the tested approaches, we identified that BARON solving the big-M formulation and MOSEK solving the HR-Conic formulation, either with IP or OA, were the most efficient solvers to tackle these convex GDP reformulated problems.
In general, we show how the conic representation of convex constraints within disjunctions can result in an exact and more efficiently solvable mixed-integer reformulation of a convex GDP.

The results in this paper also point to specific improvement opportunities.
First, the automatic reformulation of the convex constraints into cones is worth pursuing.
Previous success in the quadratic case allows commercial solvers such as CPLEX or Gurobi to automatically detect conic structures and address those more efficiently.
An extension of these routines to exponential cones is of interest.
Modeling extensions that allow for disjunctive programming are the natural place to include these automatic reformulations.
Approaches have been made at the modeling language level, e.g., in GAMS~\citep{vecchietti1999logmip} Pyomo~\citep{chenpyomo}~\footnote{\url{https://pyomo.readthedocs.io/en/latest/modeling_extensions/gdp/}}, and Julia~\footnote{\url{https://github.com/rdeits/ConditionalJuMP.jl}, \url{https://github.com/hdavid16/DisjunctiveProgramming.jl}}.
These could also be made at the solver level, with indicator constraints such as in CPLEX~\citep{cplex} and MOSEK~\citep{aps2018mosek}.
These techniques have also shown potential for the global optimization of non-convex GDP or MINLP~\citep{mahajan2010exploiting}, motivating further research into it.

We also highlight the effort made by the modeling layer in Julia, MathOptInterface~\cite{legat2021mathoptinterface}, that supports conic modeling in the form $\mathbf{A}\mathbf{x} - \mathbf{b} \in \mathcal{K}$ with $\mathcal{K}$ being one of many supported cones, including the ones covered in this manuscript.
By natively supporting conic modeling, it provides an alternative to solvers to detect these structures from scalar inequalities automatically.

Interesting future directions are the exploration of conic formulations in more advanced reformulations of GDPs, such as intermediate big-M / hull reformulations~\citep{kronqvist2021between} and basic steps reformulations~\citep{ruiz2012hierarchy}.
Moreover, conic programming tools can be used in more advanced solution methods of GDP than recasting the problem into MINLP.
Examples of those methods are Lagrangean decomposition based on the disjunctive structure of the problem~\citep{papageorgiou2018pseudo} or logic-based algorithms~\citep{chenpyomo}.
The use of conic programming has already shown the potential speedup for mixed-integer programming solutions~\citep{coey2020outer}, and expanding those findings to GDP is of great interest.

\section*{Acknowledgments}
The authors gratefully acknowledge financial support from the Center of Advanced  Process Decision-making and from the US Department of Energy, Oﬃce of Fossil Energy’s Crosscutting Research, Simulation Based Engineering Program through the Institute for the Design of Advanced Energy Systems (IDAES).

\bibliographystyle{unsrtnat}
\bibliography{conic_refs}

\appendix
\section{Background}
\label{sec:background}

In this manuscript, we use a similar notation to the one used by Ben-Tal and Nemirowski~\citep{ben2001lectures} and Alizadeh and Goldfarb~\citep{alizadeh2003second}.
We use lowercase boldface letters, e.g., $\mathbf{x,c}$, to denote column vector, and uppercase boldface letters, e.g., $\mathbf{A}, \mathbf{X}$, to denote matrices.
Sets are denoted with uppercase calligraphic letters, e.g., $\mathcal{S}, \mathcal{K}$.
Subscripted vectors denote $\mathbf{x_i}$ denote the $i^{th}$ block of $\mathbf{x}$.
The $j^{th}$ component of the vectors $\mathbf{x}$ and $\mathbf{x_i}$ are indicated as $x_j$ and $x_{ij}$.
The set $\{1,\dots,J\}$ is represented by the symbol $\llbracket J \rrbracket$.
Moreover, the subscript $\llbracket J \rrbracket$ of a vector $\mathbf{x}$ is used to define the set $\mathbf{x}_{\llbracket J \rrbracket} := \{\mathbf{x}_1,\dots,\mathbf{x}_J\}$.
We use $\mathbf{0}$ and $\mathbf{1}$ for the all zeros and all ones vector, respectively, and $0$ and $I$ for the zero and identity matrices, respectively.
The vector $e_j$ will be the vector with a single $1$ in position $j$, and its remaining elements are 0.
The dimensions of the matrices and vectors will be clear from the context.
We use $\mathbb{R}^k$ to denote the set of real numbers of dimension $k$, and for set $\mathcal{S} \subseteq \mathbb{R}^k$, we use $\textnormal{cl}(\mathcal{S})$ and $\textnormal{conv}(\mathcal{S})$ to denote the closure and convex hull of $\mathcal{S}$, respectively.

For concatenated vectors, we use the notation that ``,'' is row concatenation of vectors and matrices, and ``;'' is column concatenation.
For vectors, $\mathbf{x,y}$ and $\mathbf{z}$, the following are equivalent.
\begin{equation}
    \begin{pmatrix}
           \mathbf{x} \\
           \mathbf{y} \\
           \mathbf{z}
         \end{pmatrix} = (\mathbf{x}^{\top},\mathbf{y}^{\top},\mathbf{z}^{\top})^{\top} = (\mathbf{x};\mathbf{y};\mathbf{z}).
\end{equation}

The projection of a set $\mathcal{S} \subseteq \mathbb{R}^k$ onto the vector $\mathbf{x} \in X \subseteq \mathbb{R}^n$, with $n \leq k$ is denoted as $\textnormal{proj}_{\mathbf{x}}(\mathcal{S}) := \{ \mathbf{x} \in X : \exists \mathbf{y}: (\mathbf{x};\mathbf{y}) \in \mathcal{S} \}$.

If $\mathcal{A} \subseteq \mathbb{R}^k$ and $\mathcal{B} \subseteq \mathbb{R}^l$ we denote their Cartesian product as $\mathcal{A} \times \mathcal{B} := \{(\mathbf{x};\mathbf{y}): \mathbf{x} \in \mathcal{A}, \mathbf{y} \in \mathcal{B} \}$.

For $\mathcal{A}_1, \mathcal{A}_2 \subseteq \mathbb{R}^k$ we define the Minkowski sum of the two sets as $\mathcal{A}_1 + \mathcal{A}_2 = \{ \mathbf{u} + \mathbf{v}:  \mathbf{u} \in \mathcal{A}_1, \mathbf{v} \in \mathcal{A}_2  \}$.

\subsection{Cones}
\label{sec:background.cones}
For a thorough discussion about convex optimization and conic programming, we refer the reader to~\citep{ben2001lectures}.
The following definitions are required for the remainder of the manuscript.

The set $\mathcal{K} \subseteq \mathbb{R}^k$ is a cone if $\forall(\mathbf{z},\lambda) \in \mathcal{K} \times \mathbb{R}_+, \lambda \mathbf{z} \in \mathcal{K}$.
The dual cone of $\mathcal{K} \subseteq \mathbb{R}^k$ is
\begin{equation}
    \mathcal{K}^* = \left\{ \mathbf{u} \in \mathbb{R}^k : \mathbf{u}^T\mathbf{z} \geq \mathbf{0}, \forall \mathbf{z} \in \mathcal{K} \right\},
\end{equation}
and it is self-dual if $\mathcal{K} = \mathcal{K}^*$.
The cone is pointed if $\mathcal{K} \cap (-\mathcal{K}) = \{ \mathbf{0} \}$.
A cone is proper if it is closed, convex, pointed, and with a non-empty interior.
If $\mathcal{K}$ is proper, then its dual $\mathcal{K}^*$ is proper too.
$\mathcal{K}$ induces a partial order on $\mathbb{R}^k$:
\begin{equation}
    \mathbf{x} \succcurlyeq_{\mathcal{K}} \mathbf{y} \iff \mathbf{x} - \mathbf{y} \in \mathcal{K},
\end{equation}
which allows us to define a conic inequality as
\begin{equation}
    \mathbf{A}\mathbf{x} \succcurlyeq_{\mathcal{K}} \mathbf{b},
\end{equation}
where $\mathbf{A} \in \mathbb{R}^{m \times k}$, $\mathbf{b} \in \mathbb{R}^m$, and $\mathcal{K}$ a cone.

When using a cone that represents the cartesian product of others, i.e., $\mathcal{K} = \mathcal{K}_{n_1} \times \cdots \times \mathcal{K}_{n_r}$ with each cone $\mathcal{K}_{n_i} \subseteq \mathbb{R}^{n_i}$, its corresponding vectors and matrices are partitioned conformally, i.e.,
\begin{equation}
    \begin{aligned}
        & \mathbf{x} = (\mathbf{x_1}; \dots ; \mathbf{x_r}) & \textnormal{ where } \mathbf{x_i} \in \mathbb{R}^{n_i}, \\
        & \mathbf{y} = (\mathbf{y_1}; \dots ; \mathbf{y_r}) & \textnormal{ where } \mathbf{y_i} \in \mathbb{R}^{n_i}, \\
        & \mathbf{c} = (\mathbf{c_1}; \dots ; \mathbf{c_r}) & \textnormal{ where } \mathbf{c_i} \in \mathbb{R}^{n_i}, \\
        & \mathbf{A} = (\mathbf{A}_1; \dots ; \mathbf{A}_r) & \textnormal{ where } \mathbf{A} \in \mathbb{R}^{m \times n_i}.
    \end{aligned}
\end{equation}
Furthermore, if each cone $\mathcal{K}_{n_i} \subseteq \mathbb{R}^{n_i}$ is proper, then $\mathcal{K}$ is proper too.

A Conic Programming (CP) problem is then defined as:
\begin{equation}
	\label{prob:CP}
	\tag{CP}
	\begin{aligned}
  &\min_{\mathbf{x}} &&\mathbf{c}^\top \mathbf{x}\\
	& \textnormal{s.t. } &&\mathbf{A}\mathbf{x} = \mathbf{b},\\
	& &&\mathbf{x} \in \mathcal{K} \subseteq \mathbb{R}^k.
	\end{aligned}
\end{equation}

Examples of proper cones are:
\begin{itemize}
    \item The nonnegative orthant
    \begin{equation}
        \mathbb{R}_+^k = \left \{ \mathbf{z} \in \mathbb{R}^k: \mathbf{z} \geq \mathbf{0} \right \}.
    \end{equation}
    \item The positive semi-definite cone
    \begin{equation}
        \mathbb{S}_+^k = \left \{ Z \in \mathbb{R}^{k \times k}: Z = Z^T, \lambda_{min}(Z) \geq 0 \right \},
    \end{equation}
    where $\lambda_{min}(Z)$ denotes the smallest eigenvalue of $Z$.
    \item The second-order cone, Euclidean norm cone, or Lorentz cone
    \begin{equation}
        \mathcal{Q}^k = \left \{ \mathbf{z} \in \mathbb{R}^{k}: z_1 \geq \sqrt{\sum_{i=2}^k z_i^2} \right \}.
    \end{equation}
    \item The exponential cone~\citep{chares2009cones}
    \begin{equation}
    \begin{aligned}
        \mathcal{K}_{exp} &= \textnormal{cl} \left \{ (z_1,z_2,z_3) \in \mathbb{R}^3: z_1 \geq z_2 e^{z_3 / z_2}, z_1 \geq 0, z_2 > 0 \right \} \\ &= \left \{ (z_1,z_2,z_3) \in \mathbb{R}^3: z_1 \geq z_2 e^{z_3 / z_2}, z_1 \geq 0, z_2 > 0 \right \} \bigcup \mathbb{R}_+ \times \{ \mathbf{0} \} \times (-\mathbb{R}_+)  \\
        &= \left \{ (z_1,z_2,z_3) \in \mathbb{R}^3: z_1 \geq z_2 e^{z_3 / z_2}, z_2 \geq 0 \right \}.
    \end{aligned}
    \end{equation}
\end{itemize}
Of these cones, the only one not being self-dual or symmetric is the exponential cone.

Other cones that are useful in practice are
\begin{itemize}
    \item The rotated second-order cone or Euclidean norm-squared cone
    \begin{equation}
        \mathcal{Q}_r^k = \left \{ \mathbf{z} \in \mathbb{R}^{k}: 2z_1z_2 \geq \sqrt{\sum_{i=3}^k z_i^2}, z_1, z_2 \geq 0 \right \},
    \end{equation}
    This cone can be written as a rotation of the second-order cone, i.e., $\mathbf{z} \in \mathcal{Q}^k \iff R_k\mathbf{z} \in \mathcal{Q}_r^k$ with
    $R_k :=   \begin{bmatrix}
    \sqrt{2}/2 & \sqrt{2}/2 & 0 \\
    \sqrt{2}/2 &  \sqrt{2}/2 & 0 \\
    0 & 0 & I_{k-2}
  \end{bmatrix}$~\citep{aps2018mosek}.%, or by a linear transformation of the second-order cone, i.e., $\mathcal{Q}_r^k = \left \{ \mathbf{z} \in \mathbb{R}^{k}: (\frac{x_1+x_2}{\sqrt{2}},\frac{z_1-z_2}{\sqrt{2}},\dots,z_k) \in \mathcal{Q}^{k} \right \}$.
    \item The power cone, with $l < k, \sum_{i \in \llbracket l \rrbracket} \alpha_i = 1, \alpha_{i \in \llbracket l \rrbracket} > 0$,
    \begin{equation}
        \mathcal{P}_k^{\alpha_1, \dots, \alpha_l} = \left \{ \mathbf{z} \in \mathbb{R}^k: \prod_{i = 1}^l z_i^{\alpha_i} \geq \sqrt{\sum_{i = l+1}^k z_i^2}, \quad z_i \geq 0 \quad i \in \llbracket l \rrbracket \right \}.
    \end{equation}
    
    This cone can be decomposed using a second-order cone and  $l-1$ three-dimensional power cones
    \begin{equation}
        \mathcal{P}_3^{\alpha} = \left \{ (z_1,z_2,z_3) \in \mathbb{R}^3: z_1^\alpha z_2^{1-\alpha} \geq \lvert z_3 \rvert, \quad z_1, z_2 \geq 0 \right \},
    \end{equation}
    through $l - 1$ additional variables $(u,v_{1},\dots,v_{l-2})$,
    \begin{equation}
        \mathbf{z} \in \mathcal{P}_k^{\alpha_1, \dots, \alpha_l} \iff 
        \begin{cases}
        (u,z_{l+1},\dots,z_{k}) \in \mathcal{Q}^{k-l+1},\\
        (z_1,v_1,u) \in \mathcal{P}_3^{\alpha_1}, \\
        (z_i,v_i,v_{i-1}) \in \mathcal{P}_3^{\bar{\alpha}_i}, \quad i=2,\dots,l-1,  \\
        (z_{l-1},z_l,v_{l-2}) \in \mathcal{P}_3^{\bar{\alpha}_{l-1}},
        \end{cases}
    \end{equation}
    where $\bar{\alpha}_i=\alpha_i/(\alpha_i+\cdots+\alpha_l)$ for $i=2,\dots,l-1$~\citep{aps2018mosek}.
    $\mathcal{P}_3^{\alpha}$ can be represented using linear and exponential cone constraints, i.e., $\lim_{\alpha \to 0} (z_1,z_2,z_2+\alpha z_3) \in \mathcal{P}_3^{\alpha} = (z_1,z_2,z_3) \in \mathcal{K}_{exp}$
\end{itemize}

Most, if not all, applications-related convex optimization problems can be represented by conic extended formulations using these standard cones~\citep{aps2018mosek}, i.e., in problem~\ref{prob:CP}, the cone $\mathcal{K}$ is a product $\mathcal{K}_1 \times \cdots \times \mathcal{K}_r$, where each $\mathcal{K}_i$ is one of the recognized cones mentioned above.
Equivalent conic formulations for more exotic convex sets using unique cones can be formulated with potential advantages for improved solution performance~\citep{coey2020towards,coey2022conic}.

As mentioned in the introduction, an alternative to a convex optimization problem's algebraic description as in problem~\ref{prob:MINLP} is the following Mixed-Integer Conic Programming (MICP) problem:
\begin{equation}
	\label{prob:MICP}
	\tag{MICP}
	\begin{aligned}
  &\min_{\mathbf{z,y}} && \mathbf{c}^T\mathbf{z}\\
	& \textnormal{s.t. } 
	&&\mathbf{A}\mathbf{z} +\mathbf{B}\mathbf{y} = \mathbf{b}, \\
	& && \mathbf{y}^l \leq \mathbf{y} \leq \mathbf{y}^u,\\
	& &&\mathbf{z} \in \mathcal{K} \subseteq \mathbb{R}^k,\ 
	\mathbf{y} \in \mathbb{Z}^{n_y},
	\end{aligned}
\end{equation}
where $\mathcal{K}$ is a closed convex cone.

Without loss of generality, integer variables need not be restricted to cones, given that corresponding continuous variables can be introduced via equality constraints. 
Notice that for an arbitrary convex function $f: \mathbb{R}^k \to \mathbb{R} \cup \{ \infty \}$, one can define a closed convex cone using its recession,
\begin{equation}
\label{eq:fascone}
    \mathcal{K}_f = \textnormal{cl} \{(\mathbf{z},\lambda,t): \lambda f(\mathbf{z} / \lambda) = \tilde{f}(\mathbf{z},\lambda) \leq t, \lambda > 0 \},
\end{equation}
where the function $\tilde{f}(\mathbf{z},\lambda)$ is the perspective function of function $f(\mathbf{z})$, and whose algebraic representation is a central piece of this work.
Closed convex cones can also be defined as the recession of convex sets.
On the other hand, a conic constraint can be equivalent to a convex inequality,
\begin{equation}
    \mathbf{A}\mathbf{x} \succcurlyeq_{\mathcal{K}} \mathbf{b} \iff \mathbf{g}(\mathbf{x}) \leq \mathbf{0}.
\end{equation}
Although certain cones can be non-smooth, e.g., SOC cones, these can be reformulated using appropriately chosen smooth convex functions $\mathbf{g}(\mathbf{x})$~\citep{benson2003solving,ccezik2005cuts}.

We can therefore reformulate problem~\ref{prob:MINLP} in the following parsimonious manner~\citep{lubin2016extended}:
\begin{equation}
	\label{prob:MINLPcone}
	\tag{MINLP-Cone}
	\begin{aligned}
  &\min_{\substack{\mathbf{x,y},s_{\llbracket J \rrbracket} \\ \mathbf{x}_f,\mathbf{y}_f,t_f, \\
  \mathbf{x}_{\llbracket J \rrbracket},\mathbf{y}_{\llbracket J \rrbracket}}} && t_f\\
	& \textnormal{s.t. } &&\mathbf{x} = \mathbf{x_f}, \mathbf{y} = \mathbf{y_f},\\
	& && ((\mathbf{x_f};\mathbf{y_f}),1,t_f) \in \mathcal{K}_f,\\
	& && \mathbf{x} = \mathbf{x_j}, \mathbf{y} = \mathbf{y_j}, \\
	& && \left((\mathbf{x_j};\mathbf{y_j}),1,s_j \right) \in \mathcal{K}_{g_j}, s_j \in \mathbb{R}_+, j \in \llbracket J \rrbracket, \\
	& && \mathbf{y}^l \leq \mathbf{y} \leq \mathbf{y}^u,\\
	& &&\mathbf{x}\in \mathbb{R}_+^{n_x},\ 
	\mathbf{y} \in \mathbb{Z}^{n_y},\\
	\end{aligned}
\end{equation}
where copies of the original variables $\mathbf{x}$ and $\mathbf{y}$ are introduced for the objective function and each constraint, $\mathbf{x_f},\mathbf{y_f},\mathbf{x_j},\mathbf{y_j}, j \in \llbracket J \rrbracket$, such that each belongs to the recession cone of each constraint defined as in~\eqref{eq:fascone}.
Each conic set requires the introduction of an epigraph variable $t$ and a recession variable $\lambda$.
The epigraph variable from the objective function, $t_f$, is used in the new objective, and the ones corresponding to the constraints are set as nonnegative slack variables $s_j$.
The recession variables $\lambda$ in~\eqref{eq:fascone} are fixed to one in all cases.

Notice that problem~\ref{prob:MINLPcone} is in \ref{prob:MICP} form with $\mathcal{K} = \mathbb{R}_+^{n_x + J} \times \mathcal{K}_f \times  \mathcal{K}_{g_1} \times \cdots \times \mathcal{K}_{g_J}$.
As mentioned above, the case when $\mathcal{K} = \mathcal{K}_1 \times \cdots \times \mathcal{K}_r$ where each $\mathcal{K}_i$ is a recognized cone is more useful from practical purposes.
Lubin et al.~\citep{lubin2016extended} showed that all the convex MINLP instances at the benchmark library MINLPLib~\citep{bussieck2003minlplib} could be represented with nonnegative, second-order, and exponential cones.

\subsection{Perspective function}
\label{sec:background.persp}

For a convex function $h(\mathbf{x}): \mathbb{R}^n \to \mathbb{R} \cup \{ \infty \}$ its perspective function  $\tilde{h}(\mathbf{x},\lambda): \mathbb{R}^{n+1} \to \mathbb{R} \cup \{ \infty \}$ is defined as
\begin{equation}
\label{eq:perspective}
    \tilde{h}(\mathbf{x},\lambda) = 
    \begin{cases}
    \lambda h(\mathbf{x}/\lambda) & \textnormal{if } \lambda > 0 \\
    \infty & \textnormal{otherwise}
    \end{cases}
\end{equation}

The perspective of a convex function is convex but not closed.
Hence, consider the closure of the perspective function $(\textnormal{cl}~\tilde{h})(\mathbf{x},\lambda)$ defined as
\begin{equation}
\label{eq:clperspective}
    \left( \textnormal{cl }\tilde{h} \right) (\mathbf{x},\lambda) = 
    \begin{cases}
    \lambda h(\mathbf{x}/\lambda) & \textnormal{if } \lambda > 0 \\
    h'_\infty (\mathbf{x}) & \textnormal{if } \lambda = 0 \\
    \infty & \textnormal{otherwise}
    \end{cases},
\end{equation}

where $h'_\infty (\mathbf{x})$ is the recession function of function $h(\mathbf{x})$\citep[Section B Proposition 2.2.2]{hiriart2004fundamentals}, and which in general does not have a closed-form.

The closure of the perspective function of a convex function is relevant for convex MINLP on two ends.
On the one hand, it appears when describing the closure of the convex hull of disjunctive sets.
On the other hand, as seen above, it can be used to define closed convex cones $\mathcal{K}$ that determine the feasible region of conic programs.
Relying on the amenable properties of convex cones, conic programs can be addressed with specialized algorithms, allowing for more efficient solution methods.

The closure of the perspective function presents a challenge when implementing it for nonlinear optimization models, given that it is not defined at $\lambda=0$.
As seen below, modeling this function becomes necessary when writing the convex hull of the union of convex sets.
Several authors have addressed this difficulty in the literature through $\varepsilon$-approximations.
The first proposal was made by Lee and Grossmann~\citep{lee2000new}, where
\begin{equation}
\label{eq:lee_persp}
    \left( \textnormal{cl } \tilde{h} \right) (\mathbf{x},\lambda) \approx (\lambda + \varepsilon) h \left( \frac{\mathbf{x}}{\lambda + \varepsilon} \right).
\end{equation}
This approximation is exact when $\varepsilon \to 0$. However, it requires values for $\varepsilon$, which are small enough to become numerically challenging when implemented in a solution algorithm.

Furman et al.~\citep{furman2020computationally} propose another approximation for the perspective function such that
\begin{equation}
\label{eq:furman_persp}
    \left( \textnormal{cl } \tilde{h} \right) (\mathbf{x},\lambda) \approx ((1-\varepsilon)\lambda + \varepsilon)h \left( \frac{\mathbf{x}}{(1-\varepsilon)\lambda + \varepsilon} \right) - \varepsilon h(0) (1-\lambda),
\end{equation}
which is exact for values of $\lambda = 0$ and $\lambda = 1$, is convex for $h(\mathbf{x})$ convex, and is exact when $\varepsilon \to 0$ as long as $h(\mathbf{0})$ is defined.
Using this approximation in the set describing the system of equations of the closed convex hull of a disjunctive set also has properties beneficial for mathematical programming.

This approximation is used in software implementations when reformulating a disjunctive set using its hull relaxation~\citep{vecchietti1999logmip,chenpyomo}.
Notice that even with its desirable properties, the approximation introduces some error for values $\varepsilon > 0$; hence, it is desirable to circumvent its usage.
As shown in~\citep{gunluk2012perspective} and the Section~\ref{sec:results}, using a conic constraint to model the perspective function allows for a more efficient solution of convex MINLP problems.
\subsection{Disjunctive Programming}
\label{sec:background.disjunct}

Optimization over disjunctive sets is denoted as Disjunctive Programming~\citep{balas1979disjunctive,balas2018disjunctive}.
The system of inequalities gives a disjunctive set joined by logical operators of conjunction ($\wedge$, ``and'') and disjunction ($\vee$, ``or'').
These sets are non-convex and usually represent the union of convex sets.
The main reference on Disjunctive Programming is the book by Balas~\citep{balas2018disjunctive}.

Consider the following disjunctive set 
\begin{equation}
\label{eq:disjset}
    \mathcal{C} = \left\{ \mathbf{x} \in \mathbb{R}^n : \mathbf{x} \in \bigvee_{i \in I} \mathcal{C}_i \right\} = \bigcup_{i \in I} \left\{ \mathbf{x} \in \mathbb{R}^n : \mathbf{x} \in \mathcal{C}_i\right\},
\end{equation}
where $\lvert I \rvert$ is finite.
Each set defined as $\mathcal{C}_i := \{ \mathbf{x} \in \mathbb{R}^n \mid \mathbf{h}_i(\mathbf{x}) \leq \mathbf{0} \}$ is a convex, bounded, and nonempty set defined by a vector-valued function $\mathbf{h}_i: \mathbb{R}^n \to \left( \mathbb{R} \cup \{ \infty \} \right)^{J_i}$.
Notice that is it sufficient for $\mathcal{C}_i$ to be convex that each component of $\mathbf{h}_i$, $h_{i \llbracket J_i \rrbracket }$, is a proper closed convex function, although it is not a necessary condition.
A proper closed convex function is one whose epigraph is a nonempty closed convex set~\citep{parikh2014proximal}.

Ceria and Soares~\citep{ceria1999convex} characterize the closure of the convex hull of $\mathcal{C}$, $\textnormal{cl conv}(\mathcal{C})$, with the following result.

\begin{theorem}~\citep{ceria1999convex}
\label{th:chull}
Let $\mathcal{C}_i = \{ \mathbf{x} \in \mathbb{R}^n \mid \mathbf{h}_i(\mathbf{x}) \leq \mathbf{0} \} \neq \emptyset$, assume that each component of $\mathbf{h}_i$, $h_{i\llbracket J_i \rrbracket}$, is a proper closed convex function, and let

\begin{equation}
\label{eq:chull}
\mathcal{H} = \left\{
\begin{aligned}
    &\mathbf{x} = \sum_{i \in I}\mathbf{v}_{i}, \\ 
    &\sum_{i \in I}\lambda_{i} = 1, \\
    &\left( \textnormal{cl } \tilde{\mathbf{h}}_i \right) (\mathbf{v}_{i},\lambda_{i})\leq 0 ,  &i \in I, \\
    &\mathbf{v}_{i} \in \mathbb{R}^n ,  &i \in I, \\
    &\lambda_{i} \in \mathbb{R}_+ ,  &i \in I
\end{aligned}
\right \}.
\end{equation}

Then $\textnormal{cl conv}(\bigcup_{i \in I} \mathcal{C}_i) = \textnormal{proj}_{\mathbf{x}}(\mathcal{H})$.
\end{theorem}

\begin{proof}
See~\citep[Theorem 1]{ceria1999convex} and~\citep[Theorem 1]{bonami2015mathematical}.
\end{proof}

Theorem~\ref{th:chull} provides a description of $\textnormal{cl conv}(\mathcal{C})$ in a higher dimensional space, an \textit{extended formulation}.
This Theorem generalizes the result by~\citep{balas1979disjunctive,balas1985disjunctive,balas1998disjunctive,balas2018disjunctive} where all the convex sets $\mathcal{C}_i$ are polyhedral.
Even though the extended formulations induce growth in the size of the optimization problem, some of them have shown to be amenable for MINLP solution algorithms~\citep{tawarmalani2005polyhedral,hijazi2014outer,lubin2016extended,kronqvist2018reformulations}.

A similar formulation was derived by Stubbs and Mehrotra~\citep{stubbs1999branch} in the context of a Branch-and-cut method for Mixed-binary convex programs.
These authors notice that the extended formulation might not be computationally practical; hence, they derive linear inequalities or cuts from this formulation to be later integrated into the solution procedure.
Similar ideas have been explored in the literature~\citep{frangioni2006perspective}.
In particular cases, the dimension of the extended formulation can be reduced to the original size of the problem, e.g., when there are only two terms in the disjunction, i.e., $\lvert I \rvert =2$, and one of the convex sets $\mathcal{C}_i$ is a point~\citep{gunluk2012perspective}.
A description in the original space of variables has also been given for the case when one set $\mathcal{C}_1$ is a box and the constraints defining the other $\mathcal{C}_2$ is determined by the same bounds as the box and nonlinear constraints being isotone~\citep{hijazi2012mixed}.
This has been extended even further by Bonami et al.~\citep{bonami2015mathematical} with complementary disjunctions.
In other words, the activation of one disjunction implies that the other one is deactivated, in the case that the functions that define each set $\mathbf{h}_{\{1,2\}}$ are isotone and share the same indices on which they are non-decreasing.
The last two cases present the formulation in the original space of variables by paying a prize of exponentially many constraints required to represent $\textnormal{cl conv}(\mathcal{C})$.

In the case that $\mathcal{C}_i$ is compact, its recession cone is the origin, i.e., $\mathcal{C}_{i\infty}= \{ \mathbf{x} \in \mathbb{R}^n \mid \mathbf{h}'_{i\infty}(\mathbf{x}) \leq \mathbf{0} \} = \{ \mathbf{0} \}$~\citep[Section A, Proposition 2.2.3]{hiriart2004fundamentals}.
This fact, together with~\eqref{eq:clperspective} and Theorem~\ref{th:chull}, forces that for a compact $\mathcal{C}_i$, a value of $\lambda_i=0$ implies $\mathbf{v}_i=0$.
This fact has been used to propose mixed-integer programming formulations for expressing the disjunctive choice between convex sets by setting the interpolation variables to be binary $\lambda_i \in \{0,1\}, i \in I$~\citep{jeroslow1987representability,lee2000new}, i.e.,
\begin{equation}
\label{eq:dchull}
\mathcal{H}_{\{0,1\}} = \left\{
\begin{aligned}
    &\mathbf{x} = \sum_{i \in I}\mathbf{v}_{i}, \\ 
    &\sum_{i \in I}\lambda_{i} = 1, \\
    &\left( \textnormal{cl } \tilde{\mathbf{h}}_i \right) (\mathbf{v}_{i},\lambda_{i})\leq 0 ,  &i \in I, \\
    &\mathbf{v}_{i} \in \mathbb{R}^n ,  &i \in I, \\
    &\lambda_{i} \in \{ 0, 1 \} ,  &i \in I
\end{aligned}
\right \}.
\end{equation}

An interesting observation is that using the approximation of the closure of the perspective function from Furman et al.~\citep{furman2020computationally}, for any value of $\varepsilon \in (0,1)$, $\textnormal{proj}_\mathbf{x}(\mathcal{H}_{\{0,1\}})=\mathcal{C}$ when $\mathbf{h}_i(\mathbf{0})$ is defined $\forall i \in I$ and
\begin{equation}
    \label{eq:zeroproj}
    \left \{ \mathbf{x} \in \mathbb{R}^n : \mathbf{h}_i(\mathbf{x}) - \mathbf{h}_i(\mathbf{0}) \leq \mathbf{0}   \right \} = \{ \mathbf{0} \}, \forall i \in I
\end{equation}
see~\citep[Proposition 1]{furman2020computationally}.

The condition on~\eqref{eq:zeroproj} is required to ensure that if $\lambda_i=1$, then $\mathbf{v}_{i'}=0,  \forall i' \in I \setminus \{ i \}$.
This condition is not valid in general for a disjunctive set $\mathcal{C}$, but it is sufficient to have a bounded range on $\mathbf{x} \in \mathcal{C}_i,  i \in I$.
Moreover, when these conditions are satisfied, $\mathcal{C} \subseteq \textnormal{proj}_\mathbf{x}(\mathcal{H})$ using the approximation in~\eqref{eq:furman_persp} for $\varepsilon \in (0,1)$, with $\textnormal{cl conv} (\mathcal{C}) = \textnormal{proj}_\mathbf{x}(\mathcal{H})$ in the limit when $\varepsilon \to 0$~\citep[Proposition 3]{furman2020computationally}.

The problem formulation~\ref{prob:hr} is derived by replacing each disjunction with set $\mathcal{H}_{\{0,1\}}$~\eqref{eq:dchull}.
Notice that to guarantee the validity of the formulation, the condition on~\eqref{eq:zeroproj} is enforced implicitly by having the bounds over $\mathbf{x}$ included in each disjunct, leading to constraint $\mathbf{x}^{l}y_{ik} \leq \mathbf{v}_{ik} \leq \mathbf{x}^{u}y_{ik}$.

\section{Detailed computational results}

\begin{sidewaystable}[htbp]
\caption{Results for Quadratic GDPs using different mixed-integer reformulations and solvers. The least time and fewest nodes results for each instance within a reformulation are italicized. The best results overall are bolded.}
\label{tab:quadratic_solvers}
\resizebox{\textwidth}{!}{%
\input{tables/quadratic_solvers}
}
\end{sidewaystable}

\begin{sidewaystable}[htbp]
\caption{Results for Exponential GDPs using the Big-M reformulation and different solvers. The least time and fewest nodes results for each instance within a reformulation are italicized. The best results overall are bolded.}
\label{tab:exp_solvers_bm}
\resizebox{\textwidth}{!}{%
\input{tables/exp_solvers_bigm}
}
\end{sidewaystable}

\begin{sidewaystable}[htbp]
\caption{Results for Exponential GDPs using the Hull reformulation and different solvers. The least time and fewest nodes results for each instance within a reformulation are italicized. The best results overall are bolded.}
\label{tab:exp_solvers_hr}
\resizebox{\textwidth}{!}{%
\input{tables/exp_solvers_hr}
}
\end{sidewaystable}

\end{document}

%% file: tables/reformulation.tex
\rowcolors{1}{}{lightgray}
\begin{tabular}{lllll}
\toprule
\multicolumn{1}{c}{$\mathbf{h}(\mathbf{z}) \leq \mathbf{0}$} &
  \multicolumn{1}{c}{$\mathbf{z} \in \mathcal{K}_{\mathbf{h}}$} &
  \multicolumn{1}{c}{$\tilde{\mathbf{h}}(\mathbf{z},y) \leq \mathbf{0}$} &
  \multicolumn{1}{c}{$(\mathbf{z},y) \in \mathcal{K}_{\tilde{\mathbf{h}}} = (y\frac{\mathbf{z}}{y}) \in \mathcal{K}_{\mathbf{h}}$} &
  \multicolumn{1}{c}{Notes} \\ \midrule
$A\mathbf{x} + \mathbf{b} \geq 0$ &
  $(A\mathbf{x} + \mathbf{b}) \in \mathbb{R}_+^m$ &
  $A\mathbf{x} + \mathbf{b}y \geq 0$ &
  $(A\mathbf{x} + \mathbf{b}y) \in \mathbb{R}_+^m$ &
  Union of polyhedra~\citep{balas1979disjunctive,balas2018disjunctive} \\
$x^2 - t \leq 0$ &
  $(0.5,t,x) \in \mathcal{Q}_r^3$ &
  $x^2 - ty \leq 0$ &
  $(0.5y,t,x) \in \mathcal{Q}_r^3$ &
   \\
$\mathbf{x}^\top\mathbf{x} - t \leq 0$ &
  $(0.5,t,\mathbf{x}) \in \mathcal{Q}_r^{n+2}$ &
  $\mathbf{x}^\top\mathbf{x} - ty \leq 0$ &
  $(0.5y,t,\mathbf{x}) \in \mathcal{Q}_r^{n+2}$ &
   \\
$1/x - t \leq 0, x > 0$ &
  $(x,t,\sqrt{2}) \in \mathcal{Q}_r^3$ &
  $y^2 - tx \leq 0, x > 0$ &
  $(x,t,\sqrt{2}y) \in \mathcal{Q}_r^3$ &
   \\
$\lvert x\rvert^p - t \leq 0, p > 1$ &
  $(t,1,x) \in \mathcal{P}_3^{1/p,1-1/p}$ &
  $\lvert x\rvert^p/y^{p-1} - t \leq 0, p > 1$ &
  $(t,y,x) \in \mathcal{P}_3^{1/p}$ &
   \\
$1/x^p - t \leq 0, x > 0, p > 1$ &
  $(t,x,1) \in \mathcal{P}_3^{1/(1+p),p/(1+p)}$ &
  $y^{1+p}/x^p - t \leq 0, x > 0, p > 1$ &
  $(t,x,y) \in \mathcal{P}_3^{1/(1+p)}$ &
   \\
$x^{a/b} - t \leq 0, x > 0, a \geq b > 0, a,b \in \mathbb{Z} $ &
  $(t,1,x) \in \mathcal{P}_3^{b/a,1-b/a}$ &
  $x^{a} - t^b y^{a-b} \leq 0, x > 0, a \geq b > 0, a,b \in \mathbb{Z} $ &
  $(t,y,x) \in \mathcal{P}_3^{b/a}$ &
  \begin{tabular}[c]{@{}l@{}}Cut strengthening for rational power\\ constraints with indicators\cite{akturk2009strong}\end{tabular} \\
$e^x - t \leq 0$ &
  $(t,1,x) \in \mathcal{K}_{exp}$ &
  $ye^{x/y} - t \leq 0$ &
  $(t,y,x) \in \mathcal{K}_{exp}$ &
   \\
$t- \log(x) \leq 0$ &
  $(x,1,t) \in \mathcal{K}_{exp}$ &
  $t- y\log(x/y) \leq 0$ &
  $(x,y,t) \in \mathcal{K}_{exp}$ &
   \\
$1/\log(x) - t \leq 0, x > 1$ &
  \begin{tabular}[c]{@{}l@{}}$(u,t,\sqrt{2}) \in \mathcal{Q}_r^3, $\\ $(x,1,u) \in \mathcal{K}_{exp}$\end{tabular} &
  $y/\log(x/y) - t \leq 0, x > 1$ &
  \begin{tabular}[c]{@{}l@{}}$(u,t,\sqrt{2}y) \in \mathcal{Q}_r^3, $\\ $(x,y,u) \in \mathcal{K}_{exp}$\end{tabular} &
   \\
$xe^x - t \leq 0, x \geq 0$ &
  \begin{tabular}[c]{@{}l@{}}$(0.5,u,x) \in \mathcal{Q}_r^3, $\\ $(t,x,u) \in \mathcal{K}_{exp}$\end{tabular} &
  $xe^{x/y} - t \leq 0, x \geq 0$ &
  \begin{tabular}[c]{@{}l@{}}$(0.5y,u,x) \in \mathcal{Q}_r^3, $\\ $(t,x,u) \in \mathcal{K}_{exp}$\end{tabular} &
   \\
$a_1^{x_1} \cdots a_n^{x_n} -t \leq 0, a_1 > 0$ &
  $(t,1,\sum_{i\in \llbracket n \rrbracket}x_i \log a_i) \in \mathcal{K}_{exp}$ &
  $ya_1^{x_1/y} \cdots a_n^{x_n/y} -t \leq 0, a_1 > 0$ &
  $(t,y,\sum_{i\in \llbracket n \rrbracket}x_i \log a_i) \in \mathcal{K}_{exp}$ &
   \\
$\log(1+e^x) - t \leq 0$ &
  \begin{tabular}[c]{@{}l@{}}$(u,1,x-t) \in \mathcal{K}_{exp}, $\\ $(v,1,-t) \in \mathcal{K}_{exp}, $\\ $u+v \leq 1$\end{tabular} &
  $y\log(1+e^{x/y}) - t \leq 0$ &
  \begin{tabular}[c]{@{}l@{}}$(u,y,x-t) \in \mathcal{K}_{exp}, $\\ $(v,y,-t) \in \mathcal{K}_{exp}, $\\ $u+v \leq y$\end{tabular} &
   \\
$-\log(1/(1+e^{-\theta^\top \mathbf{x}})) - t \leq 0$ &
  \begin{tabular}[c]{@{}l@{}}$(u,1,-\theta^\top \mathbf{x}-t) \in \mathcal{K}_{exp}, $\\ $(v,1,-t) \in \mathcal{K}_{exp}, $\\ $u+v \leq 1$\end{tabular} &
  $-y\log(1/(1+e^{-\theta^\top \mathbf{x}/y})) - t \leq 0$ &
  \begin{tabular}[c]{@{}l@{}}$(u,y,-\theta^\top \mathbf{x}-t) \in \mathcal{K}_{exp},$\\ $(v,y,-t) \in \mathcal{K}_{exp},$\\ $ u+v \leq y$\end{tabular} &
  Logistic cost function \\
$x\log(x) + t \leq 0$ &
  $(1,x,t) \in \mathcal{K}_{exp}$ &
  $x\log(x/y) + t \leq 0$ &
  $(y,x,t) \in \mathcal{K}_{exp}$ &
  Entropy and relative entropy~\cite{chandrasekaran2017relative} \\
$\log(1+1/x) - t \leq 0, x > 0$ &
  \begin{tabular}[c]{@{}l@{}}$(x+1,u,\sqrt{2}) \in \mathcal{Q}_r^3, $\\ $(1-u,1,t) \in \mathcal{K}_{exp}$\end{tabular} &
  $y\log(1+y/x) - t \leq 0, x, y > 0$ &
  \begin{tabular}[c]{@{}l@{}}$(x,y+x,u) \in \mathcal{K}_{exp}, $\\ $(x+y,x,v) \in \mathcal{K}_{exp}, $\\ $t+u+v=0$\end{tabular} &
   \\
$\lVert A\mathbf{x} + \mathbf{b}\rVert_2 - \mathbf{c}^\top \mathbf{x} + d \leq 0$ &
  $(\mathbf{c}^\top \mathbf{x} + d,A\mathbf{x} + \mathbf{b}) \in \mathcal{Q}^{m+1}$ &
  $\lVert A\mathbf{x} + \mathbf{b}\rVert_2 - \mathbf{c}^\top \mathbf{x} + dy$ &
  $(\mathbf{c}^\top \mathbf{x} + dy,A\mathbf{x} + \mathbf{b}) \in \mathcal{Q}^{m+1}$ &
  \begin{tabular}[c]{@{}l@{}}Robust constraint with \\ ellipsoidal uncertainty set~\cite{el1997robust}\end{tabular} \\
$\log \left( \sum_{i \in \llbracket n \rrbracket} e^{x_i} \right) - t \leq 0$ &
  \begin{tabular}[c]{@{}l@{}}$(u_i,1,x_i-t)\in \mathcal{K}_{exp}, i \in \llbracket n \rrbracket ,$\\ $\sum_{i \in \llbracket n \rrbracket}u_i \leq 1$\end{tabular} &
  $y\log \left( \sum_{i \in \llbracket n \rrbracket} e^{x_i/y} \right) - t \leq 0$ &
  \begin{tabular}[c]{@{}l@{}}$(u_i,y,x_i-t)\in \mathcal{K}_{exp}, i \in \llbracket n \rrbracket ,$\\ $\sum_{i \in \llbracket n \rrbracket}u_i \leq y$\end{tabular} &
  Log-sum-exp \\
$\lVert \mathbf{x} \rVert_1 - t= \sum_{i \in \llbracket n \rrbracket} \lvert x_{i} \rvert - t \leq 0$ &
  \begin{tabular}[c]{@{}l@{}}$(u_i,x_i)\in \mathcal{Q}^2, i \in \llbracket n \rrbracket ,$\\ $t= \sum_{i \in \llbracket n \rrbracket}u_i$\end{tabular} &
  $\lVert \mathbf{x}\rVert_1 - t= \sum_{i \in \llbracket n \rrbracket}\lvert x_i\rvert - t \leq 0$ &
  \begin{tabular}[c]{@{}l@{}}$(u_i,x_i)\in \mathcal{Q}^2, i \in \llbracket n \rrbracket ,$\\ $t= \sum_{i \in \llbracket n \rrbracket}u_i$\end{tabular} &
  $\ell_1$-norm epigraph \\
$\lVert \mathbf{x}\rVert_2 - t= (\sum_{i \in \llbracket n \rrbracket}x_i^2)^{1/2} - t \leq 0$ &
  $(t,\mathbf{x})\in \mathcal{Q}^{n+1}$ &
  $\lVert \mathbf{x}\rVert_2 - t= (\sum_{i \in \llbracket n \rrbracket}x_i^2)^{1/2} - t \leq 0$ &
  $(t,\mathbf{x})\in \mathcal{Q}^{n+1}$ &
  $\ell_2$-norm epigraph \\
$\lVert \mathbf{x}\rVert_p - t= (\sum_{i \in \llbracket n \rrbracket}x_i^p)^{1/p} - t \leq 0, p>1$ &
  \begin{tabular}[c]{@{}l@{}}$(u_i,t,x_i)\in \mathcal{P}_3^{1/p}, i \in \llbracket n \rrbracket ,$\\ $t= \sum_{i \in \llbracket n \rrbracket}u_i$\end{tabular} &
  $\lVert \mathbf{x}\rVert_p - t= (\sum_{i \in \llbracket n \rrbracket}x_i^p)^{1/p} - t \leq 0, p>1$ &
  \begin{tabular}[c]{@{}l@{}}$(u_i,t,x_i)\in \mathcal{P}_3^{1/p}, i \in \llbracket n \rrbracket ,$\\ $t= \sum_{i \in \llbracket n \rrbracket}u_i$\end{tabular} &
  $\ell_p$-norm epigraph \\
$\lVert \mathbf{x}\rVert_{-1} - t= n(\sum_{i \in \llbracket n \rrbracket}x_i^{-1})^{-1} - t \leq 0, \mathbf{x} > \mathbf{0}$ &
  \begin{tabular}[c]{@{}l@{}}$(u_i,x_i,t)\in \mathcal{Q}_r^3, i \in \llbracket n \rrbracket ,$\\ $nt/2= \sum_{i \in \llbracket n \rrbracket}u_i$\end{tabular} &
  $\lVert \mathbf{x}\rVert_{-1} - t= n(\sum_{i \in \llbracket n \rrbracket}x_i^{-1})^{-1} - t \leq 0, \mathbf{x} \geq \mathbf{0}$ &
  \begin{tabular}[c]{@{}l@{}}$(u_i,x_i,t)\in \mathcal{Q}_r^3, i \in \llbracket n \rrbracket ,$\\ $nt/2= \sum_{i \in \llbracket n \rrbracket}u_i$\end{tabular} &
  \begin{tabular}[c]{@{}l@{}}Harmonic mean \\ $\ell_{-1}$-norm epigraph\end{tabular} \\
$(\prod_{i \in \llbracket n \rrbracket}x_i)^{1/n} - \lvert t\rvert \leq 0, \mathbf{x} > \mathbf{0}$ &
  \begin{tabular}[c]{@{}l@{}}$(u_i,x_i,u_{i+1})\in \mathcal{P}_3^{1-1/i}, i \in \{2, \dots, n \}, $\\ $u_1=x_1, u_{n+1}=t$\end{tabular} &
  $(\prod_{i \in \llbracket n \rrbracket}x_i)^{1/n} - \lvert t\rvert \leq 0, \mathbf{x} > \mathbf{0}$ &
  \begin{tabular}[c]{@{}l@{}}$(u_i,x_i,u_{i+1})\in \mathcal{P}_3^{1-1/i}, i \in \{2, \dots, n \}, $\\ $u_1=x_1, u_{n+1}=t$\end{tabular} &
  Geometric mean \\
\begin{tabular}[c]{@{}l@{}}$\prod_{i \in \llbracket n \rrbracket}x_i^{\alpha_i} - \lvert t\rvert \leq 0, \mathbf{x} > \mathbf{0},$\\ $ \mathbf{\alpha} > \mathbf{0}, \sum_{i \in \llbracket l \rrbracket} \alpha_i = 1$\end{tabular} &
  \begin{tabular}[c]{@{}l@{}}$(u_i,x_i,u_{i+1})\in \mathcal{P}_3^{1-\beta_i}, $\\ $\beta_i = \alpha_i / (\sum_{j \in \llbracket i \rrbracket} \alpha_j), i \in \{2, \dots, l \}, $\\ $u_1=x_1, u_{l+1}=t$\end{tabular} &
  \begin{tabular}[c]{@{}l@{}}$\prod_{i \in \llbracket n \rrbracket}x_i^{\alpha_i} - \lvert t\rvert \leq 0, \mathbf{x} > \mathbf{0},$\\ $ \mathbf{\alpha} > \mathbf{0}, \sum_{i \in \llbracket n \rrbracket} \alpha_i = 1$\end{tabular} &
  \begin{tabular}[c]{@{}l@{}}$(u_i,x_i,u_{i+1})\in \mathcal{P}_3^{1-\beta_i}, $\\ $\beta_i = \alpha_i / (\sum_{j \in \llbracket i \rrbracket} \alpha_j), i \in \{2, \dots, l \}, $\\ $u_1=x_1, u_{l+1}=t$\end{tabular} &
  Weighted Geometric Mean \\ \midrule
\multicolumn{5}{l}{\footnotesize{\begin{tabular}[c]{@{}l@{}}Symbols $a,b,c,d,p$ stand for scalar parameters, $\alpha, \beta, \theta$ are vector parameters, $\mathbf{x} \in \mathbb{R}^n$ and $t \in \mathbb{R}$ are variables, and $u,v$ stand for additional variables added for the conic reformulation. $y \in [0,1]$ is the perspective variable. \\ This table is heavily influenced by the Conic Modeling Cleansheet~\cite{aps2018mosek}\end{tabular}}}
\end{tabular}

%% file: tables/quadratic_solvers.tex
\rowcolors{1}{}{lightgray}
\begin{tabular}{@{}lllllllllllllll@{}}
\toprule
\multicolumn{1}{c}{} & \multicolumn{10}{c}{\textbf{Big-M}} & \multicolumn{4}{c}{} \\ \cmidrule(lr){2-11} 
\multicolumn{1}{c}{} & \multicolumn{2}{c}{BARON} & \multicolumn{2}{c}{CPLEX} & \multicolumn{2}{c}{KNITRO} & \multicolumn{2}{c}{MOSEK-OA} & \multicolumn{2}{c}{MOSEK-IP} & \multicolumn{2}{c}{} & \multicolumn{2}{c}{} \\ \cmidrule(lr){2-3} \cmidrule(lr){4-5} \cmidrule(lr){6-7} \cmidrule(lr){8-9} \cmidrule(lr){10-11}
Instance & Time [s] (Remaining gap \%) & Nodes & Time [s] (Remaining gap \%) & Nodes & Time [s] (Remaining gap \%) & Nodes & Time [s] (Remaining gap \%) & Nodes & Time [s] (Remaining gap \%) & Nodes &  &  &  &  \\ \midrule
\texttt{CLay0203} & 0.49 & \textit{\textbf{9}} & 0.35 & 56 & 0.80 & 175 & 0.61 & 145 & 0.78 & 261 &  &  &  &  \\
\texttt{CLay0204} & 0.73 & \textit{4} & \textit{\textbf{0.58}} & 557 & 7.19 & 1921 & 4.74 & 2589 & 10.88 & 2357 &  &  &  &  \\
\texttt{CLay0205} & 5.78 & \textit{61} & \textit{\textbf{3.96}} & 3972 & 102.76 & 16540 & 27.96 & 15559 & 177.25 & 20115 &  &  &  &  \\
\texttt{CLay0303} & 1.43 & \textit{\textbf{79}} & 0.61 & 203 & 2.22 & 347 & 1.36 & 367 & 1.07 & 271 &  &  &  &  \\
\texttt{CLay0304} & 5.98 & \textit{\textbf{155}} & \textit{\textbf{1.71}} & 842 & 15.33 & 1647 & 15.19 & 3313 & 59.30 & 6865 &  &  &  &  \\
\texttt{CLay0305} & 7.84 & \textit{\textbf{41}} & \textit{\textbf{4.04}} & 4629 & 99.68 & 14878 & 112.42 & 38589 & 139.41 & 12601 &  &  &  &  \\
\texttt{CLay0405} & 7.96 & \textit{\textbf{41}} & \textit{\textbf{4.07}} & 4629 & 97.83 & 14878 & 112.22 & 38589 & 138.11 & 12601 &  &  &  &  \\
\texttt{kClus\_3\_10\_2\_*} & 4.02$\pm$1.68 & \textit{186$\pm$102} & \textit{\textbf{0.54$\pm$0.32}} & 1053$\pm$507 & 1.74$\pm$.70 (.07\%) & 698$\pm$308 & 18.51$\pm$9.41 & 1846$\pm$554 & 2.56$\pm$0.56 & 916$\pm$206 &  & \textit{} &  &  \\
\texttt{kClus\_3\_10\_3\_*} & 9.87$\pm$8.32 & \textit{412$\pm$407} & \textit{\textbf{0.83$\pm$0.28}} & 1997$\pm$961 & 4.10$\pm$2.14 (.05\%) & 1613$\pm$928 & 88.93$\pm$36.88 & 3445$\pm$1093 & 5.17$\pm$2.06 & 1589$\pm$641 &  & \textit{} &  &  \\
\texttt{kClus\_3\_10\_5\_*} & 28.46$\pm$13.80 & \textit{880$\pm$511} & \textit{\textbf{4.36$\pm$3.62}} & 7555$\pm$3218 & 15.37$\pm$5.62   (.04\%) & 5119$\pm$1893 & 558.13$\pm$234.97 & 8354$\pm$2927 & 18.33$\pm$4.96 & 4322$\pm$1189 &  & \textit{} &  &  \\
\texttt{kClus\_3\_20\_2\_*} & 50.86$\pm$25.61 & \textit{1537$\pm$948} & \textit{\textbf{5.92$\pm$3.17}} & 19326$\pm$11194 & 91.70$\pm$52.41   (.11\%) & 20219$\pm$10165 & 1686.86$\pm$714.65 & 66645$\pm$26505 & 65.65$\pm$40.45 & 13118$\pm$8208 &  & \textit{} &  &  \\
\texttt{kClus\_3\_20\_3\_*} & 564.54$\pm$641.67 & \textit{12481$\pm$16401} & \textit{\textbf{85.29$\pm$87.60   (.51\%)}} & 261404$\pm$268823 & 1179.82$\pm$824.21   (12.48\%) & 130637$\pm$71286 & 3600+   (156.09\%) & 61030$\pm$12564 & 812.72$\pm$976.49 & 133093$\pm$161988 &  & \textit{} &  &  \\
\texttt{kClus\_3\_20\_5\_*} & 3302.28$\pm$591.02   (67.71\%) & \textit{27529$\pm$17966} & \textit{\textbf{516.73$\pm$49.54   (42.03\%)}} & 1000000+ & 2950.01$\pm$41.10   (85.39\%) & 200000+ & 3600+   (571.42\%) & 24085$\pm$4234 & 3594.30$\pm$17.97   (35.91\%) & 454679$\pm$8911 &  & \textit{} &  &  \\
\texttt{kClus\_5\_10\_2\_*} & 23.81$\pm$13.96 & \textit{1766$\pm$1040} & \textit{\textbf{5.22$\pm$3.18}} & 13507$\pm$9764 & 31.21$\pm$13.15   (.94\%) & 7382$\pm$3068 & 383.21$\pm$174.41 & 11391$\pm$3369 & 21.33$\pm$6.70 & 4999$\pm$1580 &  & \textit{} &  &  \\
\texttt{kClus\_5\_10\_3\_*} & 65.26$\pm$54.41 & \textit{3768$\pm$2988} & \textit{\textbf{15.00$\pm$16.73}} & 25142$\pm$23951 & 57.23$\pm$25.19   (.36\%) & 12898$\pm$5686 & 1337.58$\pm$623.21 & 17742$\pm$6985 & 51.54$\pm$31.30 & 10204$\pm$6176 &  &  &  &  \\
\texttt{kClus\_5\_10\_5\_*} & 439.59$\pm$331.61 & \textit{16507$\pm$12604} & \textit{\textbf{64.80$\pm$58.28}} & 80939$\pm$42792 & 406.47$\pm$243.50   (.14\%) & 57827$\pm$30849 & 3600+   (430.70\%) & 15037$\pm$2151 & 302.44$\pm$174.68 & 42001$\pm$24008 &  & \textit{} &  &  \\
\texttt{kClus\_5\_20\_2\_*} & 2607.50$\pm$1251.24   (93.15\%) & \textit{86762$\pm$48654} & \textit{\textbf{226.88$\pm$108.77   (9.39\%)}} & 550460$\pm$281643 & 2830.47$\pm$519.31   (560.88\%) & 191837$\pm$24425 & 3600+   (853.71\%) & 47549$\pm$9311 & 1758.49$\pm$915.35   (3.33\%) & 208111$\pm$110067 &  & \textit{} &  &  \\
\texttt{kClus\_5\_20\_3\_*} & 3600+ (509.02\%) & \textit{51192$\pm$12127} & \textit{\textbf{591.02$\pm$277.76   (126.18\%)}} & 1000000+ & 3232.92$\pm$63.71   (514.94\%) & 200000+ & 3600+ ($\infty$) & 19394$\pm$3531 & 3400.14$\pm$600.12   (67.67\%) & 335881$\pm$60762 &  & \textit{} &  &  \\
\texttt{kClus\_5\_20\_5\_*} & 3600+ & 11893$\pm$3233 & \textit{\textbf{795.61$\pm$73.33   (559.21\%)}} & 1000000+ & 3600+   (674.52\%) & 189226$\pm$2349 & 3600+ ($\infty$) & 7832$\pm$1295 & 3600+   (233.82\%) & 255259$\pm$3240 &  &  &  &  \\
\texttt{socp\_random\_10\_10\_10\_*} & 106.84$\pm$81.75 & \textit{3483$\pm$2927} & 482.81$\pm$926.53 & 160209$\pm$293928 & \textit{\textbf{33.98$\pm$12.48}} & 4233$\pm$1625 & 2772.27$\pm$829.39   (4.87\%) & 29369$\pm$8324 & 127.57$\pm$51.60 & 5273$\pm$2222 &  & \textbf{} &  &  \\
\texttt{socp\_random\_10\_10\_5\_*} & \textit{\textbf{8.98$\pm$5.99}} & \textit{325$\pm$357} & 14.25$\pm$4.61 & 3078$\pm$1477 & 12.55$\pm$3.01 & 2233$\pm$581 & 776.76$\pm$540.49 & 22688$\pm$15947 & 141.69$\pm$69.42 & 4601$\pm$2092 &  &  &  &  \\
\texttt{socp\_random\_10\_5\_10\_*} & 12.55$\pm$8.14 & \textit{455$\pm$371} & 17.23$\pm$6.77 & 3749$\pm$1562 & \textit{\textbf{7.94$\pm$5.10}} & 1777$\pm$1175 & 684.04$\pm$352.24 & 9112$\pm$4511 & 17.12$\pm$7.81 & 1734$\pm$736 &  &  &  &  \\
\texttt{socp\_random\_10\_5\_5\_*} & \textit{\textbf{1.36$\pm$.72}} & \textit{35$\pm$36} & 5.73$\pm$1.35 & 494$\pm$311 & 2.05$\pm$.85 & 559$\pm$305 & 52.39$\pm$28.37 & 1912$\pm$865 & 8.37$\pm$3.45 & 836$\pm$347 &  & \textbf{} &  &  \\
\texttt{socp\_random\_2\_2\_2\_*} & 0.03$\pm$0.01 & \textit{1$\pm$0} & 0.02$\pm$.01 & 1$\pm$1 & \textit{0.02$\pm$0.00} & 4$\pm$1 & 0.11$\pm$0.02 & 5$\pm$1 & 0.04$\pm$0.01 & 5$\pm$1 &  &  &  &  \\
\texttt{socp\_random\_5\_10\_10\_*} & 7.38$\pm$3.31 & \textit{305$\pm$202} & 6.18$\pm$1.80 & 1370$\pm$635 & \textit{\textbf{3.35$\pm$1.96}} & 680$\pm$499 & 129.11$\pm$42.08 & 3138$\pm$939 & 10.66$\pm$3.80 & 1034$\pm$403 &  &  &  &  \\
\texttt{socp\_random\_5\_10\_5\_*} & 1.41$\pm$.81 & \textit{56$\pm$58} & 3.95$\pm$1.56 & 378$\pm$221 & \textit{1.42$\pm$.46} & 330$\pm$156 & 34.08$\pm$9.66 & 1526$\pm$523 & 5.19$\pm$2.00 & 548$\pm$251 &  &  &  &  \\
\texttt{socp\_random\_5\_5\_10\_*} & 1.29$\pm$.64 & \textit{39$\pm$32} & 3.95$\pm$.88 & 219$\pm$130 & \textit{\textbf{0.72$\pm$0.20}} & 160$\pm$64 & 17.45$\pm$5.00 & 569$\pm$192 & 1.18$\pm$.43 & 208$\pm$78 &  &  &  &  \\
\texttt{socp\_random\_5\_5\_5\_*} & 0.56$\pm$0.27 & \textit{19$\pm$16} & 1.91$\pm$.66 & 66$\pm$32 & \textit{0.53$\pm$0.14} & 137$\pm$68 & 3.79$\pm$0.92 & 344$\pm$84 & 0.90$\pm$0.24 & 185$\pm$66 &  &  &  &  \\ \bottomrule \\ \\
\toprule
\multicolumn{1}{c}{} & \multicolumn{10}{c}{\textbf{HR-Cone}} & \multicolumn{4}{c}{\textbf{HR $\varepsilon$-approximation}} \\ \cmidrule(lr){2-11} \cmidrule(lr){12-15} 
\multicolumn{1}{c}{} & \multicolumn{2}{c}{BARON} & \multicolumn{2}{c}{CPLEX} & \multicolumn{2}{c}{KNITRO} & \multicolumn{2}{c}{MOSEK-OA} & \multicolumn{2}{c}{MOSEK-IP} & \multicolumn{2}{c}{BARON} & \multicolumn{2}{c}{KNITRO} \\  \cmidrule(lr){2-3} \cmidrule(lr){4-5} \cmidrule(lr){6-7} \cmidrule(lr){8-9} \cmidrule(lr){10-11} \cmidrule(lr){12-13} \cmidrule(lr){14-15}
Instance & Time [s] (Remaining gap \%) & Nodes & Time [s] (Remaining gap \%) & Nodes & Time [s] (Remaining gap \%) & Nodes & Time [s] (Remaining gap \%) & Nodes & Time [s] (Remaining gap \%) & Nodes & Time [s] (Remaining gap \%) & Nodes & Time [s] (Remaining gap \%) & Nodes \\ \midrule
\texttt{CLay0203} & 2.46 & 109 & \textit{\textbf{0.33}} & 127 & 4.27 & 239 & 0.14* (91.4\%) & \textit{69} & 0.64 & 199 & 2.34 & 287 & 4.23 & 179 \\
\texttt{CLay0204} & \textit{2.26} & \textit{\textbf{1}} & \textit{0.68} & 842 & 13.88 & 1775 & 2.00 & 1395 & 5.41 & 1579 & 2.29 & 7 & 18.73 & 1771 \\
\texttt{CLay0205} & 18.03 & \textit{\textbf{45}} & \textit{6.38} & 7552 & 187.79 & 13111 & 32.51 & 16855 & 67.20 & 9433 & 12.38 & 177 & 246.84 & 15902 \\
\texttt{CLay0303} & 4.70 & 251 & \textit{\textbf{0.41}} & \textit{186} & 15.32 & 257 & 0.79 & 329 & 1.76 & 303 & 11.85 & 4297 & 6.27 & 223 \\
\texttt{CLay0304} & 48.16 & 2351 & \textit{1.89} & \textit{1046} & 65.81 & 1713 & 5.42 & 1801 & 49.08 & 6081 & 68.00 & 13185 & 47.28 & 1457 \\
\texttt{CLay0305} & 30.11 & \textit{133} & \textit{13.78} & 10590 & 515.58 & 15955 & 48.19 & 21269 & 97.22 & 10093 & 24.74 & 231 & 295.43 & 12227 \\
\texttt{CLay0405} & 30.10 & \textit{133} & \textit{13.78} & 10590 & 515.51 & 15955 & 48.21 & 21269 & 95.99 & 10093 & 24.88 & 231 & 292.45 & 12227 \\
\texttt{kClus\_3\_10\_2\_*} & 29.22$\pm$15.33 & 525$\pm$483 & \textit{1.31$\pm$0.28} & 936$\pm$290 & 56.26$\pm$17.84 (.13\%) & 756$\pm$337 & 8.17$\pm$2.76 & 1187$\pm$390 & 2.51$\pm$.77 & 551$\pm$177 & 22.75$\pm$11.80 & \textit{357$\pm$278} & 6.02$\pm$1.96 (.07\%) & 789$\pm$276 \\
\texttt{kClus\_3\_10\_3\_*} & 87.10$\pm$48.70 & 1148$\pm$898 & 7.17$\pm$2.24 & 1529$\pm$567 & 65.00$\pm$24.18 (.06\%) & 1630$\pm$809 & 34.90$\pm$15.15 & 2365$\pm$800 & \textit{6.01$\pm$1.79} & 1124$\pm$368 & 58.53$\pm$27.51 & \textit{655$\pm$450} & 15.52$\pm$7.11 (.07\%) & 1731$\pm$916 \\
\texttt{kClus\_3\_10\_5\_*} & 479.73$\pm$194.92 & 4220$\pm$2191 & 45.97$\pm$15.44 & 5135$\pm$1642 & 135.97$\pm$40.76   (.04\%) & 4009$\pm$1351 & 197.68$\pm$59.99 & 6381$\pm$2538 & \textit{22.92$\pm$9.46} & 3146$\pm$1377 & 287.92$\pm$137.92 & \textit{2086$\pm$1451} & 48.30$\pm$12.29   (.04\%) & 4015$\pm$1147 \\
\texttt{kClus\_3\_20\_2\_*} & 821.89$\pm$582.53 & 6773$\pm$5222 & \textit{39.23$\pm$16.94} & 19176$\pm$9424 & 1028.58$\pm$445.35   (.11\%) & 21647$\pm$10121 & 415.52$\pm$204.68 & 28496$\pm$15069 & 90.01$\pm$29.74 & 10441$\pm$3507 & 783.84$\pm$538.69 & \textit{5515$\pm$5066} & 318.99$\pm$125.66   (.11\%) & 20706$\pm$8089 \\
\texttt{kClus\_3\_20\_3\_*} & 3496.43$\pm$254.15   (129.69\%) & 13494$\pm$3443 & 1568.36$\pm$1137.08   (6.31\%) & 163189$\pm$110030 & 3224.30$\pm$590.94   (79.77\%) & 50106$\pm$7991 & 3370.83$\pm$478.28   (39.17\%) & 107636$\pm$18035 & \textit{1088.85$\pm$940.14   (2.56\%)} & 107190$\pm$100179 & 3236.45$\pm$578.69   (129.32\%) & \textit{9690$\pm$3092} & 2540.24$\pm$1190.91   (42.88\%) & 103881$\pm$43238 \\
\texttt{kClus\_3\_20\_5\_*} & 3600+   (452.53\%) & 5760$\pm$1573 & 3600+   (74.97\%) & 123015$\pm$24631 & 3600+   (151.18\%) & 36059$\pm$1622 & 3600+   (226.97\%) & 42462$\pm$9595 & \textit{3600+   (53.20\%)} & 266698$\pm$18613 & 3600+   (471.50\%) & \textit{3495$\pm$1601} & 3600+   (116.23\%) & 101891$\pm$1345 \\
\texttt{kClus\_5\_10\_2\_*} & 224.28$\pm$99.44 & 5250$\pm$2504 & 25.22$\pm$33.11 & 5857$\pm$1821 & 288.74$\pm$104.15   (.84\%) & 8166$\pm$3467 & 94.17$\pm$30.15 & 8665$\pm$3343 & \textit{23.47$\pm$6.75} & 3887$\pm$1179 & 177.28$\pm$84.11 & \textit{3685$\pm$1884} & 74.39$\pm$37.66   (.97\%) & 6781$\pm$3438 \\
\texttt{kClus\_5\_10\_3\_*} & 774.57$\pm$534.53 & 13334$\pm$9416 & 67.89$\pm$35.31 & 12002$\pm$6653 & 530.88$\pm$275.35   (.36\%) & 12517$\pm$7735 & 287.76$\pm$101.44 & 11438$\pm$4146 & \textit{48.41$\pm$17.54} & \textit{6214$\pm$2503} & 516.15$\pm$245.79 & 7823$\pm$4470 & 170.73$\pm$74.84   (.39\%) & 12432$\pm$5611 \\
\texttt{kClus\_5\_10\_5\_*} & 3289.68$\pm$664.31   (66.12\%) & 23297$\pm$5042 & 886.60$\pm$517.13 & 76632$\pm$48317 & 2681.71$\pm$754.70   (5.88\%) & 48680$\pm$13897 & 2771.39$\pm$852.71   (28.37\%) & 42357$\pm$14044 & \textit{386.01$\pm$266.65} & 32660$\pm$24156 & 2911.30$\pm$734.97   (176.34\%) & \textit{17516$\pm$6177} & 1237.42$\pm$768.45   (.14\%) & 60538$\pm$36374 \\
\texttt{kClus\_5\_20\_2\_*} & 3600.17$\pm$.06   (511.63\%) & 19732$\pm$3677 & \textit{1782.80$\pm$1019.48   (20.57\%)} & 447233$\pm$275642 & 3600+   (656.11\%) & 37293$\pm$2459 & 3600+   (176.55\%) & 120790$\pm$23599 & 2089.76$\pm$984.05   (8.21\%) & 154490$\pm$78837 & 3600+   (749.47\%) & \textit{16490$\pm$3378} & 3582.12$\pm$54.81   (694.41\%) & 118488$\pm$4948 \\
\texttt{kClus\_5\_20\_3\_*} & 3600+ ($\infty$) & 7425$\pm$2841 & 3600+   (222.05\%) & 169688$\pm$72270 & 3600+   (850.78\%) & 15355$\pm$1968 & 3600+   (729.79\%) & 39462$\pm$6099 & \textit{3600+   (73.28\%)} & 207274$\pm$16444 & 3600+   (947.78\%) & \textit{3976$\pm$1893} & 3600+   (683.11\%) & 88012$\pm$2561 \\
\texttt{kClus\_5\_20\_5\_*} & 3600+ ($\infty$) & 2200$\pm$936 & 3600+   (881.16\%) & \textit{41878$\pm$19519} & 3600+ ($\infty$) & 4885$\pm$939 & 3600+ ($\infty$) & 12659$\pm$2086 & \textit{3600+   (303.19\%)} & 121861$\pm$5742 & 3600+ ($\infty$) & 654$\pm$271 & 3600+   (759.28\%) & 54139$\pm$1713 \\
\texttt{socp\_random\_10\_10\_10\_*} & 2365.50$\pm$889.74   (4.66\%) & 1531$\pm$1406 & 454.42$\pm$291.88 & 7948$\pm$6799 & 1967.52$\pm$1065.66 & 1643$\pm$1919 & 2344.96$\pm$1082.43   (4.06\%) & 2271$\pm$1004 & \textit{54.33$\pm$43.63} & 1097$\pm$855 & 3469.74$\pm$391.88   (100.28\%) & \textit{\textbf{356$\pm$146}} & 185.97$\pm$164.91 & 2776$\pm$3025 \\
\texttt{socp\_random\_10\_10\_5\_*} & 231.09$\pm$126.89 & \textit{\textbf{191$\pm$198}} & 27.22$\pm$17.03 & 735$\pm$355 & 131.87$\pm$73.17 & 539$\pm$312 & 201.58$\pm$147.33 & 1916$\pm$1260 & \textit{10.07$\pm$4.74} & 429$\pm$250 & 1161.88$\pm$360.83 & 242$\pm$209 & 27.41$\pm$11.94 & 516$\pm$330 \\
\texttt{socp\_random\_10\_5\_10\_*} & 1021.48$\pm$733.14 & 1038$\pm$1325 & 50.15$\pm$24.51 & 927$\pm$505 & 92.10$\pm$29.22 & \textbf{390$\pm$172} & 434.27$\pm$315.97 & 1037$\pm$837 & \textit{9.83$\pm$3.85} & 497$\pm$215 & 2502.07$\pm$754.56   (2.82\%) & 996$\pm$603 & 20.08$\pm$6.33 & 440$\pm$195 \\
\texttt{socp\_random\_10\_5\_5\_*} & 42.83$\pm$36.83 & 44$\pm$88 & 3.92$\pm$2.44 & 166$\pm$97 & 25.92$\pm$20.49 & 218$\pm$315 & 14.20$\pm$6.28 & 272$\pm$136 & \textit{1.96$\pm$.73} & 150$\pm$86 & 95.10$\pm$70.92 & \textit{\textbf{28$\pm$49}} & 8.22$\pm$6.27 & 258$\pm$370 \\
\texttt{socp\_random\_2\_2\_2\_*} & 0.05$\pm$0.02 & 1$\pm$0 & \textit{\textbf{0.01$\pm$0.00}} & \textit{\textbf{0$\pm$0}} & 0.03$\pm$0.00 & 1$\pm$0 & 0.04$\pm$0.02 & 0$\pm$1 & 0.04$\pm$.01 & 0$\pm$1 & 0.05$\pm$0.02 & 1$\pm$1 & 0.03$\pm$0.00 & 4$\pm$2 \\
\texttt{socp\_random\_5\_10\_10\_*} & 248.35$\pm$59.30 & \textit{\textbf{146$\pm$216}} & 43.59$\pm$18.64 & 944$\pm$580 & 92.75$\pm$76.83 & 199$\pm$291 & 238.98$\pm$181.32 & 549$\pm$585 & \textit{4.33$\pm$2.99} & 161$\pm$145 & 1011.75$\pm$306.25 & 192$\pm$160 & 16.19$\pm$7.57 & 205$\pm$307 \\
\texttt{socp\_random\_5\_10\_5\_*} & 36.89$\pm$22.51 & \textit{\textbf{29$\pm$31}} & 3.38$\pm$2.06 & 106$\pm$78 & 20.91$\pm$9.01 & 78$\pm$77 & 13.09$\pm$6.68 & 266$\pm$194 & \textit{\textbf{1.31$\pm$.55}} & 80$\pm$64 & 150.07$\pm$112.83 & 30$\pm$28 & 5.74$\pm$1.99 & 73$\pm$58 \\
\texttt{socp\_random\_5\_5\_10\_*} & 40.86$\pm$16.23 & \textit{\textbf{21$\pm$34}} & 6.17$\pm$5.28 & 126$\pm$119 & 9.25$\pm$3.27 & 47$\pm$26 & 24.11$\pm$10.54 & 97$\pm$51 & \textit{0.98$\pm$0.32} & 55$\pm$31 & 143.50$\pm$109.49 & 39$\pm$31 & 3.37$\pm$.87 & 55$\pm$37 \\
\texttt{socp\_random\_5\_5\_5\_*} & 8.38$\pm$3.39 & \textit{\textbf{5$\pm$6}} & 1.83$\pm$1.53 & 42$\pm$31 & 4.55$\pm$1.11 & 41$\pm$28 & 2.72$\pm$.73 & 59$\pm$25 & \textit{\textbf{0.38$\pm$0.12}} & 30$\pm$15 & 19.59$\pm$12.41 & 7$\pm$6 & 1.91$\pm$.22 & 40$\pm$27 \\ \bottomrule 
\multicolumn{15}{l}{\footnotesize{Results corresponding to $k$-mean clustering \texttt{kClus} and random instances \texttt{socp\_random} are the average and standard deviation over 10 randomly generated instances.  Values with * denote solver failure, and values with + denote the termination for reaching time or iteration limits. A gap of $\infty$ denotes that no feasible solution was found.}}
\end{tabular}

%% file: tables/exp_solvers_bigm.tex
\rowcolors{1}{}{lightgray}
\begin{tabular}{@{}lccccccccccccc@{}}
\toprule
 & \multicolumn{4}{c}{\textbf{Big-M}} & \multicolumn{8}{c}{\textbf{Big-M-Cone}} &  \\  \cmidrule(lr){2-5}  \cmidrule(lr){6-13} 
 & \multicolumn{2}{c}{BARON} & \multicolumn{2}{c}{KNITRO} & \multicolumn{2}{c}{BARON} & \multicolumn{2}{c}{KNITRO} & \multicolumn{2}{c}{MOSEK-OA} & \multicolumn{2}{c}{MOSEK-IP} &  \\ \cmidrule(lr){2-3} \cmidrule(lr){4-5} \cmidrule(lr){6-7} \cmidrule(lr){8-9} \cmidrule(lr){10-11} \cmidrule(lr){12-13}  
Instance & Time [s] (Remaining gap \%) & Nodes & Time [s] (Remaining gap \%) & Nodes & Time [s] (Remaining gap \%) & Nodes & Time [s] (Remaining gap \%) & Nodes & Time [s] (Remaining gap \%) & Nodes & Time [s] (Remaining gap \%) & Nodes &  \\ \midrule
\texttt{LogReg\_10\_20\_1\_*} & \textit{\textbf{.28$\pm$.05}} & \textit{\textbf{1$\pm$0}} & 436.44$\pm$41.76 & 45384$\pm$2261 & 471.02$\pm$76.56 & 13699$\pm$3061 & 720.82$\pm$33.83 & 46478$\pm$2163 & 113.76$\pm$20.31 & 21941$\pm$1742 & 244.55$\pm$20.07 & 23759$\pm$1513 &  \\
\texttt{LogReg\_10\_20\_2\_*} & \textit{\textbf{.42$\pm$.53}} & \textit{\textbf{2$\pm$2}} & 441.82$\pm$43.23 & 45418$\pm$2507 & 453.59$\pm$47.54 & 13884$\pm$2365 & 719.01$\pm$44.70 & 45368$\pm$2353 & 111.68$\pm$18.63 & 22454$\pm$1443 & 238.62$\pm$17.22 & 24278$\pm$1820 &  \\
\texttt{LogReg\_2\_20\_1\_*} & \textit{14.79$\pm$6.55} & \textit{\textbf{71$\pm$26}} & 612.61$\pm$419.51 & 73711$\pm$45473 & 22.54$\pm$10.84 & 69$\pm$28 & 710.94$\pm$453.25 & 73709$\pm$45475 & 157.14$\pm$100.41 & 80143$\pm$48546 & 358.59$\pm$228.81 & 75566$\pm$44999 &  \\
\texttt{LogReg\_2\_20\_2\_*} & \textit{9.91$\pm$3.74} & \textit{\textbf{51$\pm$21}} & 453.38$\pm$270.08 & 53247$\pm$29832 & 37.34$\pm$19.98 & 180$\pm$228 & 690.44$\pm$421.71 & 53239$\pm$29840 & 120.22$\pm$79.25 & 59030$\pm$34633 & 272.45$\pm$158.24 & 57411$\pm$31825 &  \\
\texttt{LogReg\_5\_20\_1\_*} & \textit{26.61$\pm$8.56} & \textit{\textbf{187$\pm$81}} & 39.69$\pm$11.01 & 6891$\pm$1668 & 12.55$\pm$3.23 & 64$\pm$26 & 98.81$\pm$20.07 & 7111$\pm$1694 & 33.50$\pm$7.50 & 8834$\pm$1534 & 62.72$\pm$7.09 & 7864$\pm$1066 &  \\
\texttt{LogReg\_5\_20\_2\_*} & 24.90$\pm$8.71 & \textit{\textbf{172$\pm$75}} & 36.92$\pm$10.21 & 6673$\pm$1893 & 12.47$\pm$3.71 & 68$\pm$38 & 98.16$\pm$23.75 & 6586$\pm$1842 & 34.56$\pm$12.30 & 9171$\pm$2629 & 57.55$\pm$12.03 & 7134$\pm$1746 &  \\
\texttt{RSyn0805} & \textit{0.83} & 35 & 104.75 & 14427 & 0.92 & \textit{33} & 94.49 & 12229 & 0.61 & 665 & 5.17 & 1645 &  \\
\texttt{RSyn0805M02} & \textit{8.73} & \textit{59} & 2278.77 & 87511 & 17.22 & 157 & 1995.42 & 75301 & 37.81 & 19339 & 178.07 & 17687 &  \\
\texttt{RSyn0805M03} & \textit{12.63} & \textit{51} & 3600+ (46\%) & 86418 & 15.29 & 57 & 3600+ (52\%) & 81914 & 108.12 & 40451 & 182.13 & 12213 &  \\
\texttt{RSyn0805M04} & \textit{14.08} & \textit{37} & 3600+ (64\%) & 59738 & 17.05 & 41 & 3600+ (66\%) & 51530 & 80.36 & 19417 & 630.2 & 27821 &  \\
\texttt{RSyn0810} & \textit{0.16} & \textit{\textbf{1}} & 299.27 & 34021 & \textit{0.16} & \textit{1} & 439.74 & 45863 & 0.42 & 307 & 15.52 & 3969 &  \\
\texttt{RSyn0810M02} & \textit{10.61} & \textit{57} & 3600+ (314\%) & 98341 & 42.75 & 419 & 3600+ (337\%) & 88431 & 263.01 & 103759 & 2075.97 & 177545 &  \\
\texttt{RSyn0810M03} & \textit{21.75} & \textit{55} & 3600+ (142\%) & 69420 & 139.25 & 698 & 3600+ (134\%) & 60807 & 2129.23 & 540975 & 3345.5 & 176379 &  \\
\texttt{RSyn0810M04} & \textit{20.18} & \textit{39} & 3600+ (120\%) & 39870 & 71 & 225 & 3600+ (121\%) & 33496 & 255.91 & 51883 & 3600+ (16\%) & 121298 &  \\
\texttt{RSyn0815} & \textit{0.78} & \textit{\textbf{15}} & 327.21 & 35507 & 1.31 & 17 & 430.84 & 42155 & 0.92 & 699 & 56.1 & 13117 &  \\
\texttt{RSyn0815M02} & \textit{8.64} & \textit{49} & 3600+ ($\infty$) & 99192 & 94.94 & 1418 & 3600+ (744\%) & 86741 & 77.64 & 26931 & 3600+ (10\%) & 304799 &  \\
\texttt{RSyn0815M03} & \textit{16.85} & \textit{53} & 3600+ (87\%) & 64196 & 32.47 & 93 & 3600+ (91\%) & 55528 & 254.47 & 44941 & 3600+ (16\%) & 199725 &  \\
\texttt{RSyn0815M04} & \textit{30.96} & \textit{49} & 3600+ ($\infty$) & 32893 & 724.02 & 1816 & 3600+ ($\infty$) & 30276 & 3600+ (5\%) & 574357 & 3600+ (45\%) & 135070 &  \\
\texttt{RSyn0820} & \textit{1.58} & \textit{41} & 3008.63 & 189045 & 3.42 & 73 & 3217.97 & 194571 & 5.96 & 4275 & 214.49 & 39649 &  \\
\texttt{RSyn0820M02} & \textit{14.27} & \textit{67} & 3600+ ($\infty$) & 79117 & 24.56 & 229 & 3600+ ($\infty$) & 64384 & 458.06 & 150715 & 3600+ (79\%) & 249620 &  \\
\texttt{RSyn0820M03} & \textit{27.41} & \textit{83} & 3600+ (609\%) & 50557 & 93.13 & 317 & 3600+ (613\%) & 42199 & 3601.3 (12\%) & 861705 & 3600+ (62\%) & 171127 &  \\
\texttt{RSyn0820M04} & \textit{42.33} & \textit{75} & 3600+ ($\infty$) & 30632 & 160.87 & 294 & 3600+ ($\infty$) & 27670 & 3600+ (29\%) & 471300 & 3600+ (121\%) & 114152 &  \\
\texttt{RSyn0830} & \textit{1.58} & \textit{45} & 3600+ (646\%) & 179714 & 4.9 & 113 & 3600+ (660\%) & 165011 & 2.01 & 1133 & 289.17 & 38547 &  \\
\texttt{RSyn0830M02} & \textit{8.6} & \textit{\textbf{33}} & 3600+ ($\infty$) & 66766 & 37.42 & 177 & 3600+ ($\infty$) & 54687 & 53.55 & 11553 & 3600+ (153\%) & 192187 &  \\
\texttt{RSyn0830M03} & \textit{27.19} & \textit{\textbf{61}} & 3600+ ($\infty$) & 33975 & 187.37 & 597 & 3600+ (1\%) & 30226 & 2610.51 & 262461 & 3600+ (132\%) & 115968 &  \\
\texttt{RSyn0830M04} & \textit{67.81} & \textit{117} & 3600+ ($\infty$) & 20713 & 466.32 & 1001 & 3600+ ($\infty$) & 19862 & 3600+ (6\%) & 292737 & 3600+ (143\%) & 75392 &  \\
\texttt{RSyn0840} & \textit{1.88} & \textit{39} & 3600+ ($\infty$) & 165386 & 4.3 & 75 & 3600+ ($\infty$) & 150704 & 0.87 & 367 & 1332.52 & 143695 &  \\
\texttt{RSyn0840M02} & \textit{12.23} & \textit{\textbf{45}} & 3600+ ($\infty$) & 53553 & 235.49 & 1209 & 3600+ ($\infty$) & 43142 & 259.71 & 38513 & 3600+ (190\%) & 122054 &  \\
\texttt{RSyn0840M03} & \textit{22.25} & \textit{43} & 3600+ (334\%) & 25769 & 210.43 & 738 & 3600+ (337\%) & 21384 & 762.7 & 65449 & 3600+ (55\%) & 97644 &  \\
\texttt{RSyn0840M04} & \textit{63.98} & \textit{131} & 3600+ (812\%) & 16402 & 3600+ (4.08\%) & 6134 & 3600+ ($\infty$) & 14389 & 3600+ (22\%) & 225035 & 3600+ (199\%) & 41952
 & \\
\texttt{Syn05} & \textbf{0.04} & \textit{\textbf{1}} & 0.03 & 9 & 0.04 & 3 & 0.04 & 9 & 0.07 & 5 & 0.06 & 5 &  \\
\texttt{Syn05M02} & \textit{\textbf{0.08}} & \textit{\textbf{1}} & 0.31 & 55 & 0.13 & \textit{1} & 0.33 & 47 & 0.13 & 5 & 0.12 & 7 &  \\
\texttt{Syn05M03} & 0.12 & \textit{\textbf{1}} & 0.45 & 85 & \textit{0.11} & \textit{1} & 0.47 & 79 & 0.19 & 5 & 0.15 & 9 &  \\
\texttt{Syn05M04} & \textit{0.09} & \textit{\textbf{1}} & 0.9 & 105 & 0.14 & \textit{1} & 0.89 & 75 & 0.24 & 7 & 0.25 & 11 &  \\
\texttt{Syn10} & \textit{0.04} & \textit{\textbf{1}} & 0.15 & 63 & \textit{0.04} & \textit{1} & 0.15 & 63 & 0.09 & 7 & 0.12 & 17 &  \\
\texttt{Syn10M02} & \textit{0.11} & \textit{\textbf{1}} & 6.39 & 1217 & 0.61 & 17 & 5.55 & 819 & 0.35 & 97 & 0.66 & 155 &  \\
\texttt{Syn10M03} & \textit{0.17} & \textit{\textbf{1}} & 18.91 & 2353 & 1.02 & 23 & 24.21 & 2497 & 0.47 & 125 & 1.62 & 281 &  \\
\texttt{Syn10M04} & \textit{0.24} & \textit{\textbf{1}} & 74.95 & 5727 & 2 & 35 & 96.06 & 5573 & 1.16 & 541 & 4.12 & 607 &  \\
\texttt{Syn15} & \textit{0.05} & \textit{\textbf{1}} & 0.32 & 101 & 0.22 & 9 & 0.35 & 101 & 0.19 & 17 & 0.32 & 79 &  \\
\texttt{Syn15M02} & \textit{0.16} & \textit{\textbf{1}} & 13 & 1563 & 0.85 & 11 & 14.93 & 1319 & 0.31 & 47 & 2.28 & 399 &  \\
\texttt{Syn15M03} & \textit{0.22} & \textit{\textbf{1}} & 46.28 & 3219 & 1.81 & 17 & 58.14 & 3225 & 1.27 & 405 & 10.29 & 1543 &  \\
\texttt{Syn15M04} & \textit{0.36} & \textit{\textbf{1}} & 294.81 & 13857 & 3.46 & 23 & 374.43 & 13659 & 1.66 & 419 & 38.91 & 4445 &  \\
\texttt{Syn20} & \textit{0.1} & \textit{\textbf{1}} & 2.05 & 727 & 0.34 & 13 & 2.61 & 727 & 0.23 & 89 & 1.75 & 443 &  \\
\texttt{Syn20M02} & \textit{0.27} & \textit{\textbf{1}} & 641.76 & 59013 & 1.15 & 11 & 1434.96 & 104925 & 0.47 & 143 & 19.64 & 3285 &  \\
\texttt{Syn20M03} & \textit{0.49} & \textit{\textbf{1}} & 3600+ (2\%) & 164828 & 5.67 & 43 & 3600+ (6\%) & 137094 & 3.32 & 963 & 361.29 & 44127 &  \\
\texttt{Syn20M04} & \textit{0.87} & \textit{\textbf{1}} & 3600+ (195\%) & 114242 & 12.45 & 85 & 3600+ (207\%) & 95068 & 20.36 & 6527 & - & - &  \\
\texttt{Syn30} & \textit{0.31} & \textit{\textbf{5}} & 155.74 & 33457 & 2.38 & 61 & 206.97 & 33831 & 0.27 & 19 & 10.57 & 2189 &  \\
\texttt{Syn30M02} & \textit{1.44} & \textit{7} & 3600+ (559\%) & 159527 & 10.35 & 129 & 3600+ ($\infty$) & 146231 & 1.5 & 299 & 567.27 & 59585 &  \\
\texttt{Syn30M03} & \textit{3.07} & \textit{7} & 3600+ (514\%) & 108728 & 20.39 & 163 & 3600+ (507\%) & 92989 & 6.03 & 841 & - & - &  \\
\texttt{Syn30M04} & \textit{5.67} & \textit{5} & 3600+ (585\%) & 68862 & 65.62 & 456 & 3600+ (577\%) & 45120 & 25.68 & 2803 & - & - &  \\
\texttt{Syn40} & \textit{0.34} & \textit{5} & 1412.67 & 154355 & 5.38 & 147 & 2078.01 & 176433 & 0.4 & 73 & 109.23 & 18233 &  \\
\texttt{Syn40M02} & \textit{1.63} & \textit{5} & 3600+ (248\%) & 128057 & 40.68 & 481 & 3600+ (267\%) & 109363 & 2.58 & 603 & - & - &  \\
\texttt{Syn40M03} & \textit{5.3} & \textit{11} & 3600+ ($\infty$) & 74752 & 244.77 & 1815 & 3600+ ($\infty$) & 58161 & 54.77 & 6411 & - & - &  \\
\texttt{Syn40M04} & \textit{14.24} & \textit{39} & 3600+ ($\infty$) & - & 493.84 & 2686 & - & - & - & - & - & - &  \\
\texttt{exp\_random\_10\_10\_10\_*} & \textit{\textbf{.66$\pm$.39}} & \textit{\textbf{1$\pm$0}} & 29.62$\pm$20.08 & 4553$\pm$3144 & 1.15$\pm$.25 & \textit{\textbf{1$\pm$0}} & 95.01$\pm$51.82 & 7648$\pm$4298 & 8.49$\pm$2.68 & 5924$\pm$2505 & 71.84$\pm$19.71 & 3103$\pm$713 &  \\
\texttt{exp\_random\_10\_10\_5\_*} & .69$\pm$.24 & \textit{\textbf{1$\pm$0}} & 14.87$\pm$7.20 & 2792$\pm$1485 & \textit{\textbf{.68$\pm$.24}} & \textit{\textbf{1$\pm$0}} & 43.22$\pm$18.51 & 3155$\pm$1460 & 2.74$\pm$1.54 & 1684$\pm$1321 & 50.14$\pm$34.17 & 2062$\pm$1341 &  \\
\texttt{exp\_random\_10\_5\_10\_*} & .67$\pm$.18 & \textit{\textbf{1$\pm$0}} & 5.33$\pm$1.80 & 1048$\pm$381 & \textit{\textbf{.54$\pm$.16}} & \textit{\textbf{1$\pm$0}} & 8.90$\pm$3.11 & 1020$\pm$424 & 1.21$\pm$.37 & 1084$\pm$371 & 8.69$\pm$4.54 & 721$\pm$444 &  \\
\texttt{exp\_random\_10\_5\_5\_*} & \textit{.25$\pm$.10} & \textit{\textbf{1$\pm$0}} & 4.14$\pm$1.32 & 1146$\pm$400 & \textbf{.32$\pm$.04} & \textit{\textbf{1$\pm$0}} & 8.38$\pm$2.75 & 1124$\pm$411 & 1.00$\pm$.30 & 691$\pm$322 & 5.04$\pm$2.74 & 423$\pm$188 &  \\
\texttt{exp\_random\_2\_2\_2\_*} & \textit{.03$\pm$.02} & \textit{\textbf{1$\pm$1}} & .03$\pm$.01 & 5$\pm$1 & \textbf{.04$\pm$.02} & \textit{\textbf{1$\pm$0}} & .03$\pm$.00 & 5$\pm$1 & .05$\pm$.01 & 0$\pm$0 & .07$\pm$.02 & 0$\pm$0 &  \\
\texttt{exp\_random\_5\_10\_10\_*} & .66$\pm$.20 & \textit{\textbf{1$\pm$0}} & 3.27$\pm$1.10 & 594$\pm$201 & \textit{\textbf{.65$\pm$.15}} & \textit{\textbf{1$\pm$0}} & 5.46$\pm$1.37 & 651$\pm$187 & 1.01$\pm$.25 & 651$\pm$340 & 6.12$\pm$1.79 & 501$\pm$142 &  \\
\texttt{exp\_random\_5\_10\_5\_*} & \textit{.23$\pm$.12} & \textit{\textbf{1$\pm$0}} & 1.62$\pm$.42 & 362$\pm$134 & \textbf{.28$\pm$.09} & \textit{\textbf{1$\pm$0}} & 3.32$\pm$1.04 & 413$\pm$159 & .62$\pm$.15 & 358$\pm$153 & 3.84$\pm$3.22 & 332$\pm$261 &  \\
\texttt{exp\_random\_5\_5\_10\_*} & .27$\pm$.02 & \textit{\textbf{1$\pm$0}} & .81$\pm$.24 & 196$\pm$80 & \textit{\textbf{.23$\pm$.05}} & \textit{\textbf{1$\pm$0}} & 1.27$\pm$.36 & 204$\pm$70 & .40$\pm$.06 & 244$\pm$153 & .78$\pm$.23 & 157$\pm$67 &  \\
\texttt{exp\_random\_5\_5\_5\_*} & .15$\pm$.03 & \textit{\textbf{1$\pm$0}} & .67$\pm$.32 & 171$\pm$140 & \textit{\textbf{.16$\pm$.03}} & \textit{\textbf{1$\pm$0}} & 1.17$\pm$.48 & 146$\pm$63 & .29$\pm$.07 & 102$\pm$45 & .47$\pm$.26 & 86$\pm$49 &  \\
\texttt{proc\_100} & 3600+ (627.33\%) & \textit{976} & 3165.79 (637\%) & 200000+ & 3600+ (484.9\%) & 229898 & 3600+ (741\%) & 161504 & 2682.55 (292\%) & 702218 & 3600+ (664\%) & 230134 &  \\
\texttt{proc\_100b} & 3600+ (579.56) & \textit{952} & 2756.28 (711\%) & 200000+ & 3600+ (530.6\%) & 249969 & 3600+ (853\%) & 176762 & 2225.72 (254\%) & 632605 & 3600+ (486\%) & 230521 &  \\
\texttt{proc\_21} & 1.27 & \textit{49} & 4.6 & 921 & 1.21 & 55 & 5.57 & 863 & 0.91 & 491 & 2.57 & 447 &  \\
\texttt{proc\_21b} & \textit{0.83} & \textit{25} & 43.95 & 11575 & 2.2 & 123 & 64.47 & 11833 & 1.19 & 1355 & 7.21 & 1629 &  \\
\texttt{proc\_31} & 0.6 & \textit{9} & 3.59 & 595 & 1.22 & \textit{9} & 5.02 & 555 & 0.52 & 139 & 3.09 & 273 &  \\
\texttt{proc\_31b} & \textit{1.3} & 31 & 1206.42 (1\%) & 200000+ & 1.52 & \textit{19} & 1837.92 (10\%) & 200000+ & 5.51 & 4439 & 56.14 & 10329 &  \\
\texttt{proc\_36} & 1 & \textit{11} & 5.4 & 779 & 1.45 & 9 & 8.06 & 719 & 0.87 & 373 & 4.02 & 549 &  \\
\texttt{proc\_36b} & 4.5 & 67 & 1741.1 (42\%) & 200000+ & \textit{3.32} & \textit{31} & 2446.1 (37\%) & 200000+ & 10.96 & 14387 & 206.17 & 39771 &  \\
\texttt{proc\_48} & \textit{1.3} & 13 & 45.59 & 5999 & 1.85 & \textit{13} & 74.83 & 5899 & 1.95 & 885 & 13.22 & 1331 &  \\
\texttt{proc\_48b} & \textit{7.79} & 37 & 2070.89 (81\%) & 200000+ & 8.24 & \textit{25} & 2627.41 (78\%) & 200000+ & 48.14 & 43739 & 994.5 & 145623 & 
\end{tabular}%

%% file: tables/exp_solvers_hr.tex
\rowcolors{1}{}{lightgray}
\begin{tabular}{@{}llllllllllllll@{}}
\toprule
 & \multicolumn{4}{c}{\textbf{HR-$\varepsilon$}} & \multicolumn{8}{c}{\textbf{HR-Cone}} &  \\  \cmidrule(lr){2-5}  \cmidrule(lr){6-13} 
 & \multicolumn{2}{c}{BARON} & \multicolumn{2}{c}{KNITRO} & \multicolumn{2}{c}{BARON} & \multicolumn{2}{c}{KNITRO} & \multicolumn{2}{c}{MOSEK-OA} & \multicolumn{2}{c}{MOSEK-IP} &  \\ \cmidrule(lr){2-3} \cmidrule(lr){4-5} \cmidrule(lr){6-7} \cmidrule(lr){8-9} \cmidrule(lr){10-11} \cmidrule(lr){12-13}  
Instance & Time [s] (Remaining gap \%) & Nodes & Time [s] (Remaining gap \%) & Nodes & Time [s] (Remaining gap \%) & Nodes & Time [s] (Remaining gap \%) & Nodes & Time [s] (Remaining gap \%) & Nodes & Time [s] (Remaining gap \%) & Nodes &  \\ \midrule
\texttt{LogReg\_10\_20\_1\_*} & \textit{.29$\pm$.07} & \textit{\textbf{1$\pm$0}} & 451.32$\pm$41.09 & 45384$\pm$2261 & 3600+ & 291951$\pm$11436 & - & - & .74$\pm$.26 & 96$\pm$46 & 1.01$\pm$.32 & 72$\pm$45 \\
\texttt{LogReg\_10\_20\_2\_*} & \textit{.52$\pm$.76} & \textit{\textbf{2$\pm$2}} & 448.81$\pm$44.10 (.00\%) & 45418$\pm$2507 & 3600+ & 299651$\pm$7302 & - & - & 1.26$\pm$.61 & 142$\pm$103 & .86$\pm$.39 & 69$\pm$42 \\
\texttt{LogReg\_2\_20\_1\_*} & 20.31$\pm$8.78 & \textit{\textbf{71$\pm$26}} & 622.69$\pm$467.89 & 73711$\pm$45473 & 3600+ & 354150$\pm$27159 & - & - & \textit{\textbf{1.27$\pm$.33}} & 231$\pm$36 & 1.97$\pm$.41 & 162$\pm$40 \\
\texttt{LogReg\_2\_20\_2\_*} & 13.00$\pm$5.07 & \textit{\textbf{51$\pm$21}} & 462.20$\pm$288.83 & 53247$\pm$29832 & 3600+ & 345965$\pm$24189 & - & - & \textit{\textbf{1.16$\pm$.35}} & 168$\pm$68 & 2.20$\pm$.48 & 178$\pm$45 \\
\texttt{LogReg\_5\_20\_1\_*} & 32.11$\pm$11.15 & \textit{\textbf{187$\pm$81}} & 40.33$\pm$11.30 (.00\%) & 6891$\pm$1668 & 3600+ & 337068$\pm$18346 & - & - & \textit{\textbf{3.18$\pm$1.10}} & 929$\pm$256 & 9.88$\pm$2.15 & 878$\pm$190 \\
\texttt{LogReg\_5\_20\_2\_*} & 29.64$\pm$9.88 & \textit{\textbf{172$\pm$75}} & 38.80$\pm$10.72 (.00\%) & 6673$\pm$1893 & 3600+ & 330540$\pm$15143 & - & - & \textit{\textbf{3.48$\pm$1.36}} & 944$\pm$244 & 9.71$\pm$2.69 & 879$\pm$245 \\
\texttt{RSyn0805} & 0.51 & 3 & 1.62 & 117 & 2182.86 (1\%) & 1000000+ & 1.45 & 137 & \textit{\textbf{0.12}} & \textit{\textbf{0}} & 0.3 & 15 \\
\texttt{RSyn0805M02} & 2.84 & \textit{\textbf{5}} & 39.57 & 81 & 3600+ ($\infty$) & 368888 & 251.6 ($\infty$) & 100000 & \textit{\textbf{0.55}} & 19 & 2.08 & 93 \\
\texttt{RSyn0805M03} & 5.45 & \textit{\textbf{5}} & 17.91 & 103 & 3600+ (1\%) & 441998 & 654.83 ($\infty$) & 100000 & \textit{\textbf{0.72}} & 19 & 1.28 & 23 \\
\texttt{RSyn0805M04} & 8.28 & \textit{\textbf{5}} & 5.34 & 41 & 3600+ (1\%) & 275166 & 845.01 ($\infty$) & 100000 & \textit{\textbf{0.9}} & 7 & 1.53 & 11 \\
\texttt{RSyn0810} & 0.14 & \textit{\textbf{1}} & 1.76 & 111 & 3400.57 ($\infty$) & 1000000+ & 241.11 (5\%) & 199766 & \textit{\textbf{0.36}} & 7 & 0.44 & 9 \\
\texttt{RSyn0810M02} & 14.93 & 81 & 31.83 & 175 & 3600+ ($\infty$) & 266400 & 360.38 ($\infty$) & 100000 & \textit{\textbf{0.57}} & \textit{\textbf{17}} & 1.92 & 87 \\
\texttt{RSyn0810M03} & 40.54 & 121 & 42.97 & 489 & 3600+ ($\infty$) & 161105 & 598.03 ($\infty$) & 100000 & \textit{\textbf{1.67}} & \textit{89} & 8.2 & 299 \\
\texttt{RSyn0810M04} & 56.86 & 101 & 22.93 & 71 & 3600+ ($\infty$) & 57581 & 912.28 ($\infty$) & 100000 & \textit{\textbf{1.19}} & \textit{\textbf{9}} & 10.18 & 249 \\
\texttt{RSyn0815} & 6.08 & 89 & 4.31 & 139 & 3600 ($\infty$) & 601067 & 158.09 ($\infty$) & 199989 & \textit{\textbf{0.23}} & \textit{27} & 0.46 & 23 \\
\texttt{RSyn0815M02} & 16.33 & 55 & 13.14 & 67 & 3600+ ($\infty$) & 231668 & 343.19 ($\infty$) & 100000 & \textit{\textbf{0.66}} & \textit{\textbf{29}} & 4.86 & 217 \\
\texttt{RSyn0815M03} & 69.43 & 141 & 46.48 & 469 & 3600+ ($\infty$) & 112176 & 756.75 ($\infty$) & 100000 & \textit{\textbf{1.39}} & \textit{\textbf{99}} & 12.41 & 359 \\
\texttt{RSyn0815M04} & 49.77 & 37 & 39.35 & 171 & 3600+ ($\infty$) & 56238 & 1117.5 ($\infty$) & 100000 & \textit{\textbf{1.24}} & \textit{\textbf{5}} & 2.52 & 23 \\
\texttt{RSyn0820} & 12.21 & 159 & 2.44 & 127 & 3514.55 (8\%) & 1000000+ & 167.94 ($\infty$) & 199989 & \textit{\textbf{0.31}} & \textit{\textbf{17}} & 0.64 & 39 \\
\texttt{RSyn0820M02} & 27.95 & 77 & 98.68 & 299 & 3600+ ($\infty$) & 192510 & 368.17 ($\infty$) & 100000 & \textit{\textbf{0.72}} & \textit{\textbf{51}} & 4.16 & 151 \\
\texttt{RSyn0820M03} & 81.48 & \textit{117} & 249.24 & 2171 & 3600+ ($\infty$) & 76882 & 746.46 ($\infty$) & 100000 & \textit{\textbf{1.88}} & 119 & 16.58 & 455 \\
\texttt{RSyn0820M04} & 157.7 & 131 & 72.75 & 395 & 3600+ ($\infty$) & 30904 & 1363.63 ($\infty$) & 100000 & \textit{\textbf{1.67}} & \textit{\textbf{45}} & 11.49 & 155 \\
\texttt{RSyn0830} & 9.57 & 105 & 4.27 & 162 & 3600+ ($\infty$) & 504387 & 203.8 ($\infty$) & 200000+ & \textit{\textbf{0.38}} & 49 & 0.85 & \textit{\textbf{45}} \\
\texttt{RSyn0830M02} & 117.24 & 311 & 21.39 & 213 & 3600+ ($\infty$) & 125094 & 3600+ ($\infty$) & 2002 & \textit{\textbf{0.96}} & \textit{51} & 5.64 & 183 \\
\texttt{RSyn0830M03} & 249.05 & 235 & 85.91 & 445 & 3600+ ($\infty$) & 62648 & 3600+ ($\infty$) & 1498 & \textit{\textbf{2.22}} & \textit{179} & 10.38 & 203 \\
\texttt{RSyn0830M04} & 3600+ (1\%) & 1027 & 259.27 & 1265 & 3600+ ($\infty$) & 39280 & 2298.94 ($\infty$) & 100000 & \textit{\textbf{2.95}} & \textit{119} & 55.61 & 769 \\
\texttt{RSyn0840} & 4.92 & 29 & 1.41 & 27 & 3600+ ($\infty$) & 420832 & 234.28 ($\infty$) & 200000+ & \textit{\textbf{0.33}} & \textbf{17} & 0.67 & \textit{15} \\
\texttt{RSyn0840M02} & 221 & 223 & 12.52 & 103 & 3600+ ($\infty$) & 93859 & 1279.09 ($\infty$) & 100000 & \textit{\textbf{0.98}} & \textit{\textbf{37}} & 4.02 & 95 \\
\texttt{RSyn0840M03} & 2034.36 & 454 & 72.92 & 323 & 3600+ ($\infty$) & 40607 & 1398.34 ($\infty$) & 100000 & \textit{1.65} & \textit{41} & 10.26 & 149 \\
\texttt{RSyn0840M04} & 3600+ (2\%) & 413 & 657.44 & 2339 & 3600+ ($\infty$) & 27908 & 2710.11 ($\infty$) & 100000 & \textit{3.95} & \textit{153} & 73.35 & 823 \\
\texttt{Syn05} & 0.06 & \textit{\textbf{1}} & 0.03 & 3 & 819.94 ($\infty$) & 1000000+ & 0.1 & 3 & 0.04 & 0 & 0.05 & 5 \\
\texttt{Syn05M02} & 0.12 & \textit{\textbf{1}} & 0.05 & 3 & 1164.47 (87\%) & 1000000+ & 0.04 & 3 & 0.12 & 3 & 0.13 & 3 \\
\texttt{Syn05M03} & 0.28 & \textit{\textbf{1}} & 0.06 & 3 & 3047.17 ($\infty$) & 1000000+ & 0.05 & 3 & 0.14 & 3 & 0.1 & 3 \\
\texttt{Syn05M04} & 0.15 & \textit{\textbf{1}} & 0.08 & 3 & 3600+ ($\infty$) & 785490 & 0.06 & 3 & 0.19 & 3 & 0.12 & 3 \\
\texttt{Syn10} & 0.05 & \textit{\textbf{1}} & 0.03 & 3 & 1321.8 ($\infty$) & 1000000+ & 0.04 & 3 & \textit{0.01} & \textit{0} & \textit{0.01} & \textit{0} \\
\texttt{Syn10M02} & 0.44 & \textit{3} & 0.14 & 9 & 2103.28 (6\%) & 1000000+ & 0.08 & 5 & 0.22 & \textit{3} & 0.15 & \textit{3} \\
\texttt{Syn10M03} & 0.71 & \textbf{5} & 0.29 & 7 & 3357.55 (1\%) & 1000000+ & 0.15 & 7 & 0.33 & \textit{3} & \textit{0.16} & \textit{3} \\
\texttt{Syn10M04} & 2.1 & 15 & 0.27 & 7 & 3600+ (2\%) & 635620 & 0.22 & 7 & 0.35 & \textit{3} & \textit{0.26} & \textit{3} \\
\texttt{Syn15} & 0.2 & 5 & 0.05 & \textit{3} & 1900.43 ($\infty$) & 1000000+ & 4.85 (2\%) & 16387 & 0.08 & \textit{3} & 0.11 & \textit{3} \\
\texttt{Syn15M02} & 0.26 & \textit{\textbf{1}} & 0.14 & 5 & 3600+ (5\%) & 919619 & 0.09 & 3 & 0.15 & 3 & 0.16 & 3 \\
\texttt{Syn15M03} & 1.06 & 3 & 0.21 & 5 & 3600+ (7\%) & 500753 & 0.25 & 5 & 0.22 & \textit{3} & 0.24 & \textit{3} \\
\texttt{Syn15M04} & 1.9 & 5 & 0.35 & 6 & 3600+ (1\%) & 296378 & 0.31 & 3 & \textit{0.2} & \textit{0} & 0.39 & 3 \\
\texttt{Syn20} & 0.58 & 11 & 0.06 & 5 & 2537.75 (7\%) & 1000000+ & 0.08 & 5 & 0.18 & \textit{3} & 0.11 & \textit{3} \\
\texttt{Syn20M02} & 1.42 & 5 & 0.35 & 13 & 3600+ (3\%) & 975067 & 258.36 ($\infty$) & 199997 & 0.23 & \textit{3} & \textit{0.2} & 5 \\
\texttt{Syn20M03} & 7.26 & 21 & 1.02 & 27 & 3600+ ($\infty$) & 300067 & 932.34 ($\infty$) & 199997 & 0.31 & \textit{3} & \textit{0.44} & 13 \\
\texttt{Syn20M04} & 16.82 & 21 & 1.75 & 27 & 3600+ ($\infty$) & 208090 & 839.52 ($\infty$) & 199978 & 0.44 & \textit{5} & \textit{0.68} & 13 \\
\texttt{Syn30} & 2.1 & 15 & 0.13 (0\%) & 7 & 2735.21 ($\infty$) & 1000000+ & 0.14 (0\%) & \textit{7} & 0.18 & \textit{7} & 0.17 & \textit{7} \\
\texttt{Syn30M02} & 6.11 & 19 & 0.42 & 11 & 3600+ ($\infty$) & 292942 & 0.68 & 11 & \textit{0.36} & 11 & 0.38 & \textit{9} \\
\texttt{Syn30M03} & 65.88 & 49 & 1.56 & 23 & 3600+ ($\infty$) & 148036 & 605.33 ($\infty$) & 199998 & \textit{0.69} & \textit{21} & 0.82 & \textit{21} \\
\texttt{Syn30M04} & 85.14 & 47 & 3.56 & 37 & 3600+ ($\infty$) & 99419 & 4.43 & 37 & \textit{1.07} & 37 & 1.83 & \textit{27} \\
\texttt{Syn40} & 3.27 & 15 & 0.28 & 13 & 3600+ ($\infty$) & 827854 & 113.53 ($\infty$) & 200000+ & 0.23 & \textit{11} & \textit{0.22} & \textit{11} \\
\texttt{Syn40M02} & 32.11 & 71 & 1.39 & 23 & 3600+ ($\infty$) & 251756 & 512.42 ($\infty$) & 199998 & \textit{0.58} & 31 & 0.69 & \textit{15} \\
\texttt{Syn40M03} & 276.42 & 155 & 8.4 & 99 & 3600+ ($\infty$) & 115088 & 14.79 (0\%) & 39 & \textit{1.17} & 57 & 1.79 & \textit{39} \\
\texttt{Syn40M04} & 1108.16 & 317 & 64.61 & 210 & 3600+ ($\infty$) & 62600 & 1773.96 ($\infty$) & 200000+ & \textit{1.91} & 127 & 6.18 & \textit{115} \\
\texttt{exp\_random\_10\_10\_10\_*} & \textit{15.43$\pm$8.40} & \textit{1$\pm$0} & 46.25$\pm$36.45 & 398$\pm$357 & 3600+ & 75022$\pm$11349 & 848.66$\pm$49.23 & 200000+$\pm$0 & 27.98$\pm$12.71 & 264$\pm$183 & 22.48$\pm$10.59 & 319$\pm$180 \\
\texttt{exp\_random\_10\_10\_5\_*} & 12.92$\pm$4.21 & \textit{1$\pm$0} & 15.96$\pm$5.60 & 133$\pm$128 & 3600+ & 231038$\pm$20532 & 507.71$\pm$4.95 & 200000+$\pm$0 & \textit{8.39$\pm$6.46} & 457$\pm$750 & 7.07$\pm$3.20 & 155$\pm$89 \\
\texttt{exp\_random\_10\_5\_10\_*} & 7.32$\pm$4.84 & \textit{1$\pm$0} & 14.33$\pm$6.53 & 154$\pm$99 & 3600+ & 229524$\pm$16504 & 367.52$\pm$3.68 & 200000+$\pm$0 & \textit{5.37$\pm$1.82} & 107$\pm$70 & 3.80$\pm$1.49 & 105$\pm$57 \\
\texttt{exp\_random\_10\_5\_5\_*} & 2.30$\pm$1.34 & \textit{1$\pm$0} & 7.12$\pm$2.66 & 155$\pm$102 & 3600+ & 451525$\pm$33605 & 235.56$\pm$.85 & 200000+$\pm$0 & \textit{1.89$\pm$.60} & 73$\pm$51 & 3.16$\pm$1.32 & 139$\pm$76 \\
\texttt{exp\_random\_2\_2\_2\_*} & \textit{.03$\pm$.01} & 1$\pm$0 & .03$\pm$.00 & 1$\pm$0 & 1053.56$\pm$124.36 & 1000000+$\pm$0 & - & - & .04$\pm$.02 & \textit{0$\pm$0} & .04$\pm$.02 & 1$\pm$1 \\
\texttt{exp\_random\_5\_10\_10\_*} & 10.99$\pm$5.87 & \textit{3$\pm$5} & 7.17$\pm$3.33 & 35$\pm$23 & 3600+ & 323515$\pm$42747 & - & - & 5.60$\pm$2.24 & 44$\pm$40 & \textit{2.54$\pm$.99} & 44$\pm$28 \\
\texttt{exp\_random\_5\_10\_5\_*} & 2.11$\pm$1.51 & \textit{1$\pm$0} & 3.22$\pm$1.47 & 18$\pm$11 & 3600+ & 561311$\pm$57215 & - & - & 1.65$\pm$1.01 & 61$\pm$134 & \textit{1.35$\pm$.50} & 44$\pm$29 \\
\texttt{exp\_random\_5\_5\_10\_*} & 1.97$\pm$2.18 & \textit{1$\pm$0} & 2.46$\pm$1.12 & 30$\pm$24 & 3600+ & 664300$\pm$56737 & - & - & 1.72$\pm$.63 & 23$\pm$15 & \textit{.94$\pm$.24} & 33$\pm$22 \\
\texttt{exp\_random\_5\_5\_5\_*} & .72$\pm$.53 & \textit{1$\pm$0} & 1.10$\pm$.59 & 19$\pm$12 & 3511.73$\pm$114.18 & 972676$\pm$47392 & - & - & .61$\pm$.19 & 16$\pm$10 & \textit{.59$\pm$.13} & 28$\pm$16 \\
\texttt{proc\_100} & 129.79 & 287 & 1.25 & 21 & 3600+ (484.9\%) & 259774 & 474.53 ($\infty$) & 200000+ & 16.12 & 1831 & \textit{\textbf{1.47}} & \textit{\textbf{9}} \\
\texttt{proc\_100b} & 31.33 & \textit{7} & 5.46 & 87 & 3600+ (530.6\%) & 204643 & 472 ($\infty$) & 200000+ & 32.23 & 5199 & \textit{\textbf{4.97}} & \textbf{37} \\
\texttt{proc\_21} & 2.91 & \textit{6} & 1.15 & 79 & 3600+ ($\infty$) & 1000000+ & - & - & \textit{0.36} & 55 & \textbf{0.59} & \textbf{47} \\
\texttt{proc\_21b} & 2.16 & \textit{5} & 2.33 & 225 & 3600+ ($\infty$) & 1000000+ & 89.93 ($\infty$) & 200000+ & \textit{0.24} & 97 & \textbf{0.47} & \textbf{91} \\
\texttt{proc\_31} & 1.32 & \textit{1} & 0.18 (0\%) & 9 & 1832.16 ($\infty$) & 1000000+ & 150.63 ($\infty$) & 200000+ & 0.28 & 3 & \textit{\textbf{0.13}} & \textbf{3} \\
\texttt{proc\_31b} & 1.55 & \textit{1} & 0.31 (0\%) & 17 & 1424.67 ($\infty$) & 1000000+ & 142.84 ($\infty$) & 200000+ & 0.62 & 3 & \textit{\textbf{0.38}} & \textbf{3} \\
\texttt{proc\_36} & 3.61 & \textit{3} & 0.18 & 7 & 2123.32 ($\infty$) & 1000000+ & 166.94 ($\infty$) & 200000+ & 0.46 & 5 & \textit{\textbf{0.29}} & \textbf{5} \\
\texttt{proc\_36b} & 4.36 & \textit{1} & 0.6 & 27 & 1940.37 ($\infty$) & 1000000+ & 161.01 ($\infty$) & 200000+ & 0.67 & 19 & \textit{\textbf{0.22}} & \textbf{3} \\
\texttt{proc\_48} & 3.37 & \textit{1} & 0.28 & 11 & 2618.29 ($\infty$) & 1000000+ & 241.47 ($\infty$) & 200000+ & 0.99 & 7 & \textit{\textbf{0.35}} & \textbf{5} \\
\texttt{proc\_48b} & 5.08 & \textit{3} & 1.23 & 57 & 2444.62 ($\infty$) & 864362 & 223.77 ($\infty$) & 200000+ & 1.41 & 27 & \textit{\textbf{0.8}} & \textbf{29} \\
\bottomrule
\end{tabular}